\documentclass[12pt,a4paper]{amsart}
\usepackage{amsthm,amsfonts,amsmath,amssymb,latexsym}
\usepackage{epsfig,graphics,color}
\usepackage{mathtools}
\usepackage{soul}
\usepackage{etoolbox}
\usepackage{extpfeil}
\usepackage[matrix,arrow,curve]{xy}
\usepackage{todonotes}
\usepackage{hyperref}
\usepackage{mathrsfs}
\usepackage{verbatim}
\usepackage{bm}
\usepackage{cite}

\usepackage{mathtools}

\renewcommand{\t}[1]{\mathrm{#1}}

\newtheorem{PARA}{}[section]
\newtheorem{theorem}[PARA]{Theorem}
\newtheorem*{ctheorem}{Theorem}
\newtheorem{corollary}[PARA]{Corollary}
\newtheorem{lemma}[PARA]{Lemma}
\newtheorem{proposition}[PARA]{Proposition}
\newtheorem{definition}[PARA]{Definition}

\theoremstyle{definition}
\newtheorem{remark}[PARA]{Remark}
\theoremstyle{theorem}
\newtheorem{example}[PARA]{Example}
\newcommand{\para}{\begin{PARA}\rm}
\newcommand{\arap}{\end{PARA}\rm}
\newcommand{\dfn}{\begin{definition}\rm}
\newcommand{\nfd}{\end{definition}\rm}
\newcommand{\rmk}{\begin{remark}\rm}
\newcommand{\kmr}{\end{remark}\rm}
\newcommand{\xmpl}{\begin{example}\rm}
\newcommand{\lpmx}{\end{example}\rm}

\newcommand{\cC}{\mathcal{C}}
\newcommand{\cD}{\mathcal{D}}

\newcommand{\cM}{\mathcal{M}}

\newcommand{\cU}{\mathcal{U}}

\newcommand{\cV}{\mathcal{V}}

\newcommand{\fS}{\mathfrak{S}}

\newcommand{\fun}{\operatorname{Fun}}
\newcommand{\Top}{\operatorname{Top}}

\renewcommand{\top}{\operatorname{Top}}

\newcommand{\one}
{{{\mathchoice \mathrm{ 1\mskip-4mu l} \mathrm{ 1\mskip-4mu l}
\mathrm{ 1\mskip-4.5mu l} \mathrm{ 1\mskip-5mu l}}}}

\newcommand{\C}{{\mathcal{C}}}
\newcommand{\cc}{\mathbb{C}}
\newcommand{\D}{{\mathbb{D}}}

\newcommand{\M}{{\mathcal{M}}}
\newcommand{\N}{{\mathbb{N}}}
\newcommand{\Q}{{\mathbb{Q}}}
\newcommand{\R}{{\mathbb{R}}}

\newcommand{\Z}{{\mathbb{Z}}}
\renewcommand{\Im}{\mathrm{ Im\,}}       

\newcommand{\Aut}{\mathrm{ Aut}}          

\renewcommand{\hom}{\mathrm{Hom}}
\newcommand{\pt}{\operatorname{pt}}
\newcommand{\from}{\leftarrow}

\newcommand{\eps}{{\varepsilon}}

\def\NABLA#1{{\mathop{\nabla\kern-.5ex\lower1ex\hbox{$#1$}}}}
\def\Nabla#1{\nabla\kern-.5ex{}_{#1}}
\def\Tabla#1{\Tilde\nabla\kern-.5ex{}_{#1}}
\renewcommand{\Tilde}{\widetilde}

\newcommand{\p}{{\partial}}

\newcommand{\MM}{\mathbb{M}}

\newcommand{\X}{\mathcal{X}}
\newcommand{\Y}{\mathcal{Y}}
\newcommand{\E}{\mathcal{E}}

\newcommand{\funnel}{{\boldsymbol{\t{funnel}}}}

\begin{document}

\title[Deligne-Mumford operad and the circle]{The Deligne-Mumford operad as a trivialization of the circle action}
\author{Alexandru Oancea}
\address{Universit\'e de Strasbourg
\newline
Institut de Recherche Math\'ematique Avanc\'ee (IRMA)
\newline
Strasbourg, France}
\email{oancea@unistra.fr}
\author{Dmitry Vaintrob}
\address{Institut des Hautes \'Etudes Scientifiques (IHES) 
\newline 
Bures-sur-Yvette, France}
\email{mvaintrob@gmail.com}
\date{April 11, 2025}


\begin{abstract} We prove that the tree-like Deligne-Mumford operad is a homotopical model for the trivialization of the circle in the higher-genus framed little discs operad. Our proof is based on a geometric argument involving nodal annuli. We use as a model for the higher-genus framed little discs an operad of Riemann surfaces with analytically parametrized boundary. We develop the formalism of topological moduli problems as a framework to accommodate the orbifold nature of the Deligne-Mumford operad.  
\end{abstract}

\maketitle

{\tiny 
\tableofcontents
}


\section{Introduction}  \label{sec:intro}

An operad is a structure that contains spaces of operations with multiple inputs and one output, and rules for composing these operations. An algebra over an operad is a given incarnation of these operations and composition rules. Associative algebras, Lie algebras, commutative algebras, Gerstenhaber algebras, Batalin-Vilkovisky algebras, Hypercommutative algebras, these are all examples of algebras over suitably defined operads Ass, Lie, Com, Gerst, BV, HyperCom. 

Many such algebraic operads can be described as homologies of topological operads, i.e., operads with topological spaces of operations. Famously, Gerst is the homology of the operad of little 2-discs, whose spaces of operations consist of Euclidean embeddings of smaller discs, seen as inputs, into the unit disc, seen as the output, and where the composition is given by rescaling the unit disc and fitting it into some other small disc. Similarly, BV is the homology of the operad of framed little 2-discs, analogous to that of little 2-discs but involving the extra data of a marked point on the boundary of each of the discs under consideration. See Figure~\ref{fig:FLD}. Another example is HyperCom, which is the homology of the Deligne-Mumford-Knudsen operad, whose space of $k$-to-$1$ operations is the compactified moduli space $\overline{\cM}_{0,k+1}$ of genus 0 curves with $k+1$ marked points, of which one is labeled as an output and the other ones are labeled as inputs. The composition is represented by the nodal curve obtained by identifying an input of one curve with the output of another. See Figure~\ref{fig:DMK}.

\begin{figure} [ht]
\centering
\scalebox{.5}{\input{FLD.pstex_t}}
\caption{Composition in the operad of framed little $2$-discs.}
\label{fig:FLD}
\end{figure}

\begin{figure} [ht]
\centering
\scalebox{.5}{\input{DMK.pstex_t}}
\caption{Composition in the Deligne-Mumford-Knudsen operad.}
\label{fig:DMK}
\end{figure}

The following result was proved by Drummond-Cole~\cite{drummond2013homotopically}. Let 
$
FLD
$ 
be the operad of framed little disks and let 
$DMK$ 
be the genus zero Deligne-Mumford-Knudsen operad with $k$-to-$1$ operations indexed by points in $\overline{\cM}_{0,k+1}$.
Let 
$FLD_{1,1}$ 
be framed little disks with one input and one output (with only a space of $1$-to-$1$ operations, which is up to homotopy the group $S^1$), and let $\pt$ be the operad with only one identity $1$-to-$1$ operation. Then, 
in any model structure
on operads with weak equivalences spanned by maps of topological operads which are levelwise weak equivalences, we have the following result.
\begin{ctheorem}[Drummond-Cole~\cite{drummond2013homotopically}]
  The homotopy colimit of the diagram $$\pt\from FLD_{1,1}\to FLD$$ is related by a canonical sequence of weak equivalences to 
  $DMK$.
\end{ctheorem}

The intuition behind this theorem is the following: to trivialize the $S^1$-action in $FLD$ amounts to collapsing each small disc, as well as the boundary of the outer disc, to a point. The outcome is a genus 0 curve, i.e., an element of ${\cM}_{0,k+1}$, the uncompactified moduli space of genus $0$ curves with $k+1$ marked points. The homotopy pushout corresponds to the compactification $\overline{\cM}_{0,k+1}$: the latter involves nodal curves whose dual graph is a tree, and trees are precisely created by the operadic bar construction, which provides cofibrant replacements. 

In this paper we give a higher-genus generalization of this result.
Our method of proof gives a more geometric (and indeed motivic, as seen in~\cite{vaintrob_motivic}) interpretation of Drummond-Cole's theorem.

We define the operad 
$
\t{Fr}_\partial
$ 
of \emph{framed surfaces} with spaces of 
operations given by the moduli spaces of complex, i.e., conformal, surfaces with 
analytically parametrized boundary, and with composition given by gluing boundary components 
along compatible parametrizations. See~\S\ref{sec:framed_surfaces} and Figure~\ref{fig:split-surface}. 
The operad $FLD$ embeds in $\t{Fr}_\partial$ by viewing the boundary of the ``large'' disk in which the framed little disks embed as the outgoing boundary of a conformal surface of genus zero, and the boundaries of the interior disks as incoming boundaries. This embedding establishes a homotopy equivalence between $FLD$ and the suboperad $\t{Fr}_{\partial, g=0}$ of framed surfaces of genus zero, and so $\t{Fr}_\partial$ is a natural higher-genus generalization of $FLD$ in the category of topological operads. 

Let 
$
\mathrm{Ann}
$ 
be the suboperad of $\t{Fr}_\partial$ consisting of annuli, i.e., genus zero framed surfaces with one incoming and one outgoing boundary component. This operad is homotopy equivalent to $FLD_{1,1}$ and to $S^1$, and each annulus is a homotopy unit for $\t{Fr}_\partial$. It is convenient to enlarge $\mathrm{Ann}$ and $\t{Fr}_\partial$ to strictly unital operads 
$
\widetilde{\mathrm{Ann}}
$
and
$
\widetilde{\t{Fr}}_\partial
$ 
by including infinitely thin annuli. See~\S\ref{sec:framed_annuli}.

The main result of the present paper is the following theorem. As before, suppose we are working 
with a model structure with weak equivalences spanned by 
maps of topological operads which are levelwise weak equivalences. 
Let 
$
\t{DM}^{\t{tree}}
$ 
be the operad of ``tree-like'' nodal surfaces of arbitrary genus, whose spaces of operations are the partial compactifications of the moduli spaces $\cM_{*, *}=\{\cM_{g,k+1}\, : \, g\ge 0, \, k\ge 0\}$ of closed Riemann surfaces of genus $g$ with $k+1$ marked points by boundary components consisting of nodal curves whose dual graph is a tree. 
\begin{theorem}\label{thm:mainintro}
  The homotopy colimit of the diagram \begin{align}\label{diag:main_pushout}\pt\from \widetilde{\mathrm{Ann}}\to \widetilde{\t{Fr}}_\partial\end{align} of unital operads is related by a canonical sequence of weak equivalences to the Deligne-Mumford operad $\t{DM}^{\t{tree}}$. 
  
The same statement holds for the homotopy colimit of the diagram of nonunital operads $$\pt \from \mathrm{Ann} \to \t{Fr}_\partial.$$ 
\end{theorem}

The Deligne-Mumford operad is not an operad in topological spaces, but rather an operad in topological orbifolds, or stacks. We discuss the corresponding formalism in Appendix~\ref{sec:app_moduli}, where we call the relevant objects \emph{topological moduli problems}. 
If one is only interested in the operad as an object of a \emph{rational homotopy category} (e.g., by considering its chains over a field of characteristic zero), the orbifold structure can be ignored without changing the homotopy type and the homotopy colimit result holds on the level of coarse moduli spaces (note that in genus zero this distinction is irrelevant as there are no stabilizers). 

Let $\t{DM}_\text{coarse}^\text{tree}$ be the operad built out of the underlying coarse moduli spaces of the orbifold-valued operad $\t{DM}^\text{tree}$. Let $k$ be a field of characteristic zero. In this context, the operads of chains with coefficients in $k$ on $\t{DM}^{\t{tree}}$ and $\t{DM}_\text{coarse}^{\t{tree}}$ are equivalent, as the homology of finite groups is trivial in characteristic zero. The following corollary, motivated by mirror symmetry considerations which we discuss below, follows from Theorem~\ref{thm:mainintro} by the universal property of colimits. 
\begin{corollary} \label{cor:chains}
  Let $k$ be a field of characteristic $0$. 
  The data of an algebra over the operad of chains $C_*\big(\t{DM}^{\t{tree}}\big)$ is equivalent to the data of a dg algebra $A$ over $C_*(\widetilde{\t{Fr}}_\partial)$ together with a derived $S^1$-trivialization, i.e., a chain of quasi-isomorphisms of $C_*(S^1)$-modules $\tau:A\cong V,$ with $V$ a complex of $k$-modules carrying a trivial $S^1$-action. \qed
\end{corollary}

Corollary~\ref{cor:chains} is a higher-genus generalization of a result of Drummond-Cole and Vallette~\cite[Theorem~7.8]{drummond-cole-vallette}. The derived $S^1$-trivialization is equivalent to a Hodge-to-de Rham degeneration data in the sense of \cite{drummond-cole-vallette}. The characteristic zero condition can be removed at the cost of working with $\t{DM}^\text{tree}$ instead of $\t{DM}^\text{tree}_\text{coarse}$ and considering algebras over appropriate model-theoretic replacements of the operads involved.

{\bf Motivation, history of the problem and state of the art.} 
The main motivation for Theorem~\ref{thm:mainintro} and for Drummond-Cole's theorem comes from the homological mirror symmetry conjecture of Kontsevich~\cite{kontsevichICM94}. This conjecture postulates an equivalence between, on the symplectic side, the Fukaya category, and on the complex side, the category of coherent sheaves. In contrast, the original discovery and formulation of mirror symmetry was enumerative~\cite{Candelas_et_al,Cox-Katz}, and postulated an equivalence between Gromov-Witten invariants on the symplectic side and Hodge-type numbers on the complex side. Hence the question of describing Gromov-Witten invariants, or quantum cohomology, of a closed symplectic manifold, in terms of its Fukaya category.
In recent years, this has led to a flurry of activity around so-called ``categorical enumerative invariants''~\cite{costello2009,caldararu-costello-tu-2020,caldararu-tu-2020}.

Gromov-Witten invariants and the Fukaya category are related by the so-called \emph{closed-open map}. This map induces, for sufficiently nice symplectic manifolds, 
an isomorphism between symplectic cohomology, which is a variant of Floer homology, and Hochschild cohomology of the Fukaya category, see~\cite{kontsevichICM94,ganatra,abouzaid2010}. 
This isomorphism intertwines the $S^1$-action on symplectic cohomology with the $S^1$-action on Hochschild cohomology. Both these actions are part of naturally defined BV-algebra structures, which can also be refined at chain level as algebra structures over the operad of chains on the framed little discs $FLD$. On the Floer side this was proved recently by Abouzaid-Groman-Varolgunes~\cite{abouzaid-groman-varolgunes}, and on the Hochschild side this is closely related to the famous problem known as the ``Deligne conjecture''~\cite{tamarkin,voronov,kontsevich-soibelman-deligne,mcclure-smith,berger_fresse}. 

When the symplectic manifold is closed, we infer two different structures on its quantum cohomology. On the one hand, the fixed point map identifies it with symplectic cohomology, wherefrom a BV-algebra structure with trivial $S^1$-action. On the other hand, it classically carries the structure of a HyperCom algebra. (This  involves genus 0 Gromov-Witten invariants, e.g., the operation in arity 2 corresponds to the quantum multiplication.) To explain this phenomenon, Kontsevich formulated in 2003 the conjecture that the framed little discs operad with a trivialization of the circle should be equivalent to the $DMK$ operad.

This was proved in algebraic settings by Drummond-Cole and Vallette~\cite{drummond-cole-vallette}, as well as Khoroshkin, Markarian, Shadrin~\cite{khoroshkin_markarian_shadrin}, and in a topological setting by Drummond-Cole~\cite{drummond2013homotopically}. The statement at the topological level is the strongest, since it implies the previous ones by passing to chains. See also Dotsenko, Shadrin, Vallette~\cite{dotsenko_shadrin_vallette} for the relation between this picture and the Givental group action. 

Costello~\cite{costello2005} adopts a point of view on categorical enumerative invariants which is inspired by cohomological field theory, see also Kontsevich and Manin~\cite{kontsevich-manin}. From that perspective, it becomes relevant to study analogues of Kontsevich's conjecture in higher genus. Our Main Theorem~\ref{thm:mainintro} is the first result in that direction: we prove the conjecture of Kontsevich at the topological level, i.e., in the strongest sense, for the operadic part of the higher genus Deligne-Mumford moduli spaces (one output).

Very recently Tu~\cite{tu2021} proved a homological generalization of our Main Theorem~\ref{thm:mainintro} in the context of modular operads. This implies the equality between Costello's categorical enumerative invariants of the ground field and the Gromov-Witten invariants of a point, and is an important step towards inferring enumerative mirror symmetry from homological mirror symmetry~\cite{costello2009,caldararu-costello-tu-2020,caldararu-tu-2020}. Much more in line with our topological approach, Deshmukh~\cite{deshmukh2022} proved a generalization of our Main Theorem~\ref{thm:mainintro} 
from operads to input-output properads, i.e., properads that have no operations with $0$ inputs and $0$ outputs.
The starting object in~\cite{deshmukh2022} is the properadic version of our operad of framed nodal surfaces $\t{Fr}_\partial$, which we view as confirmation of the fact that this is the correct higher genus generalization of the little 2-discs operad. 
The paper~\cite{deshmukh2022} relies on 
the full machinery of $\infty$-categories developed by Lurie~\cite{lurie-higher-algebra}. 
In contrast, we keep technicalities to a minimum. 
The geometric perspective adopted in the current paper should make it appealing to  both topologists and geometers. 

Our paper brings into the picture a number of new ideas. We define the operad $\t{Fr}_\partial$ of framed surfaces as a higher genus analogue of the operad of framed little discs. Our Main Theorem~\ref{thm:mainintro} extends the equivalence of operads proved by Drummond-Cole~\cite{drummond2013homotopically} to higher genus, and Corollary~\ref{cor:chains} extends to higher genus the algebraic formulations from~\cite{khoroshkin_markarian_shadrin,dotsenko_shadrin_vallette}. Our proof is geometric and makes use of certain explicit and canonical degenerations of Riemann surfaces. Remarkably, our use of Riemann surfaces with analytically parametrized boundary, which makes the gluing operation well-defined, has a motivic counterpart discussed by the second author in~\cite{vaintrob_motivic}. 

\noindent {\bf Sketch of the proof.} The key technique in our proof consists in replacing the diagram $$\pt\from \mathrm{Ann}\to \t{Fr}_\partial$$ by the homotopy equivalent, but much more geometrically meaningful diagram (cf. Theorem~\ref{thm:geomthm})
\begin{align}\label{geom_pushout_diag}
\mathrm{NodAnn}\from \mathrm{Ann}\to \t{Fr}_\partial.
\end{align}
Here 
$\mathrm{NodAnn}$
is the operad of (stable) \emph{nodal annuli}, with only $1\to 1$ operations consisting of a compactification of $\mathrm{Ann}$ by allowing the modulus to tend to $\infty,$ which we geometrically interpret as the annulus developing a node (disjoint from either parametrized boundary component). The resulting operad turns out to be contractible (Lemma~\ref{lem:nodann-contractible}), hence gives rise to a diagram whose homotopy colimit is equivalent to the homotopy colimit $\pt\from \mathrm{Ann}\to \t{Fr}_\partial$ of the theorem above. In fact, in this formulation the homotopy colimit result is visible geometrically, as the Geometric Pushout Theorem~\ref{thm:geomthm} from~\S\ref{sec:dendritic}. In that statement we use a unital version $\widetilde{\mathrm{NodAnn}}$ of $\mathrm{NodAnn}$, which corresponds to the partial compactification in the ``modulus zero'' limit. (This partial compactification does not change the homotopy type.)

The proof of our Main Theorem~\ref{thm:mainintro} relies in fact on a mild homotopy enhancement of the proof of the Geometric Pushout Theorem~\ref{thm:geomthm}. 

Our proof of Proposition~\ref{prop:from-protected-to-unprotected}, stating that a certain operation of erasing seams on surfaces gives rise to a weak homotopy equivalence, has counterparts both in~\cite[\S7]{drummond2013homotopically} and in~\cite[\S\S5.2-5.3]{deshmukh2022}, in different setups. 
These different perspectives complement each other, and it is instructive to compare them, including from a complexity perspective.

{\bf Topological moduli problems and operads.} 
An extra layer of technical complexity is added to our work because the ``target'' of our comparison, the operad $\t{DM}^\t{tree}$, is not a topological operad. Instead, it is an operad valued in topological moduli problems, because certain stable marked complex curves have automorphisms that preserve the markings. The notion of a topological moduli problem, which we discuss in Appendix~\ref{sec:app_moduli}, is a variant of the notion of topological stack that is suitable for our purposes. We thus have to take care of two issues.
\begin{enumerate}
\item\label{Q:stack_to_top} How does one compare topological moduli problems to topological spaces?

\item\label{Q:stacky_op} What is an operad in topological moduli problems?

\end{enumerate}
We discuss the first question in Appendix~\ref{sec:app_moduli}, and the second question in Appendix~\ref{sec:dendroidal}. The outcome is that every topological moduli problem $\X$ has a ``classifying space'' and can thus be treated essentially as a topological space. 
The answer to the second question requires more formalism than the first, and the model we use is that of Segal operads, defined in \cite{cisinski-moerdijk-dendroidal-sets}. 
In particular, the Deligne-Mumford operad $\t{DM}^\t{tree}$ is a Segal operad rather than an ordinary operad. 
Luckily, the two questions can be treated separately for the purposes of this paper, and neither of them needs to interfere with the model category structure.

Before comparing $\t{DM}^\t{tree}$ with the homotopy pushout, we need to replace it by a model that allows boundary. To this end, we define in \S\ref{sec:dendritic} an operad $\t{NodFr}_\partial^\t{tree}$ of classifying spaces of nodal curves with parametrized boundary (again, an operad of topological moduli problems defined using the same framework as $\t{DM}^\t{tree}$). This operad contains both $\t{DM}^\t{tree}$ as a closed sub-operad and $\t{Fr}_\partial$ as an open suboperad, and moreover the map of operads $\t{DM}^\t{tree}\to \t{NodFr}_\partial^\t{tree}$ is a homotopy equivalence (after taking classifying spaces).

We use the homotopy equivalence $$\t{DM}^\t{tree}\xrightarrow{\funnel} \t{NodFr}_\partial^\t{tree}$$ 
(which we call the ``funnel'' map, as it is geometrically represented by attaching ``funnels", see Figure~\ref{fig:fr-genus0}) as one half of a ``roof'' of equivalences between $\t{DM}^\t{tree}$ and the homotopy pushout. Indeed, after taking canonical resolutions in an appropriate model category, we identify the pushout operad in~\eqref{geom_pushout_diag} with a certain space $\t{NodHD}^\t{tree}_{\t{protected}}$ of decorated curves which we call ``Humpty-Dumpty curves'', fitting into a diagram
\begin{align}\label{eq:main_sequence} \t{hocolim}(\text{diagram (\ref{geom_pushout_diag})})\cong \t{NodHD}^\t{tree}_{\t{protected}}\to \t{NodFr}_\partial^\t{tree}\from \t{DM}^\t{tree}
  \end{align}
all of whose maps are equivalences of operads in an appropriate model category.

\noindent {\bf Model category structures.}
We use the Berger-Moerdijk model category structure on topological operads throughout most of the paper, induced by a given model category structure on topological spaces. Note that the topological spaces we work with are not CW complexes. This makes it inconvenient to use the standard (Quillen) model category structure on the category of topological spaces, and we replace it by the so-called \emph{mixed model category structure} due to Cole~\cite{cole}. 
All spaces we work with are homotopy equivalent (and not just weak homotopy equivalent) to CW complexes, which implies that our results will also hold in the Str\o m model category structure \cite{strom}, where only homotopy equivalences are inverted. 
This is explained in~\S\ref{sec:model_category_operads}. 

{\bf Structure of paper.} Taking advantage of the analogous nature of the proofs of the homotopical Main Theorem~\ref{thm:mainintro} and of the Geometric Pushout Theorem~\ref{thm:geomthm}, we first give in~\S\ref{sec:operads} and~\S\ref{sec:operads_from_surfaces} a self-contained statement and proof of Theorem~\ref{thm:geomthm}, along with a brief 
introduction to topological operads and their pushouts. A reader interested in the flavor of our proof without the topological technicalities can read those sections only. In~\S\ref{sec:models}, \S\ref{sec:model_category_operads} we introduce the formalism of model categories and the Berger-Moerdijk model category structure on topological operads, which we will be working with. We prove the Main Theorem~\ref{thm:mainintro} in~\S\ref{sec:proof_main}. For the reader's convenience we have given separately the proofs for the genus $0$ case and for the higher genus case. In the genus $0$ case we only work with operads in topological spaces and there are not stacky phenomena involved, and we recover by a different, and perhaps more geometric method, the theorem of Drummond-Cole~\cite{drummond2013homotopically}. In the higher genus case we make full use of the language of topological moduli problems and Segal operads, which we introduce and discuss in Appendix~\ref{sec:app_moduli}, 
respectively Appendix~\ref{sec:dendroidal}.

{\bf Acknowledgements.} This paper got started in 2017 when both authors were members of the Institute for Advanced Study in Princeton within the special year on Mirror symmetry. We would both like to acknowledge the inspirational role played by the seminar on Hodge theory organized by P.~Seidel, where a talk by the first author initiated this collaboration. We have had fruitful discussions with M.~Abouzaid, V.~Dotsenko and E.~Getzler. The mixed model category structure was pointed out to us by A.~Lahtinen and S.~Schwede. The first author acknowledges financial support via the ANR grants ENUMGEOM ANR-18-CE40-0009 and COSY ANR-21-CE40-0002, as well as a Fellowship of the University of Strasbourg Institute for Advanced Study (USIAS). The second author acknowledges hospitality of the Institut de Math\'ematiques de Jussieu-Paris Rive Gauche (IMJ-PRG, Paris) and Institut de Recherche Math\'ematique Avanc\'ee (IRMA, Strasbourg). Finally, we would like to thank the anonymous referee and the Editors for their thoughtful suggestions and for their patience.

\section{Operads and topology}  \label{sec:operads}

\subsection{A brief reminder on operads}   \label{sec:brief_reminder_operads}
Operads were initially defined by May~\cite{May} in a topological context. 

An \emph{operad $O$} is a structure that specifies a class of composable operations with multiple inputs. An operad in {\sl sets} is a collection of sets $O_n$, $n\ge 0$ of operations ``with $n$ inputs and one output'', or operations ``of arity $n$'', together with {\sl composition rules} 
$$\gamma :O_k\times O_{n_1}\times O_{n_2}\times \dots\times O_{n_k}\to O_{n_1+\dots + n_k}$$
and \emph{permutation rules}, consisting of \emph{right} actions of the symmetric groups $\fS_n$ on $O_n$, $n\ge 0$, where $\fS_0=1$ by convention,
$$
O_n\times\fS_n\to O_n, \qquad (o,\sigma)\mapsto o\sigma.
$$
The composition rules and the permutation rules are required to satisfy certain tautological relations which essentially encode the fact that they behave like composition and permutation of inputs. Generally, operads are also required to have a \emph{unit}, $1\in O_1,$ with the property that composing $o\in O_k$, $k\ge 1$ by $1$ on the left or with the tuple $(1, 1, \dots, 1)$ on the right does not change $o$. By default, when we use the word ``operad'' we will mean \emph{unital operad}. 

A \emph{representation} of an operad $O$ (in sets), or an \emph{algebra over $O$}, is a set $(S,\rho)$ with a collection of maps $\rho_o:S^n\to S$ indexed by $o\in O_n$, $n\ge 0$. By convention $S^0$ consists of a single point and therefore we interpret the collection of maps $\rho_o$, $o\in O_0$ as a distinguished collection of elements in $S$. The case in which $O_0$ consists of a single element is historically important, see May~\cite{May}, but the operads that we will construct in this paper will have naturally richer spaces $O_0$. 
The collections of maps $\rho_o$, $o\in O_n$, $n\ge 1$ are interpreted as spaces of operations with $n$ inputs and one output in $S$. We require these maps to satisfy the permutation rule 
$$
\rho_{o\sigma}(s_1, \dots, s_n) =  \rho_o(s_{\sigma(1)}, \dots, s_{\sigma(n)})
$$ 
for $\sigma\in \fS_n$ a permutation, and also the associativity rule
$$
\rho_o\circ \big(\rho_{o_1}\times\dots\times \rho_{o_k}\big)=\rho_{\gamma(o, (o_1, \dots, o_k))} :S^{n_1}\times\dots\times S^{n_k}\to S.
$$

Note that the only property needed in order to define operads and algebras over operads in this context is that the category $\t{Set}$ has a symmetric monoidal structure with respect to the cartesian product and the permutation action $\fS_n\times S^n\to S^n$, $\sigma(s_1,\dots,s_n)=(s_{\sigma(1)},\dots,s_{\sigma(n)})$. In particular, we can define the notion of an operad and of an algebra over an operad in any symmetric monoidal category $(\mathcal{C}, \otimes)$ with choice of unit object. The category of operads in $\mathcal{C}$ is denoted $\t{Op}_{\mathcal{C}}$.   

For convenience, we shall impose a slightly stronger condition: namely, that the symmetric monoidal category $\cC$ we work with be \emph{closed} (see~\cite[\S2]{berger_moerdijk_model}), which in particular implies it has (small) colimits, the colimits distribute over pushouts and there is an internal Hom functor. 
The cases of most interest to us are the categories $\Top$ of 
compactly generated weakly Hausdorff topological spaces (this is the standard category used in homotopy theory~\cite[Definition~2.4.21, Theorem~2.4.25]{hovey}), 
${\text{Vect}}$ of vector spaces and ${\text{Vect}}_\t{dg}$ of differential graded vector spaces. Given a lax symmetric monoidal functor $\mathcal{C}\to \mathcal{D}$, we get a functor of associated operad categories $\t{Op}_{\mathcal{C}}\to \t{Op}_{\mathcal{D}}$. 

Operads can be equivalently defined by specifying a smaller set of composition rules, the so-called \emph{partial compositions}. More precisely, given a unital operad $O$ one defines the partial compositions 
$$
- \circ_i - : O_k\times O_\ell \to O_{k+\ell-1},\qquad 1\le i\le k
$$
as $u\circ_i v = \gamma(u;1,\dots,1,v,1,\dots,1)$ for $1\le i\le k$. These obey the tautological relations $u\circ_i (v\circ_j w) = (u\circ_i v)\circ _{i-1+j}w$ for $i,j\ge 1$ and $(u\circ_j w)\circ_i v = (u\circ_i v)\circ_{j-1+\ell}w$ for $j>i\ge 1$ and $v\in O_\ell$, called respectively \emph{sequential composition} and \emph{parallel composition}. These relations determine uniquely all the other composition and permutation rules for the operad $O$, allowing for an equivalent definition of the operad structure. 

\begin{remark}
When we work with the Deligne-Mumford operad and its variations, we will need the slightly more sophisticated theory of Segal operads, which is adapted to handle stacky objects, i.e., objects with self-symmetries. This refinement is explained in Appendix~\ref{sec:dendroidal}.
  \end{remark}

\subsection{Free operad} \label{sec:free_operad} This section is based on~\cite[\S5.8]{berger_moerdijk_model} and~\cite[\S3]{berger_moerdijk_bv}.

Many constructions in algebra canonically output graded objects, i.e., objects of the form $\sqcup X_i$ for $\sqcup$ the coproduct operation ($\oplus$ for vector spaces) and 
$i$ running over some indexing set $I$. 
For example the free unital monoid on a set $\Gamma$ (resp., a vector space $V$) is $\t{Free}(\Gamma) = \bigsqcup_{n\in \N} \Gamma^n$ (resp., the tensor algebra $\t{Free}(V) = \bigoplus_{n\in \N}V^{\otimes n}$), indexed by 
$\N$.  Note that the monoid structure on $\t{Free}(\Gamma)$ ``lives over'' the standard additive monoid structure on $\N$. 
It is a common procedure to resolve
a monoid, or an associative algebra, by free ones using a simplicial object (or chain complex) based on the free algebra construction of the bar complex. The analogue of the bar complex in the theory of operads is indexed not by the monoid of natural numbers but by the operad of trees, whose $n\to 1$ operations are given by certain trees with $n$ distinguished ``input edges'' and one distinguished ``output edge''. 
Some trees have automorphisms, which interact with the $\fS_n$-actions on spaces of operations, so properly speaking the free construction is indexed by 
a \emph{groupoid} of trees. 

Let $\mathcal{C}$ be any symmetric monoidal category. The \emph{symmetric groupoid} is the category $\fS$ with objects the finite sets $[n]=\{1,\dots,n\}$, $n\ge 0$ and morphisms the permutations of $[n]$. We define the category of \emph{$\fS$-collections} in $\mathcal{C}$ as $\t{Fun}(\fS^{op},\mathcal{C})$, also denoted $\fS-\t{Mod}_{\mathcal{C}}$. Explicitly, the objects of $\fS-\t{Mod}_{\mathcal{C}}$ are sequences $X_* : = (X_0,X_1, X_2, \dots)$ with $X_n \in \mathcal{C}$ a $\fS_n$-module for $n\ge 0$, and the morphisms are equivariant sequences of morphisms in $\mathcal{C}$.

Let $\t{Set}_f$ be the category with objects the finite sets and morphisms the bijections between finite sets. The categories $\fS$ and $\t{Set}_f$ are canonically equivalent. As a consequence, the category of $\fS$-collections $\t{Fun}({\fS}^{op},\mathcal{C})$ is canonically isomorphic to the category $\t{Fun}(\t{Set}_f^{op},\mathcal{C})$~\cite[Proposition~1.51]{MSS}. When viewing a $\fS$-collection $X_*$ as a functor $\t{Set}_f^{op}\to\mathcal{C}$, we denote $X_F$ the object in $\mathcal{C}$ that is associated to a finite set $F$.

We have a canonical forgetful functor 
$$
\t{forg}: \t{Op}_{\mathcal{C}}\to \fS-\t{Mod}_{\mathcal{C}}
$$ 
which associates to an operad $O$ the sequence $(O_0,O_1,O_2,\dots)$ of its spaces of operations. This functor has a left adjoint
$$
\t{Free} : \fS-\t{Mod}_{\mathcal{C}}\to \t{Op}_{\mathcal{C}}
$$ 
called the \emph{free operad functor}. The adjunction relation reads 
$$
\t{Hom}_{\t{Op}_{\mathcal{C}}}(\t{Free}(X_*), O) \cong \t{Hom}_{\fS-\t{Mod}_{\mathcal{C}}}(X_*,\t{forg}(O)).
$$

\subsubsection{The operad of labeled rooted trees} \label{sec:operad_of_labeled_rooted_trees} We first need to describe an important operad based on the following heuristic idea: an operation with $n$ inputs is represented by a rooted tree with $n$ distinguished leaves labeled by the set $\{1,\dots,n\}$, \emph{up to isomorphism}. Composition of operations is represented by grafting such trees one upon another. 

A {\sl graph with half-edges} is a graph $\Gamma$ with a set of vertices $\t{Vert}_\Gamma$, a set of oriented edges $\t{Edge}_\Gamma$ each having one tail and one head vertex, and an additional set of oriented half-edges $\t{Half}_\Gamma$ with only one end (either head or tail). We denote $\t{Half}^+_\Gamma$ the set of incoming half-edges. We denote $\bar \Gamma$ the oriented graph of full edges. 
We say that a 
graph with half-edges $\Gamma$ is a \emph{tree of operations} if 
$\bar\Gamma$ is a 
rooted tree and $\Gamma$ has exactly one outgoing half-edge which is attached to the root of $\bar \Gamma$. 
This definition allows for an arbitrary number (including $0$) of incoming half-edges for $\Gamma$, it allows for some (or all) of the leaves of $\bar\Gamma$ to have no incoming half-edge attached to them, and it allows for the incoming half-edges of $\Gamma$ to be attached at any vertex of $\bar\Gamma$. Each interior vertex of $\bar\Gamma$ has a unique outgoing edge attached to it.  See Figure~\ref{fig:tree}.
In addition to the above, we also introduce the \emph{trivial tree} $|$ consisting of a unique edge and no vertex. We do not consider our trees of operations as being endowed with a planar structure.  

\begin{figure} [ht]
\centering
\input{tree.pstex_t}
\caption{A tree of operations $\Gamma$ and its associated graph of full edges $\bar\Gamma$.}
\label{fig:tree}
\end{figure}

A \emph{labeling} of a tree of operations $\tau$ with $n\ge 1$ incoming half-edges is the data of a bijection 
$\lambda:\{1,\dots,n\} \stackrel\sim\longrightarrow \t{Half}^+_\tau$, which we view as assigning an element in $\{1,\dots,n\}$ to each incoming half-edge. A \emph{labeled tree} of operations is a pair $(\tau,\lambda)$ consisting of a tree of operations $\tau$ and a labeling $\lambda$.
Two labeled trees $(\tau,\lambda)$ and $(\tau',\lambda')$ are \emph{equivalent} if there exists an 
isomorphism $\phi:\tau\stackrel\sim\longrightarrow\tau'$ 
that intertwines the labelings, i.e., such that $\lambda'=\phi\lambda$. Write 
\[
\t{Tree}_n
\] 
for the set of all labeled trees of operations with $n\ge 0$ incoming half-edges, and write 
$$
Tree_n
$$ 
for the equivalence classes under the above equivalence relation. This is a $\fS_n$-equivariant groupoid with respect to the right action of $\fS_n$ on labelings.

The $\fS_n$-groupoids ${Tree}_n$, $n\ge 0$ form an operad in the following way:
\begin{itemize}
\item Given an equivalence class of a labeled tree $[\tau,\lambda] \in {Tree}_k$ and a collection $[\tau_i,\lambda_i]\in {Tree}_{n_i}$, $1\le i\le k$ we define the composition $\gamma([\tau,\lambda], ([\tau_1,\lambda_1], \dots, [\tau_k,\lambda_k]))$ as follows. We choose representatives for each of the previous equivalence classes, we build a tree $T$ by gluing for each $i\in\{1,\dots,k\}$ the outgoing half-edge of $\tau_i$ to the $i$-th incoming half-edge of $\tau$ as distinguished by the labeling $\lambda$, producing thus for each $i$ a new interior edge whose tail vertex is the root of $\tau_i$ and whose head vertex is the same as that of the $i$-th incoming edge of $\tau$. (If $\tau_i$ is the trivial tree, the gluing is innocuous.) We define a labeling $\ell$ of the tree $T$ by concatenating the labelings $\lambda_1,\dots,\lambda_k$. The result of the composition is the equivalence class of the labeled tree $(T,\ell)$.
\item The unit is provided by the trivial tree $|$ with its unique labeling. 
\end{itemize}
The resulting operad, denoted 
$
Tree,
$ 
is \emph{the operad of labeled rooted trees}. 

\begin{remark}
We refer to~\cite{kock2011} for a different description of trees. 
\end{remark}

\subsubsection{The free operad functor} \label{sec:free_operad_functor} Let $X_*$ be a $\fS$-collection, which we view as a functor $\t{Set}_f^{op}\to\mathcal{C}$. The heuristic idea for the construction of the free operad $\t{Free}(X_*)$ is the following: the space of operations in arity $n\ge 0$ consists of elements of $Tree_n$, decorated at each vertex $v$ by an element of 
$X_{\t{in}(v)}$, where $\t{in}(v)$ is the set of incoming edges and half-edges at $v$.
The composition of operations is inherited from the composition of trees. 

Given a tree of operations $\tau$, define
$$
X^\tau = \prod_{v\in \text{Vert}_\tau} X_{\t{in}(v)}.
$$ 
A \emph{labeled rooted tree with vertices colored by elements of $X_*$ and with $n\ge 0$ incoming half-edges} is a triple $(\tau,\mathbf{x},\lambda)$ with $(\tau,\lambda)\in\t{Tree}_n$ and $\mathbf{x}\in X^\tau$. Two such triples $(\tau,\mathbf{x},\lambda)$ and $(\tau',\mathbf{x'},\lambda')$ are \emph{equivalent} if
there exists an 
isomorphism $\phi:(\tau,\lambda)\stackrel\sim\longrightarrow(\tau',\lambda')$ such that 
the colors $\mathbf{x}=(x_v)_{v\in\t{Vert}_\tau}$ and $\mathbf{x'}=(x'_w)_{w\in\t{Vert}_{\tau'}}$ satisfy the condition $x'_{\phi(v)}=x_v\sigma_\phi$, 
where $\sigma_\phi:X_{\t{in}(v)}\to X_{\t{in}(\phi(v))}$ is the isomorphism determined by the bijection $\phi:\t{in}(v)\to\t{in}(\phi(v))$. Let 
$$
\t{Tree}_n(X_*)
$$ 
be the space of labeled rooted trees with vertices colored by elements of $X_*$ and with $n\ge 0$ incoming half-edges. This carries a natural topology and splits as a disjoint union of topological spaces indexed by the elements of $\t{Tree}_n$. Let 
$$
Tree_n(X_*)
$$ 
be the space of equivalence classes under the above equivalence relation, which again carries a natural topology and splits as a disjoint union of topological spaces indexed by the elements of $Tree_n$. This is naturally a $\fS_n$-space under the action of the permutation group on labelings. 

\begin{definition} The spaces of operations in the \emph{free operad} $\t{Free}(X_*)$ are 
$$
\t{Free}(X_*)_n = Tree_n(X_*),\qquad n\ge 0. 
$$
These form a topological operad with compositions, unit, and $\fS$-structure inherited from the operad $Tree$. 
\end{definition}

\subsection{Pushout of operads} \label{sec:pushout_operads}
Suppose that 
$$
P\leftarrow A\to Q
$$ 
is a diagram of topological operads. We define the \emph{amalgamated product}, or \emph{pushout} 
$$
P \sqcup_A Q
$$
to be the colimit of the diagram in topological operads. Explicitly, this is a quotient (interpreted as a colimit) of the free operad $\t{Free}(P_*\sqcup Q_*)$ and is defined as follows. Consider the counit of the free-forgetful adjunction: this is the natural transformation between the functors $\t{Free}\circ\t{forg}$ and $\mathrm{Id}_{\t{Op}}$ which associates to each operad $O$ the ``product'' morphism of operads $\prod:\t{Free}(O_*)\to O$ obtained by applying composition maps in $O$ to a tree of elements in $O$ recursively until the tree has a single vertex (this is independent of the order by the associativity of operations in operads). 
Now $P\sqcup_A Q$ is the quotient of $\t{Free}(P_*\sqcup Q_*)$ by the equivalence relation generated by the relations 
$$
\sim_1\sqcup \sim_2
$$
described as follows.
\begin{itemize}\label{poutreln}
\item ($\sim_1$) If $o_{free}\in \t{Free}(P_*\sqcup Q_*)$ is a free element over a tree $\tau$ and
  $\tau$ has a sub-tree $\tau_0$ all of whose vertices are labeled by
  elements of $P$ (resp., of $Q$) then $o_{free}$ is equivalent to
  $o_{free}'$ with all vertices and all full edges of $\tau_0$
  contracted to a point, and with the product $\prod(o_{free}|_{\tau_0})$
  written at that point.
\item ($\sim_2$) Denote the two operad maps $i:A\to P$ and $j:A\to Q$. If $o_{free}\in \t{Free}(P_*\sqcup Q_*)$ is a free element over a tree $\tau$ which on some vertex $v\in \tau$ has a label which is equal to $i(a)$ for some $a\in A$, we set $o_{free}\sim o'_{free}$ where $o'_{free}$ has the label on $v$ replaced by $j(a)$. 
\end{itemize}
The amalgamated product can be defined more generally for operads in categories which do not live over the category of sets. The  above relations should then be understood as coequalizer conditions in the underlying category.

\section{Operads based on Riemann surfaces with boundary}  \label{sec:operads_from_surfaces}
\subsection{The operad of framed surfaces}  \label{sec:framed_surfaces}

\begin{definition}
A \emph{framed surface} is a compact Riemann surface $\Sigma$ with boundary $\p \Sigma$ locally analytically modelled on the upper half plane $\{z\in\mathbb{C}\, :  \, \Im z\ge 0\}$, together with an analytic parametrization $\varphi_i:S^1\to C_i$ for each boundary component $C_i\subset \p \Sigma$. 

A component $C_i\subset \p \Sigma$ is called an \emph{input} or an \emph{output} if the orientation induced by the parametrization coincides, respectively is opposite to the boundary orientation of $C_i$.  

Write $\t{Fr}_\partial^{m,n}$ for the moduli space of framed surfaces with $m$ incoming and $n$ outgoing boundary components.
\end{definition}

The space of oriented analytic diffeomorphisms $S^1\to S^1$ which preserve a basepoint $1\in S^1$ is contractible. Indeed, this set is identified with the space of analytic functions $f:\R\to \R$ satisfying the conditions $f(0)=0$, $f(x+1)=f(x)+1$ for all $x\in\R$, and $f'>0$, which is convex. (The function $f(x)=x$ can be taken as a basepoint.) As such, once an orientation of each boundary component has been specified (which is the same as a labelling of the components as inputs or outputs), an analytic parametrization $\varphi_i:S^1\to C_i$ is determined up to homotopy by the choice of a basepoint $p_i=\varphi_i(1)\in C_i$. 

Framed surfaces can be glued at inputs and outputs because of the following phenomenon. 

A framed surface is canonically isomorphic in the neighborhood of each of its boundary components to a closed annulus 
$$
A_\epsilon = \{z\in\cc \, : \, 1-\epsilon \le |z| \le 1\}
$$ 
for some $\epsilon>0$. Indeed, given a component $C_i$ with an analytic parametrization $\varphi_i$, the latter locally extends uniquely, and these local extensions coincide on the overlaps by uniqueness of holomorphic continuation. The original parametrization $\varphi_i$ corresponds then to the restriction of the extended parametrization to the circle $\{|z|=1\}$ if $C_i$ is an output, respectively to the restriction to the circle $\{|z|=1-\epsilon\}$ if $C_i$ is an input. As a consequence, any two framed surfaces are uniquely locally isomorphic in the neighborhood of any of their incoming, respectively outgoing boundary components.

Given two annuli $A_{\epsilon}, A_{\epsilon'}$ (viewed as complex manifolds with canonically parametrized boundary) the incoming boundary of the first can be glued to the outgoing boundary of the second (to produce an annulus with modulus $\ln 1/(1-\epsilon) + \ln 1/(1-\epsilon')$, cf.~\S\ref{sec:framed_annuli} below). Since every framed surface is isomorphic in a neighborhood of each of its boundary components to such an annulus, this gives us the local data necessary for gluing two framed surfaces along boundary components of opposite orientation, $$(\Sigma, \gamma), (\Sigma, \gamma')\mapsto \Sigma\sharp_{\gamma, \gamma'}\Sigma'.$$ Note that this also makes sense if $\Sigma, \Sigma'$ are disconnected and also if $\gamma, \gamma'$ are boundary components consisting of multiple circles, as long as the orientations are compatible.

In particular, the moduli spaces $\t{Fr}_\partial^{m,n}$ form a topological PROP, and the moduli spaces $\t{Fr}_\partial^{m, 1}$ with one output form a topological operad. 
We denote this latter topological operad by 
$
\t{Fr}_\partial.
$
We call it \emph{the operad of framed surfaces}. 

Note that the moduli space $\t{Fr}_\partial^{m,n}$ is \emph{a priori} a stacky object, as a surface can have automorphisms. However, this can only happen when both $m$ and $n$ are equal to zero, as no nontrivial automorphism of a connected complex surface can fix an embedded curve or boundary component pointwise. Since we will only be interested in the operad $\t{Fr}_\partial$, which involves the moduli spaces $\t{Fr}_\partial^{m,n}$ with $n = 1$, we will never encounter any stacky phenomena involving framed surfaces. 

\begin{remark} \label{rmk:labelings} It is understood here that the elements of $\t{Fr}_\partial$ are \emph{labeled} framed Riemann surfaces, meaning that, for each framed Riemann surface $\Sigma\in\t{Fr}_\partial^{m,1}$, we are given a bijection $\lambda$ between $\{1,\dots,m\}$ and the set of incoming boundary components of $\Sigma$. The bijection $\lambda$ is called a \emph{labeling}, and there are of course $m!$ choices of labelings. The labeling is necessary in order to define composition by gluing and hence the operad structure on $\t{Fr}_\partial$. This additional presence of labelings is standard for operads constructed out of Riemann surfaces, similarly to the case of the Deligne-Mumford spaces $\overline{\mathcal{M}}_{g,n}$ where the $n$ marked points are also labeled. The symmetric group $\fS_m$ acts on the right on the set of labelings of a framed Riemann surface $\Sigma$ by composition at the source $(\lambda,\sigma)\mapsto \lambda\sigma$, $\sigma\in \fS_m$. For readability we will henceforth not mention explicitly the labelings of surfaces, but whenever we will write ``framed surface'' we will mean ``labeled framed surface''.  
\end{remark}

\begin{remark}\label{rmk:topology_on_fr}
  We will be interested in $\t{Fr}_\partial$ as a topological operad and 
  we now specify the topology on the moduli spaces involved. 
   Given any point of $\t{Fr}_\partial^{m,1}$ corresponding to a surface $S$, we can glue in disks (with standard parametrization of the boundary) to all the inputs and outputs of $S$ to obtain a closed Riemann surface $\bar{S}$. This gives an identification of $\t{Fr}_\partial^{m, 1}$ with the moduli space of Riemann surfaces with $m+1$ parametrized loops bounding disks isomorphic to the standard disk $D\subset \mathbb{C}$ and with standard boundary parametrization. In particular, $\t{Fr}_\partial$ is a subspace of the space of tuples $(X, \gamma_1,\dots, \gamma_{m+1})$ with $X$ a closed Riemann surface (corresponding to a point of some $\mathcal{M}_{g,m+1}$) and the $\gamma_i$, $i=1,\dots,m+1$ pairwise nonintersecting contractible embedded analytic loops in $X$. This is a bundle over $\mathcal{M}_{g,m+1}$. 
   We topologize $\t{Fr}_\partial^{m,1}$ as a locally closed subset of this bundle of tuples.
   
This presents $\t{Fr}_\partial$ as a complex infinite-dimensional manifold. Its local model at a framed Riemann surface of genus $g$ with $m+1$ boundary components is the total space of a fibration over a neighborhood of the corresponding element in $\mathcal{M}_{g,m+1}$ with fiber given by $m+1$-tuples of embeddings of the disc in $\mathbb{C}$ close to the standard one. The fact that the corresponding element in $\mathcal{M}_{g,m+1}$ may be an orbifold point is irrelevant here. 
   
\end{remark}

\subsection{The monoid of framed annuli} \label{sec:framed_annuli} The genus $0$ and arity $1$ part of $\t{Fr}_\partial$ forms a topological monoid which we denote 
$
\t{Ann}
$ 
and call \emph{the monoid of framed annuli}. 

A \emph{framed annulus} is a genus $0$ Riemann surface $A$ with two boundary components $\p A=\p^+A \sqcup \p^-A$ labeled as input and output, together with analytic parametrizations $f_+$ of the input $\p^+A$ and $f_-$ of the output $\p^-A$. Ignoring the parametrizations of the boundary components, such an annulus is conformally determined by its \emph{modulus} $\alpha\in (0,\infty)$ (Schottky's theorem, \cite{schottky}). This is the logarithm of the ratio of the radii 
$$
\alpha=\ln R/r
$$ of a \emph{standard annulus} $A_{R,r}=\{z\in\mathbb{C}\, : \, r\le |z|\le R\}$, $r<R$ which is conformally equivalent to $A$, where the outer circle $|z|=R$ is labeled as input and the inner circle $|z|=r$ is labeled as output. The group of conformal automorphisms of the underlying Riemann surface $A$ is canonically isomorphic to $S^1$: up to replacing $A$ with a conformally equivalent standard annulus, its group of automorphisms is represented by the rotations of $\mathbb{C}$ which fix the origin. As such, the pair $(f_-,f_+)$ is considered modulo global rotations $\theta\cdot(f_-,f_+)=(\theta + f_-,\theta+f_+)$, $\theta\in S^1$. With this understood, we write $[(A,f_-,f_+,\alpha)]$ for the equivalence class of a framed annulus $(A,f_-,f_+,\alpha)$. 

{\bf Remark.} The modulus behaves additively under gluing of standard annuli. However, it \emph{does not} behave additively under gluing of general framed annuli. This can be seen explicitly by studying configurations of nested circles in $\mathbb{C}$.

The topological monoid $\t{Ann}$ is not unital. In order to achieve unitality, it is convenient to enlarge it to the topological monoid of \emph{possibly degenerate framed annuli}, denoted 
$
\widetilde{\t{Ann}},
$ 
by including the \emph{moduli space of framed annuli of modulus $0$}, denoted 
$
{\t{Ann}}^0.
$ 

A \emph{framed annulus of modulus $0$} is a triple $(C,f_-,f_+)$ consisting of a connected closed analytic $1$-dimensional manifold $C$ together with analytic diffeomorphisms $f_\pm:S^1\to C$. We will also refer to $(C,f_-,f_+)$ as being a \emph{framed annulus of thickness zero}, or as being a \emph{degenerate framed annulus}. Two such framed annuli $(C,f_-,f_+)$ and $(D,g_-,g_+)$ are \emph{equivalent} if there exists an analytic diffeomorphism $\psi:C\to D$ such that $g_\pm=\psi f_\pm$. As such, the framed annulus $(C,f_-,f_+)$ is equivalent to $(S^1,\mathrm{id},f_-^{-1}f_+)$ and also to $(S^1,f_+^{-1}f_-,\mathrm{id})$. We choose the first expression to realize a bijection 
$$
{\t{Ann}}^0\stackrel\sim\longrightarrow \t{Aut}(S^1),\qquad [(C,f_-,f_+)]\mapsto f_-^{-1}f_+.
$$
The composition of the equivalence classes of two framed annuli of modulus $0$ is defined by
$$
[(C,f_-,f_+)]\circ [(D,g_-,g_+)]= [(C,f_-,f_+g_-^{-1}g_+)] = [(D,g_-f_+^{-1}f_-,g_+)].
$$
This makes ${\t{Ann}}^0$ into a group. The neutral element is the class $[(S^1,\mathrm{id},\mathrm{id})]$, consisting of degenerate annuli $(C,f_-,f_+)$ with $f_-=f_+$. The inverse of $[(C,f_-,f_+)]$ is $[(C,f_+,f_-)]$. As such the above bijection 
$$
{\t{Ann}}^0\stackrel\sim\longrightarrow \t{Aut}(S^1)
$$
is a group isomorphism. (Had we chosen to associate to the class of an annulus $[(C,f_-,f_+)]$ the element $f_+^{-1}f_-\in\mathrm{Aut}(S^1)$, suggested by choosing as a representative the degenerate annulus $(S^1,f_+^{-1}f_-,\mathrm{id})$, we would have obtained a bijective group anti-homomorphism.)

The topological monoid $\t{Ann}$ is a trivial fiber bundle over $(0,\infty)$, which is the space of moduli of unframed annuli, with fiber $\t{Aut}(S^1)\times_{S^1}\t{Aut}(S^1)$, where $\t{Aut}(S^1)$ stands for the group of analytic automorphisms of the circle and $S^1$ acts diagonally on $\t{Aut}(S^1)\times\t{Aut}(S^1)$ by translations in the target. We topologize $\widetilde{\t{Ann}}$ by extending this trivial fiber bundle to a trivial fiber bundle over $[0,\infty)$ and collapsing the fiber at $0$ via the diagonal action of $\t{Aut}(S^1)$ given by $\varphi\cdot(f_-,f_+)=(\varphi f_-,\varphi f_+)$. We identify the quotient with $\t{Aut}(S^1)$ via $(f_-,f_+)\mapsto f_-^{-1} f_+$ as above. 

We extend the monoid structure from $\t{Ann}$ to $\widetilde{\t{Ann}}$ as described above for two elements in ${\t{Ann}}^0$ and by defining 
$$
[(A,f_-,f_+,\alpha)]\circ [(C,g_-,g_+)]= [(A,f_-,f_+g_-^{-1}g_+,\alpha)],
$$
and 
$$
[(D,h_-,h_+)]\circ [(A,f_-,f_+,\alpha)]= [(A,f_-h_+^{-1}h_-,f_+,\alpha)]
$$ 
for $[(A,f_-,f_+,\alpha)]\in \t{Ann}$ and $[(C,g_-,g_+)], [(D,h_-,h_+)]\in {\t{Ann}}^0$. 

We claim that this monoid structure is compatible with the above topology, i.e., $\widetilde{\t{Ann}}$ is a topological monoid. To prove the claim, let us consider sequences $[(A^\nu,f_-^\nu,f_+^\nu,\alpha^\nu)]$ and $[(B^\nu,g_-^\nu,g_+^\nu,\beta^\nu)]$, $\nu\ge 1$ with $\alpha^\nu,\beta^\nu>1$, and such that, for $\nu\to\infty$, we have $\alpha^\nu\to\alpha$, $\beta^\nu\to\beta$ with $\alpha$ or $\beta$ equal to $1$. We can assume without loss of generality that $A^\nu$ and $B^\nu$ are standard annuli whose inner radius is equal to $1$ and whose outer radius is equal to $\alpha^\nu$, respectively $\beta^\nu$, and also that $f_\pm^\nu\to f_\pm$, $g_\pm^\nu\to g_\pm$, the limits being analytic parametrizations of the standard circles of corresponding radii $1$, $\alpha$ and $\beta$. 

We prove the claim in the case $\alpha>1$ and $\beta=1$. The glued annulus $A^\nu \# B^\nu$ has input given by the boundary component $\p^+B^\nu$ with parametrization $g_+^\nu$, and output given by the boundary component $\p^-A^\nu$ with parametrization $f_-^\nu$. See Figure~\ref{fig:annuli}. As $\nu\to\infty$, the input $\p^+B^\nu$ of $B^\nu$ --- which is the standard circle of radius $\beta^\nu$ in $\mathbb{C}$ --- converges pointwise with respect to the standard parametrization to the standard circle of radius $1$ with its standard parametrization, viewed as $\p^-B^\nu$ for all $\nu$. The latter is identified with $\p^+A^\nu$ via $f_+^\nu (g_-^\nu)^{-1}$. As such, the limit of the composition $A^\nu \# B^\nu$ is canonically identified with the limit $A$ of the sequence $A^\nu$, and this identification is given by $f_+g_-^{-1}$ along the input boundary component. The input boundary component of the limit inherits the parametrization $g_+$, and via this identification the latter corresponds to the parametrization $f_+g_-^{-1}g_+$ of the input boundary component of $A$. As far as the output boundary component of the limit is concerned, it is canonically identified with the output boundary component of $A$ and inherits as such the parametrization $f_-$. This shows that 
\begin{eqnarray*}
\lefteqn{\!\!\!\!\!\!\!\!\!\!\!\!\!\!\!\!\!\!\lim_{\nu\to\infty}[(A^\nu,f_-^\nu,f_+^\nu,\alpha^\nu)]\circ [(B^\nu,g_-^\nu,g_+^\nu,\beta^\nu)]}\\
& = & [(A,f_-,f_+g_-^{-1}g_+,\alpha)] \\
& = & [(A,f_-,f_+,\alpha)]\circ [(S^1,g_-,g_+)]\\
& = & \lim_{\nu\to\infty}[(A^\nu,f_-^\nu,f_+^\nu,\alpha^\nu)] \circ \lim_{\nu\to\infty}[(B^\nu,g_-^\nu,g_+^\nu,\beta^\nu)].
\end{eqnarray*}

The proof of the claim in the cases $\alpha=1$, $\beta >1$ and $\alpha=\beta=1$ is analogous and we omit it. 

\begin{figure} [ht]
\centering
\input{annuli.pstex_t}
\caption{We depict a (framed) annulus as a horizontal cylinder of finite length, with its input boundary component to the right and its output boundary component to the left. The composition $A\circ B$ of two framed annuli is depicted by drawing $A$ to the left of $B$.}
\label{fig:annuli}
\end{figure}

\begin{definition}
We define $\widetilde{\mathrm{Fr}}_\partial$ to be the extension of $\mathrm{Fr}_\partial$ by possibly degenerate framed annuli,
$$
\widetilde{\mathrm{Fr}}_\partial = \mathrm{Fr}_\partial \sqcup _{\mathrm{Ann}}\widetilde{\mathrm{Ann}}.
$$
\end{definition}

 \subsection{Framed nodal annuli} \label{sec:framed_nodal_annuli}
 Ordinary annuli have modulus parameter $\alpha\in (0, \infty).$ By introducing degenerate annuli, we have extended the possible parameters to $[0, \infty)$. In this section we will further extend the possible modulus parameters from $[0, \infty)$ to $[0, \infty].$ We do this by adding a new class of annuli, called \emph{nodal annuli}, which have modulus parameter $\infty$. While introducing degenerate moduli did not change the homotopy type of the topological monoid $\t{Ann}$, adding in nodal annuli has a strong destructive effect: it makes the monoid contractible.

 \begin{definition}
   We say that a complex surface with analytically para\-metrized boundary is a \emph{framed nodal annulus} if it has two  boundary components, genus zero, and at most nodal singularities. (In order to shorten notation, the term ``nodal annuli'' includes ordinary annuli with no nodes.)
 \end{definition}

 We say that a framed nodal annulus is \emph{unstable} if it has an irreducible component which contains no boundary components (equivalently, if it has a component of genus zero and infinite automorphism group), and \emph{stable} otherwise. See Figure~\ref{fig:unstable}. Note that all stable framed nodal annuli either have one irreducible component containing both boundary circles (i.e., they are ordinary framed annuli), or two irreducible components of which one contains the incoming circle and the other contains the outgoing circle. The {\sl stabilization} of an unstable nodal annulus is obtained by contracting all irreducible components which have no boundary. We will be interested in the moduli space of stable framed nodal annuli, viewed as quotients of possibly unstable framed nodal annuli by the equivalence relation induced by stabilization. We write 
 $$
 \t{NodAnn}
 $$ 
 for the moduli space of stable framed nodal annuli. We topologize this space similarly to our moduli space of surfaces with boundary above. Namely, given a stable framed nodal annulus, we get a point of $\overline{\mathcal{M}}_{0,4}$ by gluing in disks along both parametrized boundary components, and marking the images of $\pm 1\subset S^1$ in both boundary components in the resulting genus zero curve. In this way, we can view $\t{NodAnn}$ as a subspace in the bundle over $\overline{\mathcal{M}}_{0,4}$ whose fiber consists of pairs of parametrized disjoint embedded analytic closed curves whose parametrizations map $\pm 1\in S^1$ to the marked points. 
 
\begin{figure} [ht]
\centering
\input{unstable.pstex_t}
\caption{Unstable/stable framed nodal annuli.}
\label{fig:unstable}
\end{figure}
 
 For the next lemma, recall that we denote $\Aut(S^1)$ the group of analytic automorphisms of $S^1$ with analytic inverse, and $\mathrm{Aut}_0(S^1)\subset\mathrm{Aut}(S^1)$ denotes the subgroup of automorphisms which fix $1\in S^1$. As explained in~\S\ref{sec:framed_surfaces}, the group $\mathrm{Aut}_0(S^1)$ is contractible. 
  \begin{lemma} \label{lem:nodann-contractible}
   The moduli space of stable framed nodal annuli is homeomorphic to 
   $$\big( \mathrm{Aut}_0(S^1)\times \mathrm{Aut}_0(S^1) \big)\times \mathbb{C}.$$  
   In particular, it is contractible. 
 \end{lemma}
 
\begin{proof} Consider the action of $S^1$ on $\mathrm{Aut}(S^1)$ by translations in the target. The moduli space of stable framed nodal annuli containing a node is identified with $\mathrm{Aut}(S^1)/S^1\times \mathrm{Aut}(S^1)/S^1$. Indeed, each of the two irreducible components of the underlying Riemann surface is equivalent to a disk with a marked point at the origin. The group of automorphisms of the latter is $S^1$, given by rotations, and it acts on the analytic parametrizations of its boundary by translations in the target. Writing $f (\mathrm{mod}\, S^1)$ for the class of an element of $\mathrm{Aut}(S^1)$ modulo the action of $S^1$, an arbitrary element of this moduli space can thus be written $(f_- (\mathrm{mod}\, S^1), f_+ (\mathrm{mod}\, S^1))$. 

With this understood, the topology on $\t{NodAnn}$ can be alternatively described as follows. Let $[(A^\nu,f_\pm^\nu)]$, $\nu\ge 1$ be a sequence in $\t{Ann}$ with moduli $\alpha^\nu\to\infty$, $\nu\to\infty$. Choose representatives $A^\nu=[-\alpha^\nu/2,\alpha^\nu/2]\times S^1$ and $(f_-^\nu,f_+^\nu)\in \mathrm{Aut}(S^1)\times_{S^1}\mathrm{Aut}(S^1)$ and assume that $(f_-^\nu,f_+^\nu)\to (f_-,f_+)$ as $\nu\to\infty$. We then have by definition 
\[
[(A^\nu,f_\pm^\nu)]\to (f_- (\mathrm{mod}\, S^1),f_+ (\mathrm{mod}\, S^1)), \qquad \nu\to\infty.
\] 

By marking the point $1\in S^1$ we obtain homeomorphisms 
\[
\mathrm{Aut}(S^1)/S^1\simeq \mathrm{Aut}_0(S^1), \qquad \mathrm{Aut}(S^1)\simeq \mathrm{Aut}_0(S^1)\times S^1,
\] 
and 
\[
\mathrm{Aut}(S^1)\times_{S^1}\mathrm{Aut}(S^1)\simeq \mathrm{Aut}_0(S^1)\times\mathrm{Aut}_0(S^1)\times S^1.
\] 
(None of these identifications preserves any group structure, see also Remark~\ref{rmk:semi-direct-product} below.) 

We have already seen that the moduli space $\t{Ann}$ of framed annuli is a trivial bundle over $(0,\infty)$ with fiber $\mathrm{Aut}(S^1)\times_{S^1}\mathrm{Aut}(S^1)$, where $S^1$ acts diagonally. In view of the isomorphism $S^1\times (0,\infty)\simeq \mathbb{C}^\times$, after choosing a trivialization of the bundle $\t{Ann}\to (0,\infty)$ we obtain a homeomorphism 
\[
\t{Ann}\simeq \mathrm{Aut}_0(S^1)\times \mathrm{Aut}_0(S^1)\times \mathbb{C}^\times.
\]
With respect to this identification, the projection $\t{Ann}\to (0,\infty)$ corresponds to the projection $\mathbb{C}^\times\to(0,\infty)$, $z\mapsto |z|$. Also, with respect to the identification of the moduli space of stable framed nodal annuli containing a node with $\mathrm{Aut}_0(S^1)\times\mathrm{Aut}_0(S^1)$, the definition of convergence for a sequence $(f_-^\nu,f_+^\nu,z^\nu)\in \t{Ann}\simeq \mathrm{Aut}_0(S^1)\times \mathrm{Aut}_0(S^1)\times \mathbb{C}^\times$ such that $|z^\nu|\to\infty$ and $(f_-^\nu,f_+^\nu)\to (f_-,f_+)$ as $\nu\to\infty$ translates into 
$(f_-^\nu,f_+^\nu,z^\nu)\to (f_-,f_+)$. In other words, we have a homeomorphism 
\[
\t{NodAnn}\simeq \mathrm{Aut}_0(S^1)\times \mathrm{Aut}_0(S^1) \times (\mathbb{C}^\times\cup\{\infty\}).
\]
Up to an inversion on the factor $\mathbb{C}^\times$, this is the statement of the lemma.
\end{proof} 
 
\begin{remark} \label{rmk:semi-direct-product} Consider the group homomorphism with kernel $\mathrm{Aut}_0(S^1)$ given by the evaluation $\mathrm{Aut}(S^1)\to S^1$, $f\mapsto f(1)$. This admits a section which associates to each element of $S^1$ the corresponding translation, and thus exhibits $\mathrm{Aut}(S^1)$ as a semi-direct product $\mathrm{Aut}(S^1)\simeq \mathrm{Aut}_0(S^1)\rtimes S^1$. Although the action of $S^1$ on $\mathrm{Aut}_0(S^1)$ by conjugation is nontrivial, we do nevertheless have  a homeomorphism at the level of the underlying topological spaces $\mathrm{Aut}(S^1)\simeq \mathrm{Aut}_0(S^1)\times S^1$. On the other hand, there is of course no canonical group structure on the quotient $\mathrm{Aut}(S^1)/S^1$. 
\end{remark}
 
Nodal annuli provide a partial compactification of the space of annuli ``in the modulus $\infty$ limit'', whereas in the previous section we gave a compactification of the space of annuli ``in the modulus $0$ limit''. In particular, these two compactifications can be combined into a new separable topological space of \emph{possibly degenerate stable framed nodal annuli}, 
\[
\widetilde{\t{NodAnn}}  = \widetilde{\t{Ann}}\sqcup_{\t{Ann}}\t{NodAnn}.
\]

Given two possibly degenerate nodal annuli we can glue them to produce a new possibly degenerate nodal annulus. Note that if both annuli have modulus $\infty$ (i.e., have two irreducible components), the resulting glued space will be unstable. Under our convention, we identify the resulting space with its stabilization. It is immediate to check that the resulting composition operation is associative; it is continuous by an argument analogous to the one used in the previous section for the continuity of the multiplication operation on $\widetilde{\t{Ann}}.$

 \subsection{Tree-like nodal surfaces} \label{sec:dendritic}

We recall that all our framed surfaces are labeled, see Remark~\ref{rmk:labelings}. 

 \begin{definition}
Define 
$$
\mathrm{NodFr}_\partial^{\t{tree}}
$$ 
to be the moduli space of stable nodal Riemann surfaces with non-nodal analytically parametrized boundary, with the restriction that \emph{the dual graph of irreducible components is a tree}. Further define 
$$
\widetilde{\mathrm{NodFr}}_\partial^{\t{tree}} = \mathrm{NodFr}_\partial^{\t{tree}}\sqcup_{\mathrm{Ann}}\widetilde{\mathrm{Ann}}.
$$ 
\end{definition}
Note that an element of $\mathrm{NodFr}_\partial^{\t{tree}}$ can have (stable) interior components which carry no boundary parametrizations, and these can have discrete automorphism groups. 
We view $\mathrm{NodFr}_\partial^{\t{tree}}$ as a \emph{topological moduli problem} in the sense of the next definition. We build a theory of such spaces in Appendix~\ref{sec:app_moduli}.

\begin{definition}[See~Definition~\ref{def:topomopro}]
A \emph{topological moduli problem} is a contravariant functor $\Top^{op}\to \t{Gpd}$ from the category of topological spaces to the category of groupoids.
\end{definition}
We write $\t{TMP}$ for the category of such functors, with maps $\X\to \Y$ given by natural transformations. Given a map $f:S\to S'$ of topological spaces we write $f^*:\X(S')\to \X(S)$ for the (contravariantly) associated functor of groupoids. We refer to Appendix~\ref{sec:app_moduli} for further details, and simply recall here that a groupoid is a category $\mathcal{C}$ all of whose morphisms are invertible and such that the isomorphism classes form a set denoted $\pi_0(\mathcal{C})$. 

We view $\mathrm{NodFr}_\partial^{\t{tree}}$ as a topological moduli problem as follows. Given a topological space $S$, an object of the groupoid $\mathrm{NodFr}_\partial^{\t{tree}}(S)$ consists of a continuously varying $S$-family of stable nodal framed Riemann surfaces, and a morphism in this groupoid is an isomorphism of two such families that preserves the structure. We refer to Appendix~\S\ref{sec:app_moduli} for a more precise definition of the meaning of continuity for an $S$-family, which is the analogue of defining the topology for a topological moduli problem.

For the purposes of the current section we limit ourselves to considering the corresponding coarse moduli spaces 
$$
\mathrm{NodFr}_\partial^{\t{tree, coarse}} \qquad \mbox{and}\qquad  \widetilde{\mathrm{NodFr}}_\partial^{\t{tree, coarse}}. 
$$
While any topological moduli problem has an associated coarse moduli space as described in Appendix~\ref{sec:app_moduli}, in our situation the coarse moduli spaces can be obtained from $\mathrm{NodFr}_\partial^{\t{tree}}$, resp. $\widetilde{\mathrm{NodFr}}_\partial^{\t{tree}}$, by topologizing isomorphism classes of points, i.e., the isomorphism classes in the groupoids obtained by applying these TMP's to $\pt$.

Note that $\mathrm{NodAnn}\subset \mathrm{NodFr}_\partial^{\t{tree, coarse}}$. Also, $\widetilde{\mathrm{NodFr}}_\partial^{\t{tree, coarse}}$ differs from $\mathrm{NodFr}_\partial^{\t{tree, coarse}}$ in that it contains degenerate annuli.  We topologize $\widetilde{\mathrm{NodFr}}_\partial^{\t{tree, coarse}}$ as before, by viewing it as embedded in a bundle over the (tree-like) coarse moduli space of closed nodal Riemann surfaces (possibly with some marked points). 

Gluing along the boundary and possibly collapsing determines an operad structure on $\widetilde{\mathrm{NodFr}}_\partial^{\t{tree, coarse}}$. We call it \emph{the operad of possibly degenerate coarse tree-like framed nodal surfaces}, with $n$-to-1 operations given by the coarse moduli spaces of framed nodal curves with $n$ incoming boundary components.  

\begin{remark} In contrast, the fine moduli spaces $\widetilde{\mathrm{NodFr}}_\partial^{\t{tree}}$ fit into the structure of a \emph{Segal operad} in the sense of Appendix~\ref{sec:dendroidal}. We call it \emph{the Segal operad of possibly degenerate tree-like framed nodal surfaces}. This operad will play a role in the proof of our Main Theorem in~\S\ref{sec:proof_main}. 
\end{remark}

For further reference we denote by
$$
\t{NodFr}_\partial
$$
the moduli space of stable nodal Riemann surfaces with non-nodal analytically parametrized boundary, without any restriction on the dual graph, and also 
$$
\widetilde{\t{NodFr}}_\partial = \t{NodFr}_\partial \sqcup_{\mathrm{Ann}}\widetilde{\mathrm{Ann}}.
$$

\begin{theorem}[Geometric Pushout Theorem]\label{thm:geomthm}
The operad 
$$
\widetilde{\mathrm{NodFr}}_\partial^{\t{tree, coarse}}
$$ 
of coarse moduli spaces of possibly degenerate tree-like framed nodal surfaces is canonically isomorphic to the pushout of the following diagram, in which both arrows are inclusions and we work in the category of topological operads:
$$
\widetilde{\mathrm{NodAnn}} \longleftarrow \widetilde{\mathrm{Ann}} \longrightarrow \widetilde{\mathrm{Fr}}_\partial.
$$
\end{theorem}

The geometric idea of the proof is that a nodal surface can be described, though not uniquely, by a successive gluing of framed non-nodal surfaces and nodal annuli. See Figure~\ref{fig:gluing}. When the dual graph of irreducible components is a tree, this data is equivalently encoded in the pushout construction. The equivalence relations defining the pushout construction precisely eliminate the ambiguity, i.e., non-uniqueness, of this description. The equivalence relations underlie pushouts in the topological category, and in particular (as we are not taking homotopy pushouts yet), self-equivalences are ignored, and this explains the ``coarse'' nature of the result.

\begin{figure} [ht]
\centering
\input{gluing.pstex_t}
\caption{One possible presentation of a nodal surface by gluing.}
\label{fig:gluing}
\end{figure}

As such the proof of Theorem~\ref{thm:geomthm} on the level of operads in sets is quite straightforward. 
By re-interpreting the free-forgetful adjunction on the operad of framed surfaces and its relatives, we give a proof of this theorem which also accounts for the topology on the two sides. While the theorem does not imply the homotopy-theoretic pushout result (in order to get a correct model for the homotopy pushout, the diagram of operads must be replaced by a suitable resolution), it is a good intuitive approximation for it. Indeed, the eventual homotopical proof will be based on a topologically enhanced version of exactly the argument presented in the next section.

\subsection{Split Surfaces and the Geometric Pushout Theorem}\label{sec:proofgeom}
The objects of interest in this section will be various moduli spaces of framed surfaces with ``seams'' at embedded curves, which we call ``split surfaces''. We again recall that all our framed surfaces are labeled, see Remark~\ref{rmk:labelings}. 

\begin{definition}
  A \emph{split framed surface with $k$ interior seams} is a pair $(\Sigma,S)$ consisting of a framed surface $\Sigma$ with boundary, together with an analytic embedding $S:(S^1)^{\sqcup k}\xhookrightarrow{} \Sigma$ mapping into the interior of $\Sigma$. 
\end{definition}

By definition, the seams are \emph{parametrized} curves: the \emph{interior seams} are the components $S_i:S^1\to \Sigma$ of the embedding $S=\sqcup_{i=1}^k S_i:(S^1)^{\sqcup k}\xhookrightarrow{} \Sigma$; the parametrized boundary components of $\Sigma$ are called \emph{exterior seams}. In the definition we allow $k=0$, i.e., no interior seams. 

Given a framed surface $\Sigma$, write 
\[
\t{Split}_\Sigma^k
\] 
for the moduli space of all split surface structures on $\Sigma$ with $k$ unordered interior seams. Equivalently, $\t{Split}_\Sigma^k$ is the space of analytic embeddings $(S^1)^{\sqcup k}\xhookrightarrow{} \Sigma$ endowed with the compact-open topology. Write 
\[
\t{Split}^k \,\,\,\,(\,=\,\sqcup_{\Sigma\in \t{Fr}_\partial} \t{Split}_\Sigma^k)
\] 
for the moduli space of all split framed surfaces with $k$ unordered interior seams, and write 
\[
\t{Split}\,\,\,\,(\,=\,\sqcup_{k\ge 0} \t{Split}^k)
\]
for the moduli space of all split framed surfaces with an arbitrary number of unordered interior seams, and 
\[
\t{Split}_\Sigma\,\,\,\,(\,=\,\sqcup_{k\ge 0} \t{Split}^k_\Sigma)
\]
for the moduli space of all split surface structures on $\Sigma$ with an arbitrary number of unordered interior seams. 

To every split surface $(\Sigma, S)$ is associated a ``dual graph'' $$\Gamma_{\Sigma, S}$$ with $k$ interior edges, which is a directed graph with half-edges. Vertices are indexed by the connected components of $\Sigma\setminus S,$ internal edges are indexed by interior seams (the orientation of the normal bundle along a seam determines a direction for the corresponding edge) and half-edges are indexed by external seams (each of these belongs to the closure of a single connected component of $\Sigma\setminus S$). In particular, since the incoming external seams of $\Sigma$ are labeled by definition, the dual graph inherits a labeling of its incoming half-edges. Note that two split surfaces in the same connected component of $\t{Split}$ have the same dual graph, so given a labeled graph $\Gamma$ we can write 
\[
\t{Split}_\Gamma
\] 
for the union of connected components of $\t{Split}$ with dual graph $\Gamma$. 
The following observation is straightforward. 

\begin{lemma}
Let $\Sigma$ be connected. The dual graph $\Gamma_{\Sigma, S}$ associated to a split surface $(\Sigma, S)$ is a tree if and only if the image of each interior seam is separating, i.e., its complement is disconnected. \qed
\end{lemma}

Split framed surfaces are a convenient model for the free operad on the $\fS$-graded space underlying $\t{Fr}_\partial$ (the source of the free-forgetful adjunction map), as we now explain. Write 

\[
\t{Split}_\Sigma^{k,\t{tree}} \subset \t{Split}_\Sigma^k
\] 

for the moduli space of all split surface structures $S$ on $\Sigma$ with $k$ unordered interior seams such that the dual graph $\Gamma_{\Sigma,S}$ is a tree. Further denote 

\[
\t{Split}^{k,\t{tree}} \,\,\,\,\,=\,\sqcup_{\Sigma\in \t{Fr}_\partial} \t{Split}_\Sigma^{k,\t{tree}} \subset \t{Split}^k,
\] 

\[
\t{Split}^{\t{tree}}\,\,\,\,\,=\,\sqcup_{k\ge 0} \t{Split}^{k,\t{tree}}\subset \t{Split},
\]

\[
\t{Split}^{\t{tree}}_\Sigma\,\,\,\,\,=\,\sqcup_{k\ge 0} \t{Split}^{k,\t{tree}}_\Sigma \subset \t{Split}_\Sigma,
\]
and 
\[
\t{Split}^{\t{tree}}_\Gamma = \t{Split}_\Gamma 
\]
for any labeled tree $\Gamma$. 
We call these moduli spaces of split surface structures \emph{tree-like}. 

Let $\tau$ be a labeled tree of operations 
and let $[\tau]$ be its isomorphism class with respect to the isomorphism relation described in~\S\ref{sec:operad_of_labeled_rooted_trees}.   
Recall that, for any operad $O$, the free operad $\t{Free}(O)$ has components $\t{Free}_{[\tau]}(O)$ indexed by such isomorphism classes. 

\begin{lemma} \label{lem:from_free_to_seamed}
Let $\tau$ be a labeled tree of operations. 
We have a canonical homeomorphism 
\[
G:\t{Free}_{[\tau]}(\t{Fr}_\partial)\stackrel{\simeq}{\longrightarrow} \mathrm{Split}^{\t{tree}}_\tau.
\] 
\end{lemma}

The fact which underlies the proof of Lemma~\ref{lem:from_free_to_seamed} is that, given a framed surface $\Sigma''$, the data of an interior seam whose image is separating is equivalent to the data of a decomposition of $\Sigma''$ as a gluing of two framed surfaces along one boundary component.  Obviously, such a seam determines such a decomposition of $\Sigma''$. Conversely, given two framed surfaces $\Sigma,\Sigma'$ and a choice of boundary components $\gamma\subset \Sigma$ which is incoming and $\gamma'\subset \Sigma'$ which is outgoing, with corresponding framings $f:S^1\to \gamma$ and $f':S^1\to\gamma'$, the glued surface $\Sigma''=\Sigma\sharp_{\gamma, \gamma'}\Sigma'$
inherits a seam, i.e., a distinguished analytic embedding of $S^1$ into its interior, given with respect to the canonical inclusions $\Sigma,\Sigma'\hookrightarrow \Sigma''$ by either of the equal compositions 
\[
S^1\stackrel{f}{\longrightarrow} \gamma\hookrightarrow\Sigma\hookrightarrow \Sigma''
\]
or
\[
S^1\stackrel{f'}{\longrightarrow} \gamma'\hookrightarrow\Sigma'\hookrightarrow \Sigma''. 
\]

\begin{proof}[Proof of Lemma~\ref{lem:from_free_to_seamed}] 
Let $n\ge 0$ 
be the number of incoming half-edges of $\tau$. Recall from~\S\ref{sec:free_operad_functor} that $\t{Tree}_n(\t{Fr}_\partial)$ denotes the space of labeled rooted trees with vertices colored by elements of $\t{Fr}_\partial$ and with $n$ incoming half-edges. Denote $\t{Tree}_{[\tau]}(\t{Fr}_\partial)\subset \t{Tree}_n(\t{Fr}_\partial)$ the subset consisting of those elements whose underlying labeled rooted tree is isomorphic to $\tau$. We have a canonical ``gluing'' map 
\[
G:\t{Tree}_{[\tau]}(\t{Fr}_\partial)\to \t{Split}^{\t{tree}}_\tau
\]
given by gluing framed surfaces according to the underlying labeled rooted tree. Indeed, the incoming boundary components of the element of $\t{Fr}_\partial$ that colors a vertex $v\in\tau$ are labeled by the finite set $\t{in}(v)$, and this prescribes the gluing uniquely. By definition, the resulting split surface belongs to $\t{Split}^{\t{tree}}_\tau$. 

The map is clearly continuous, surjective, and the fiber over each element of $\t{Split}^{\t{tree}}_\tau$ is canonically identified with an equivalence class as described in~\S\ref{sec:free_operad_functor}. As such, it descends to a homeomorphism
\[
G:\t{Free}_{[\tau]}(\t{Fr}_\partial)\stackrel\simeq\longrightarrow \t{Split}^{\t{tree}}_\tau,
\]
where $\t{Free}_{[\tau]}(\t{Fr}_\partial)=Tree_{[\tau]}(\t{Fr}_\partial)$ is the quotient of $\t{Tree}_{[\tau]}(\t{Fr}_\partial)$ under the equivalence relation described in~\S\ref{sec:free_operad_functor}. 
\end{proof}

To extend the above result to $\widetilde{\t{Fr}}_\partial,$ we compactify $\t{Split}$ by allowing interior components of thickness zero:
\begin{definition} \label{defi:split-tilde}
  Let 
  $$
  \widetilde{\mathrm{Split}}
  $$ 
  be the partial compactification of $\mathrm{Split}$ which allows two seams (internal or external) $S^1\to \Sigma$ to intersect if and only if they have the same image with the same orientation, and which also allows $\Sigma$ to be a framed degenerate annulus. 
\end{definition}
We have corresponding partial compactifications 
$$
\widetilde{\t{Split}}^k_\Sigma,\qquad \widetilde{\t{Split}}^k, \qquad \widetilde{\t{Split}}_\Sigma
$$
of the moduli spaces $\t{Split}^k_\Sigma$, $\t{Split}^k$, and $\t{Split}_\Sigma$ respectively, and also for their tree-like and labeled tree-like counterparts, with similar notations $\widetilde{\t{Split}}^{\t{tree}}$ etc. 

Points of $\widetilde{\t{Split}}_\Sigma$ over a fixed surface $\Sigma$ are indexed by maps $S: (S^1)^{\sqcup k}\to \Sigma$ which allow seams with compatible orientation to coincide as above, with the additional data of an ordering of all copies of $S^1$ mapping to a given closed oriented curve. The notion of dual graph $\Gamma=\Gamma_{(\Sigma,S)}$ for such an element $(\Sigma,S)$ is defined as follows. The vertices of $\Gamma$ are of two kinds: they correspond either to the connected components of $\Sigma\setminus S$, or to pairs of interior seams which have the same image and which are immediate successors for the given ordering. The edges correspond to interior seams. One sees that the ordering of the copies of $S^1$ mapping to a given closed oriented curve precisely resolves the ambiguity in the dual graph by specifying a ``composition order'' of the thickness-zero annuli they ``bound''. 

Given a labeled tree $\Gamma$ 
we have corresponding moduli spaces $\widetilde{\t{Split}}_\Gamma=\widetilde{\t{Split}}^{\t{tree}}_\Gamma$. The proof of the following lemma is in all points similar to that of Lemma~\ref{lem:from_free_to_seamed}, hence we omit it. 

\begin{lemma} \label{lem:from_free_to_seamed_unital}
Let $\tau$ be a labeled tree of operations. 
We have a canonical homeomorphism 
\[
G:\t{Free}_{[\tau]}(\widetilde{\t{Fr}}_\partial)\stackrel{\simeq}{\longrightarrow} \widetilde{\t{Split}}^{\t{tree}}_\tau. 
\] 
\qed
\end{lemma}

In order to extend the result to
$\widetilde{\t{NodFr}}_\partial^\t{tree, coarse}$
we need to further define moduli spaces of framed nodal surfaces with seams. 

\begin{definition}
  Let
  
  \[
  \t{NodSplit}
  \]
  
  be the moduli space of framed nodal surfaces $\Sigma$ endowed with an embedding $(S^1)^{\sqcup k}\to \mathring{\Sigma}_{\t{smooth}}$ of a finite number $k\ge 0$ of parametrized seams \emph{in the open smooth locus}. 
  The objects classified by $\t{NodSplit}$ are called \emph{split framed nodal surfaces}. 
\end{definition}

The notion of dual graph for a split framed nodal surface $(\Sigma,S)$ is defined as follows: its vertices are the \emph{connected} components (not the irreducible components) of $\Sigma\setminus S$, and in particular the dual graph in this context ignores nodes. The edges correspond to interior seams as before. We can further define moduli spaces $\t{NodSplit}^{\t{tree}}$, $\t{NodSplit}^{\t{tree, coarse}}$ etc. as above. 

It is again convenient for unitality purposes to extend the setup by including degenerate annuli. 

\begin{definition} \label{defi:nod-split-tilde}
  Let 
  $$
  \widetilde{\t{NodSplit}}
  $$ 
  be the partial compactification of $\t{NodSplit}$ obtained by allowing $S$ to include coinciding circles bounding thickness-zero annuli, as in Definition~\ref{defi:split-tilde}. 
\end{definition}

Similarly to the non-nodal case, we consider as part of the data an ordering of the interior seams which have the same oriented image. We have the same notion of dual graph, and we can further define moduli spaces $\widetilde{\t{NodSplit}}^{\t{tree}}$, $\widetilde{\t{NodSplit}}^{\t{tree, coarse}}$ etc. as above.

In the proof of the Geometric Pushout Theorem~\ref{thm:geomthm} we will encounter the following new kind of moduli space. We single out the definition before the proof, for the convenience of the reader. 

\begin{definition}
  Define 
  \[
  \widetilde{\t{NodSplit}}^{\t{tree}}_{\t{protected}}
  \] 
  to be the moduli space of split nodal surfaces with dual graph a tree (with half-edges) and such that every nodal component is a nodal annulus. We call such surfaces \emph{tree-like and protected}. 
\end{definition}

The idea of the definition is that every node has to be ``protected'' on two sides by a pair of seams.

\begin{figure} [ht]
\centering
\input{split-surface.pstex_t}
\caption{Tree-like split structure on a nodal surface, together with its dual graph. It becomes ``protected" by adding one seam around the node $N$ on the trivalent component.}
\label{fig:split-surface}
\end{figure}

\begin{proof}[Proof of the Geometric Pushout Theorem~\ref{thm:geomthm}] 
Consider the diagram 
\[
\widetilde{\mathrm{NodAnn}} \longleftarrow \widetilde{\mathrm{Ann}} \longrightarrow \widetilde{\mathrm{Fr}}_\partial.
\] 
Recall from~\S\ref{sec:pushout_operads} the definition of its pushout 
\[
P\simeq \t{Free}\big(\widetilde{\mathrm{Fr}}_\partial\sqcup \widetilde{\mathrm{NodAnn}}\big) /  \sim,
\]
where $\sim$ is the equivalence relation generated by relations $\sim_1\sqcup \sim_2$. 

As in Lemma~\ref{lem:from_free_to_seamed} there is a tautological gluing map
\[
\t{Free}\big(\widetilde{\mathrm{Fr}}_\partial\sqcup \widetilde{\mathrm{NodAnn}})\to \widetilde{\t{NodSplit}}^{\t{tree, coarse}}.
\] 
The preimage of this map over a given split surface $(\Sigma,S)$ consists of all possible choices of decorating the vertices of the dual graph of $(\Sigma, S)$ by the corresponding point of 
$
\widetilde{\mathrm{NodAnn}}
$ 
or  
$
\widetilde{\mathrm{Fr}}_\partial.
$ 
Thus, in order for a split surface $(\Sigma,S)$ to have nonempty preimage, connected components of $\Sigma\setminus S$ must be either smooth surfaces or nodal annuli. Components indexed by non-nodal annuli can be labeled either way and contribute to the ambiguity of the lifting. It is precisely this ambiguity that is resolved by the relation $\sim_2$, in a way that is compatible with the topology as it identifies connected components in their entirety. Thus the map to 
$\widetilde{\t{NodSplit}}^{\t{tree, coarse}}$ above factors as 
\begin{equation} \label{eq:from_free_to_protected}
\xymatrix{
\frac{\t{Free}\big(\widetilde{\mathrm{Fr}}_\partial\sqcup \widetilde{\mathrm{NodAnn}})}{\sim_2} \ar[r] \ar[dr]_\simeq & 
\widetilde{\t{NodSplit}}^{\t{tree, coarse}}, \\
& \widetilde{\t{NodSplit}}^{\t{tree}}_{\t{protected}} \ar@{^(->}[u]
}
\end{equation}
where it maps homeomorphically the partial quotient to the target $\widetilde{\t{NodSplit}}^{\t{tree}}_{\t{protected}}$, which is in turn a union of connected components of $\widetilde{\t{NodSplit}}^{\t{tree}}$ consisting of split nodal curves with dual graph a tree, and such that each nodal component is a nodal annulus. (Note that as protected split curves are glued out of smooth framed curves and nodal annuli, neither of which have automorphisms, there is no need to take the coarse space here.)

Consider now the map 
\begin{equation}\label{eq:from_protected_to_unprotected}
\widetilde{\t{NodSplit}}^{\t{tree}}_{\t{protected}}\to \widetilde{\t{NodFr}}_\partial^{\t{tree, coarse}}
\end{equation}
defined by erasing the seams. Note that erasing a seam which is a common boundary component of two framed surfaces in $\widetilde{\t{Fr}}_\partial$ corresponds precisely to gluing, i.e., composition in the operad $\widetilde{\t{Fr}}_\partial$. Similarly, erasing a seam which is a common boundary component of two nodal annuli creates an unstable component which must be further discarded, and this corresponds again to gluing, i.e., composition in the operad $\widetilde{\t{NodAnn}}$. 

It thus follows that the above map is constant along the equivalence classes defined by relation $\sim_1$, which identifies pairs of points inside $\widetilde{\t{NodSplit}}^{\t{tree}}_{\t{protected}}$ which are related by removing a single seam (note that such a seam must either be between two nodal annuli or between two smooth framed surfaces). On the level of sets, it is clear that $\sim_1$ identifies any two points in $\widetilde{\t{NodSplit}}^{\t{tree}}_{\t{protected}}$ which correspond to splittings of the same nodal curve. We turn this intuition into a precise topological colimit argument as follows.

Given a tree-like nodal surface $\Sigma$ with $k$ nodes and given mutually disjoint neighborhoods $\mathcal{V}_i$, $i=1,\dots,k$ of its nodes, denote $\mathcal{V}=\sqcup_{i=1}^k \mathcal{V}_i$ and write $\t{Split}_\Sigma^{\mathcal{V}}\subset \t{Split}_\Sigma^{\t{tree}}$ for those tree-like split surface structures on $\Sigma$ whose seams lie away from $\mathcal{V}$. Since seams are not allowed to pass through nodes, these spaces filter $\t{Split}_\Sigma^{\t{tree}}$ as $\mathcal{V}$ runs over a neighborhood basis of the nodes of $\Sigma$. Now write $(\Sigma,S_{\mathcal{V}})\in \t{Split}_\Sigma^{\t{tree}}$ for a splitting given by a collection of $2k$ circles parametrized in some analytic fashion and with images contained in $\mathcal{V}$, such that each neighborhood $\mathcal{V}_i$ of a node contains exactly two such circles, one on each irreducible component adjacent to the node. Then every element in $\t{Split}_\Sigma^{\mathcal{V}}$ is identified (in a way consistent with the topology) with $(\Sigma, S_{\mathcal{V}})$ via $\sim_1$. Further, for $\mathcal{V}'\subset \mathcal{V}$ we have $(\Sigma, S_{\mathcal{V}'})\sim (\Sigma, S_{\mathcal{V}})$: indeed, by $\sim_1$ used for $\widetilde{\t{NodAnn}}$ they are both equivalent to the split surface $(\Sigma,S_{\mathcal{V}''})$ for some sufficiently small neighborhood $\mathcal{V}''$ which does not intersect $S_{\mathcal{V}}\cup S_{\mathcal{V}'}$.

We have thus proved that the fiber of the map~\eqref{eq:from_protected_to_unprotected} is given by the equivalence classes with respect to $\sim_1$.  As a consequence, the map~\eqref{eq:from_protected_to_unprotected} is a bijection, and because the previous identifications can be performed continuously in a neighborhood of any given nodal surface $\Sigma$, this map is also continuous. Finally, we claim that the map is a homeomorphism. To prove the claim note that, given any tree-like nodal split surface $(\Sigma,S)$, its image is $\Sigma$ with the \emph{same} analytic parametrization of the boundary. Thus, in order to prove the claim, it is enough to prove that the induced map on the moduli spaces of surfaces with one marked point on each boundary component is a homeomorphism. This holds true because it is a continuous bijection (just like~\eqref{eq:from_protected_to_unprotected}), with a Hausdorff source and a locally compact target. We infer that the map~\eqref{eq:from_protected_to_unprotected} is a homeomorphism as claimed. (The reduction to moduli spaces of surfaces with one marked point on each boundary component, which gets rid of the infinite dimensional degrees of freedom given by the analytic parametrization, was necessary precisely in order to place ourselves in a setup with locally compact target.) 

Together with the homeomorphism~\eqref{eq:from_free_to_protected}, we obtain a homeomorphism
\[
\widetilde{\mathrm{Fr}}_\partial \sqcup_{\widetilde{\mathrm{Ann}}}\widetilde{\mathrm{NodAnn}}\cong \widetilde{\mathrm{NodFr}}_\partial^{\t{tree, coarse}}.
\] 
\end{proof}

\section{Model Categories and Homotopy (Co)limits}  \label{sec:models}

Our references for this section are Lurie~\cite[Appendix~A.2]{lurie-htt}, May-Ponto~\cite{may-ponto}, Hovey~\cite{hovey}, Hirschhorn~\cite{hirschhorn} and Ginot~\cite{ginot}.

\subsection{Model category theory}  \label{sec:model_cats}
Suppose that $\cC$ is a category and $I$ is a class of morphisms in $\cC$ ``to be inverted''. We say that $I$ is a \emph{class of weak equivalences} if the following conditions are satisfied:
\begin{itemize}
\item{\sc (category structure).} The objects of $\cC$ with the morphisms in $I$ form a subcategory. 
\item{\sc (2 out of 3).} Given any commutative diagram
\[
\xymatrix{
  & A \ar[d] \ar[dl]\\
  B \ar[r] &C
}
\]
with two of the three morphisms in $I$, the third is also in $I$. 
\end{itemize}
Note that the first axiom is sometimes replaced by an identity axiom, as composition compatibility is part of the {\sc (2 out of 3)} axiom. Now given a class of weak equivalences, one would like to produce a ``localized'' category in which these are inverted, i.e., a category $\cC_I$ with a functor $\cC\to \cC_I$ such that the image of any morphism in $I$ is invertible, and which is initial---up to taking care of set-theoretic issues---among such categories. Modulo some set-theoretic difficulties such a $\cC_I$ can be proven to exist. In fact, when $\cC$ is an ordinary category, the localization $\cC_I$ comes naturally as the set of connected components of morphism spaces in a simplicial category, which should be considered in the context of $\infty$-category theory.

The problem is that for a general class $I$ of weak equivalences, the localization $\cC_I$ (whether as a category or a simplicial category) is incredibly difficult to access. In particular, it is hopeless to calculate $\hom_{\cC_I}(X,Y)$ for two objects $X, Y$ of $\cC$. In order to turn $\cC_I$ into a manageable object, it is necessary to endow $\cC$ with some additional data. One remarkably elegant and versatile solution is to exhibit a so-called model category structure. A \emph{model category structure} consists in endowing $\cC$ with two new classes of morphisms called \emph{fibrations}, $P$, and \emph{cofibrations}, $Q$, such that the objects of $\cC$ with either $P$ or $Q$ form subcategories of $\cC$. 
We call the elements of $I\cap P$ \emph{trivial fibrations}, and the elements of $I\cap Q$ \emph{trivial cofibrations}. 
The category $\cC$ together with the classes $I, P, Q$ need to satisfy a collection of conditions among themselves, for which we refer the reader to \cite[\S3]{hinich}. Some conditions that we will use here are as follows.
\begin{enumerate}
\item The category $\cC$ has an initial object, $\emptyset$, a final object, $\text{pt},$ and all finite limits and colimits.
\item For any morphism $X\xrightarrow{f} Y$ of objects, there is a ``fibrant factorization'' $X\xrightarrow{i} X'\xrightarrow{f'}Y$ such that $i\in I \cap Q$ is a trivial cofibration and $f'\in P$ is a fibration.
\item Similarly, for any morphism $X\xrightarrow{f}Y$ of objects, there is a ``cofibrant factorization'' $X\xrightarrow{f'}X'\xrightarrow{j}Y$ such that $f'\in Q$ is a cofibration and $j\in I\cap P$ is a trivial fibration. 
\item All three categories $P, Q, I$ are closed with respect to taking retracts of morphisms.
\item Given the subcategories $I$ of weak equivalences and $Q$ of cofibrations (resp., the subcategory $P$ of fibrations), the subcategory $P$ of fibrations (respectively, $Q$ of cofibrations) is uniquely characterized by a lifting property.
\end{enumerate}
Note that neither cofibrant nor fibrant factorization is required to be functorial, though there often is a functorial choice (in fact, there is a sense in which the choice is unique up to homotopy). 
If a map $X\xrightarrow{f} Y$ is a fibration we write shorthand
\[X\xtwoheadrightarrow{f}Y,\]
and similarly if $X\xrightarrow{f} Y$ is a cofibration we write
\[X\xhookrightarrow{f}Y.\]
If $X\xrightarrow{f}Y$ is an equivalence we write 
$X\xrightarrow[\sim]{f} Y$, with evident compound meanings for $X\xhookrightarrow[\sim]{f} Y$ 
(trivial cofibration) 
and $X\xtwoheadrightarrow[\sim]{f} Y$ 
(trivial fibration). 

\subsection{The homotopy category}  \label{sec:homotopy_category}
Suppose that $\cC$ is a category with weak equivalences $I$ and model structure $P, Q$. 
\begin{itemize}
\item We say that an object $X$ is \emph{fibrant} if the map $X\to \pt$ to the terminal object is a fibration. 
\item We say that an object $X$ is \emph{cofibrant} if the map $\emptyset\to X$ is a cofibration. 
\end{itemize}
Note that, by applying a suitable factorization axiom to the map $\emptyset \to X$ or $X\to \pt$, every object $X$ admits a trivial cofibration to a fibrant object, $X\xhookrightarrow{\sim} X_P\xtwoheadrightarrow{} \pt$, and a trivial fibration from a cofibrant one, $\emptyset \xhookrightarrow{} X_Q\xtwoheadrightarrow{\sim} X$. We call $X_P$, resp. $X_Q$, a \emph{fibrant, resp. cofibrant, replacement of $X$}. The fibrant, resp. cofibrant, replacements are in general not canonical, but in many situations of interest they can be chosen to be functorial. The $W$-construction discussed in~\S\ref{sec:W} provides such a functorial cofibrant replacement for operads. 

Let $X\sqcup X\xrightarrow{\one\sqcup \one} X$ be the codiagonal map, and $X\sqcup X\xhookrightarrow{}  C_X\xtwoheadrightarrow{\sim} X$ a factorization. Any such object $C_X$ is called \emph{a cylinder object} for $X$. It admits a trivial fibration $C_X\xtwoheadrightarrow{\sim} X$ and two cofibrations $X\xhookrightarrow{i_0, i_1} C_X$, which are also weak equivalences by the {\sc (2 out of 3)} axiom. 

Similarly, let $\Delta:X\to X\times X$ be the diagonal map, and $X\xhookrightarrow{\sim} P_X\xtwoheadrightarrow{} X\times X$ a factorization. Any such object $P_X$ is called \emph{a path object} for $X$. It admits a trivial cofibration $X\xhookrightarrow{\sim} P_X$ and two fibrations $P_X\xtwoheadrightarrow{p_0, p_1} X$, which are also weak equivalences by the {\sc (2 out of 3)} axiom.
\begin{definition}
  Write $\cC_P, \cC_Q, \cC_{QP}$ for the full subcategories of $\cC$ consisting of fibrant, cofibrant, and fibrant-cofibrant objects, respectively. 
\end{definition}
\begin{definition}
  Suppose that $f, g:X\to Y$ is a pair of maps, and choose a cylinder object $C_X$ and a path object $P_Y$. 
  \begin{itemize}
  \item $f$ and $g$ are \emph{left homotopic} if the map $f\sqcup g: X\sqcup X\to Y$ factors through $C_X$ as 
  $$
  X\sqcup X\xhookrightarrow{i_0\sqcup i_1} C_X\xrightarrow{h} Y
  $$ 
  for some choice of map (``homotopy'') $h$. 
  \item $f$ and $g$ are \emph{right homotopic} if the map $X\xrightarrow{f\times g} Y\times Y$  
  factors through $P_Y$ as 
  $$
  X\xrightarrow{k} P_Y\xtwoheadrightarrow{} Y\times Y
  $$ 
  for some choice of map (``cohomotopy'') $k$.
 \end{itemize}
\end{definition}
\begin{lemma}[\cite{quillen}, {\cite[Proposition~1.2.5]{hovey}}]\qquad 

  If $X$ is cofibrant (and $Y$ arbitrary), the relation $\sim_L$ of left homotopy equivalence on $\hom(X, Y)$ is an equivalence relation, and does not depend on the choice of cylinder object $C_X.$ 
  
  If $Y$ is fibrant (and $X$ is arbitrary), the relation $\sim_R$ of right homotopy equivalence on $\hom(X, Y)$ is an equivalence relation and does not depend on choice of path object $P_Y.$ 
  
  If $X$ is fibrant and $Y$ is cofibrant, then the two equivalence relations $\sim_L$ and $\sim_R$ on $\hom(X,Y)$ are the same. 
\end{lemma}
\begin{definition}
  The category $Ho_\cC$ is the category with objects $\cC_{QP}$ and morphisms $\hom_{Ho_\cC}(X, Y)$ defined as the quotient of $\hom_{\cC}(X, Y)$ by left (or, equivalently, right) homotopy equivalence.
\end{definition}
\begin{theorem}[\cite{quillen}, {\cite[Theorem~1.2.10]{hovey}}]
  The homotopy category $Ho_\cC$ is canonically equivalent to the localized category $\cC[I^{-1}].$ 
\end{theorem}

\begin{remark}
  Recall that, given a ring $A$ with a localizing set of elements $I$, there is a condition on $I$ called the \emph{left (resp., right) Ore condition} which allows one to write down the localization $A[I^{-1}]$ as the ring of fractions $i^{-1}f$ (resp., $fi^{-1}$) for $i\in I$. Similarly, given a category $\cC$ there is a notion of left (resp., right) Ore condition, which is part of a so-called ``calculus of fractions" on $\cC$~\cite[A.2.1.11(h)]{johnstone-vol1}. If the left Ore condition is satisfied then the category $\cC[I^{-1}]$ can be expressed as the category of objects of $\cC$ with morphisms $X\to Y$ represented by ``roofs'' $X\xrightarrow{f} Z\xleftarrow[\sim]{g}Y,$ with $Z$ arbitrary and $g$ a weak equivalence, subject to a straightforward equivalence relation determined by diagrams of maps commuting with a weak equivalence $Z'\xrightarrow{\sim} Z$. If $\cC$ is a model category then the category of cofibrant objects and maps up to left homotopy satisfies the left Ore condition with quotient $Ho_\cC$, and the category of fibrant objects and maps up to right homotopy satisfies the right Ore condition with quotient $Ho_\cC.$ 
\end{remark}
\subsection{Some important model categories}
We will give a few examples of model category structures on simplicial sets, topological spaces and differential complexes that will be important to us. Recall that in order to define a model structure, it suffices to specify just two classes of morphisms: either weak equivalences and fibrations, or weak equivalences and cofibrations. The third class is then determined by a lifting property. 

\subsubsection{Model category structure on simplicial sets} \label{sec:model-SSet}
Let $\t{SSet}$ be the category of simplicial sets. A map of simplicial sets $f:X\to Y$ is called a \emph{weak homotopy equivalence} if it induces a weak homotopy equivalence between geometric realizations $|f|:|X|\to |Y|$.  We denote by $\operatorname{WE}$ the class of weak homotopy equivalences. We say that $f$ is a \emph{Kan fibration} if it has the right lifting property with respect to the inclusions of all horns $\Lambda^n_k\hookrightarrow \Delta^n$, $n\ge 0$, $0\le k\le n$. Here $\Lambda^n_k$ is the simplicial subset of $\Delta^n$ obtained by removing the nondegenerate $n$-simplex and the face opposite to the $k$-th vertex.
\begin{theorem}[{Quillen model structure~\cite[II.3, Theorem~3]{quillen}}]
There is a model structure on the category $\t{SSet}$ with weak equivalences given by $\operatorname{WE}$ and fibrations given by Kan fibrations. The cofibrations are the maps of simplicial sets that are degree-wise inclusions. In particular, any simplicial set is a cofibrant object.  
\end{theorem}
Other references for this foundational theorem are~\cite[Chapter~I, Theorem~11.3]{goerss-jardine}, \cite[Chapter~3]{hovey}, \cite[Chapter~8, Theorem~8.19]{heuts-moerdijk}.

\subsubsection{Model category structures on topological spaces}
Let $\Top$ be the category of 
compactly generated weakly Hausdorff topological spaces, see~\cite{strickland}, \cite[Definition~2.4.21, Theorem~2.4.25]{hovey}, \cite{nlab:stroem_model_structure}, \cite[\S6.4]{may_concise}. 
Recall that a map $f:X\to Y$ is a homotopy equivalence if it admits a homotopy inverse, and is a weak homotopy equivalence if it is a bijection on path-connected components and, for any $x\in X$, the map $\pi_n(X, x)\to \pi_n(Y, f(x))$ is an isomorphism for each $n\ge 1$. Any homotopy equivalence $X\to Y$ is a weak homotopy equivalence; the converse is true provided $X,Y$ are CW complexes, but not true in general. Both homotopy equivalences and weak homotopy equivalences clearly satisfy the conditions required to define a class of weak equivalences. We denote by $\operatorname{WE}$ the class of weak homotopy equivalences. 
\begin{theorem}[{Quillen model structure~\cite[II.2, Theorem~1]{quillen}}]
There is a model structure on the category $\top$ with weak equivalences given by $\operatorname{WE}$ and fibrations given by Serre fibrations. A space is cofibrant in this model structure if and only if it is a retract of a relative CW complex, and any space is fibrant. \qed
\end{theorem}
\begin{theorem}[{Str\o m model structure~\cite{strom}}]
  There is a model structure on the category $\top$ with weak equivalences given by homotopy equivalences and with fibrations given by Hurewicz fibrations. The cofibrations are retracts of Hurewicz cofibrations with closed image, and in particular any space is cofibrant. Also, any space is fibrant. 
\end{theorem}
\begin{theorem}[{Mixed model structure, Cole~\cite{cole}, see also~\cite[\S17.3-4]{may-ponto}}]
   There is a model structure on the category $\top$, called \emph{mixed} model structure, with weak equivalences given by $\operatorname{WE}$ and fibrations given by Hurewicz fibrations. A space is cofibrant in the mixed model structure if and only if it is homotopy equivalent to a CW complex, and any space is fibrant. 
 \end{theorem}

\subsubsection{Chain complexes}
Let $k$ be a ring (e.g. $k=\Z$ or $k=\Q$). Then the categories $C(k)$ (resp. $C_+(k)$) 
of chain complexes of $k$-modules (resp. supported in non-negative degrees) have model structures with weak equivalences given by quasi-isomorphisms and fibrations given by maps of complexes which are term-wise surjective (resp. in all positive degrees). In particular all objects are fibrant. Cofibrant objects in $C_+(k)$ are term-wise projective complexes of $k$-modules (see~\cite[Remark~2.3.7]{hovey} for a discussion of cofibrant objects in $C(k)$). This is called the \emph{standard} or \emph{projective} model structure on the category of chain complexes~\cite[\S2.3]{ginot}, \cite[\S18.4-5]{may-ponto}, \cite[\S2.3]{hovey}.

\subsection{Quillen adjunction}  \label{sec:Quillen}
It is a natural question to ask when a functor of model categories induces a functor of homotopy categories, and when this functor is a weak equivalence. (The functor most interesting for us will be the functor of chains from topological operads up to weak homotopy equivalence to dg operads up to quasi-isomorphism.) A convenient condition on a functor $F:\cC\to \cD$ of model categories that guarantees (in a functorial way) a functor on homotopy categories is the notion of so-called \emph{Quillen adjunction}.
\begin{definition}
  A functor $F:\cC\to \cD$ between model categories is a \emph{left Quillen functor} if it admits a right adjoint $G:\cD\to \cC$ such that $F$ preserves cofibrations and trivial cofibrations and $G$ preserves fibrations and trivial fibrations. In this situation we call $G$ a \emph{right Quillen functor}, and the adjunction $(F,G)$ is called a \emph{Quillen adjunction}.  
\end{definition}
Quillen adjunctions induce pairs of adjoint (in the conventional sense) functors between homotopy categories: one gets 
$$
ho_F:Ho_\cC\leftrightarrows Ho_\cD:ho_G,
$$ 
defined by applying $F$, resp. $G$ to fibrant-cofibrant representatives (in fact, it is sufficient to apply $F$ to a fibrant representative and $G$ to a cofibrant representative to get the correct functor on homotopy categories). These should be thought of as the left (resp. right) derived functors of $F$ (resp. $G$). A Quillen adjunction is called a  \emph{Quillen equivalence} if $ho_F$ (equivalently, $ho_G$) is an equivalence on homotopy categories.

The primordial Quillen adjunction, and one that will be important in this paper, is the adjunction 
$$
C_*:\top\leftrightarrows C_+(\Z):|\cdot |.
$$ 
Here we define $C_*(X)$ for $X\in\top$ to be the complex of singular chains on $X$. Its adjoint $|\cdot |$ is given by taking the geometric realization of an associated simplicial set. This adjunction can be written as a composition 
\[
\Top \leftrightarrows \t{SSet} \leftrightarrows \t{SAb} \leftrightarrows C_+(\Z). 
\]
The first one associates to a topological space its singular simplicial set, and to a simplicial set its geometric realization, and is a Quillen equivalence. The second one is the free-forgetful adjunction between simplicial sets and simplicial Abelian groups, which is not a Quillen equivalence. The third one is the Dold-Kan correspondence, and is an equivalence of categories. See~\cite[\S II.3]{quillen}, \cite[Cor.~3.2.15, Th.~3.4.4]{ginot} and~\cite[\S8.4]{weibel}.

\subsection{Homotopy (co)limits}  \label{sec:hocolim}
Our sources for this section are Ginot~\cite[\S2.5]{ginot}, Dwyer and Spalinski~\cite{dwyer-spalinski}, and Hirschhorn~\cite[Chapter~13]{hirschhorn}.
  
Suppose $\cC$ is a model category which is cocomplete, i.e., it has all small colimits. Let $J$ be a small ``diagram'' category, which we are interested in mapping to $\cC$. The functor category
\[\cC^J: = \fun(J, \cC)\]
inherits a natural notion of \emph{weak equivalence}: we say that a natural transformation $F\to G$ of functors $F, G: J\to \cC$ is a weak equivalence if $F(j)\to G(j)$ is such for each object $j\in J$. There are several natural model structures on the diagram category, one of which is the \emph{projective} model structure, with fibrations determined objectwise on a map of diagrams. If $X$ is an object of $\cC$, there is a constant diagram $\underline X$ with every object of $J$ sent to $X$ and every arrow sent to the identity morphism of $X$. This determines a functor $\t{const}: \cC\to \cC^J$. Its left adjoint is by definition the colimit functor:
\footnote{The right adjoint is the limit functor, which is defined when $J$-indexed limits exist, e.g., if $\cC$ is complete.} 
\[
\t{colim}:\cC^J\leftrightarrows \cC : \t{const}.
\]
Assume now that $J$ is given by a poset (more generally, assume $J$ to be \emph{very small} in the sense of~\cite[\S10.13 sqq.]{dwyer-spalinski} or~\cite[D\'efinition~2.5.11 sqq.]{ginot}). The projective structure defines in this case a model category structure on $\cC^J$~\cite{dwyer-spalinski,ginot}. The above adjunction is a Quillen adjunction, and thus induces a functor of associated homotopy categories, called the \emph{homotopy colimit functor}, written 
\[
\operatorname{hocolim}:Ho_{\cC^J}\to Ho_{\cC}.
\] 

We will be primarily interested in calculating homotopy pushouts, i.e. homotopy colimits in the functor category $\cC^J$ with $\cC$ a model category and $J$ the poset 
$$
a \longleftarrow b \longrightarrow c. 
$$
Given an object $X$ in $\cC^J$, the homotopy colimit $\operatorname{hocolim}(X)$ is by definition isomorphic to $\operatorname{colim}(X')$, where $X'$ is a cofibrant object in $\cC^J$ weakly equivalent to $X$. It is shown in~\cite[Proposition~10.6]{dwyer-spalinski} that a pushout diagram $X'$ is cofibrant in the projective model structure if and only if  the maps $X'(a)\longleftarrow X'(b) \longrightarrow X'(c)$ are cofibrations and $X'(b)$ is cofibrant. (Hence the objects $X'(a)$ and $X'(c)$ are also cofibrant.) 

The homotopy colimit of a diagram is well-defined up to equivalence, but giving an explicit model depends on the choice of cofibrant resolution of the diagram.\footnote{When passing from the homotopy category to the richer $\infty$-category language, the category of such choices is contractible, and thus homotopy colimits are unique up to homotopy in a strong sense.}

In certain situations it is possible to compute the homotopy pushout with fewer cofibrant replacements. The next result is stated in~\cite{lurie-htt} as Proposition~A.2.4.4.(i). We will use it in the proof of Proposition~\ref{prop:hocolim2-operads}.
\begin{lemma} \label{lem:hocolim2}
Let $\cC$ be a model category. Given a diagram 
$$A\longleftarrow B\hookrightarrow C$$ 
with $B\hookrightarrow C$ a cofibration and $A, B$ (and hence $C$) cofibrant, we have 
$$
\operatorname{hocolim}(A\longleftarrow B\hookrightarrow C) = \operatorname{colim}(A\longleftarrow B\hookrightarrow C).
$$
\end{lemma}

\begin{proof}
Let $B\hookrightarrow A' \xtwoheadrightarrow{\sim} A$ be a factorization of the map $B\to A$ into a cofibration and a trivial fibration. The diagram $A'\hookleftarrow B\hookrightarrow C$ is then a cofibrant replacement of the initial diagram (see~\S\ref{sec:hocolim})
$$
\xymatrix{
A' \ar@{->>}[d]^\sim & B \ar@{=}[d] \ar@{_(->}[l] \ar@{^(->}[r] & C \ar@{=}[d] \\
A & B \ar[l] \ar@{^(->}[r] & C ,
}
$$
and therefore $\operatorname{hocolim}(A\longleftarrow B\hookrightarrow C) = \operatorname{colim}(A'\hookleftarrow B\hookrightarrow C)$, also denoted $A'\sqcup_B C$. 

We claim that the canonical map $A'\sqcup_B C\to A\sqcup_B C$ is a weak equivalence. This is seen by considering the two pushout squares 
$$
\xymatrix{
B \ar@{^(->}[r] \ar@{^(->}[d] & C \ar@{^(->}[d] \\
A' \ar@{^(->}[r] \ar@{->>}[d]^\sim & A'\sqcup_B C \ar[d]^\sim \\ 
A \ar@{^(->}[r] & A\sqcup_B C .
}
$$
By assumption the maps $B\hookrightarrow A'$ and $B\hookrightarrow C$ are cofibrations. Since cofibrations are stable under pushout (\cite[2.1.12]{ginot}), the maps $A'\hookrightarrow A'\sqcup_B C$ and $C\hookrightarrow A'\sqcup_B C$ are cofibrations. By assumption the map $A'\xtwoheadrightarrow{\sim} A$ is a weak equivalence between cofibrant objects, and a result of Reedy (\cite[Proposition~13.1.2]{hirschhorn}) states that the pushout of a weak equivalence \emph{between cofibrant objects} along a 
cofibration is a weak equivalence. Therefore the map $A'\sqcup_B C\xrightarrow{\sim} A\sqcup_B C$ is a weak equivalence.  
\end{proof}

Call a model category \emph{left proper} if weak equivalences are preserved by pushouts along cofibrations, \emph{right proper} if weak equivalences are preserved by pullbacks along fibrations, and \emph{proper} if it is both left proper and right proper. The Quillen model category structures on $\t{SSet}$ and $\t{Top}$ are proper~\cite[Theorems~13.1.10 and~13.1.13]{hirschhorn}. 

\begin{lemma} \label{lem:left-proper-pushout} Let $\cC$ be a left proper model category. Given a diagram 
$$A\longleftarrow B\hookrightarrow C$$ 
with $B\hookrightarrow C$ a cofibration, we have 
$$
\operatorname{hocolim}(A\longleftarrow B\hookrightarrow C) = \operatorname{colim}(A\longleftarrow B\hookrightarrow C).
$$\qed
\end{lemma}
As a consequence, in a left proper model category the homotopy pushout of a diagram $A\longleftarrow B \longrightarrow C$ is weakly equivalent to the ordinary pushout of the diagram obtained by replacing one arrow by a cofibration. Lemma~\ref{lem:left-proper-pushout} is stated in~\cite[Proposition~A.2.4.4.(ii)]{lurie-htt} and proved in~\cite[Proposition~5.4]{nlab:proper_model_category}.

The previous discussion of homotopy colimits and homotopy pushouts has a dual counterpart for homotopy limits and homotopy pullbacks. The previous results hold true for homotopy pullbacks by reversing the direction of the arrows and exchanging cofibrations and left properness into fibrations and right properness. We will use this in the discussion of homotopy fibers of maps of simplicial sets in Appendix~\ref{sec:reminders-groupoids-SSet}. The dual of Lemma~\ref{lem:left-proper-pushout} is stated and proved in~\cite[Proposition~13.3.7]{hirschhorn}.

\section{The Berger-Moerdijk Model Structure for Operads}  \label{sec:model_category_operads}

\subsection{Existence of model structure} \label{sec:BM}
Suppose that $\cC$ is a symmetric monoidal category with weak equivalences. Then we say that a map of operads $O\to O'$ in $\cC$ is a {\it weak equivalence} if it is so objectwise, i.e., if $O_n\to O'_n$ is a weak equivalence for each $n$. Berger and Moerdijk~\cite{berger_moerdijk_model} show that if $\cC$ is a model category satisfying certain additional conditions, then this notion of weak equivalence is part of a model category structure on operads in $\cC,$ for which the fibrations $O\to O'$ are objectwise fibrations. In particular, they prove the following result. 
\begin{theorem}[{\cite[Theorem~3.2]{berger_moerdijk_model}}]
If $\cC$ is a \emph{cartesian+} closed symmetric monoidal model category, then the category of operads in $\cC$ has a model structure with weak equivalences and fibrations determined levelwise. 
\end{theorem}

We have not spelled out the meaning of ``cartesian+''. This is a shorthand notation for cartesian category satisfying some additional properties (cofibrantly generated with cofibrant terminal object and admitting symmetric monoidal fibrant replacement functor), cf. the assumptions of Theorem~3.2 in~\cite{berger_moerdijk_model}. For our purposes it suffices to record that this  holds for all three model structures that we consider on $\Top$.

\subsection{$W$-construction and cofibrant replacement} \label{sec:W} 

\qquad 

The \emph{$W$-construction} for operads plays the role of the familiar bar resolution for algebras. Our references here are Vogt~\cite{vogt_cofibrant} and Berger-Moerdijk~\cite{berger_moerdijk_bv}. We refer to~\S\ref{sec:free_operad} for notation concerning the definition of the free operad associated to a graded object. 

Given a topological operad $O$, recall that we denote $O_*$ the graded topological space $O_*=(O_1,O_2,\dots)$. We define a new operad $W(O)$ out of $\t{Free}(O_*)$ as follows. For each $n\ge 1$ we define 
\[
W(O)_n = \coprod_{[\tau]\in Tree_n} O^{[\tau]} \times [0,1]^{\mathrm{Edge}_{[\tau]}} /\sim_W,
\]
for a certain equivalence relation $\sim_W$. Here  $O^{[\tau]}\times [0,1]^{\mathrm{Edge}_{[\tau]}}$ is a notation for the quotient of
$\sqcup _{\tau\in [\tau]} O^\tau\times [0,1]^{\mathrm{Edge}_\tau}$, where $\tau\in\t{Tree}_n$ ranges over the elements of the equivalence class $[\tau]\in Tree_n$, by the equivalence relation given by isomorphisms of labeled trees, which act on the first factor as in~\S\ref{sec:free_operad_functor} and which act on the second factor via their action on the sets of edges of trees. Thus $O^{[\tau]} \times [0,1]^{\mathrm{Edge}_{[\tau]}}$ should be interpreted as the $[\tau]$-component of $\t{Free}(O_*)$, which consists of all possible labelings of the vertices $v$ of a tree $\tau$ by elements of $O_{|\mathrm{Child}(v)|}$, with the additional data of a length in $[0,1]$ for each internal edge. The equivalence relation $\sim_W$ consists simply in identifying two vertices $v$, $w$ which are connected by an edge of length $0$, and replacing their corresponding labels, which are elements of $O_{|\mathrm{Child}(v)|}$ and $O_{|\mathrm{Child}(w)|}$, by their composition in $O$ which is an element of $O_{|\mathrm{Child}(v)|+|\mathrm{Child}(w)|-1}$. 

Loosely speaking, {\it $W(O)$ is obtained from $\t{Free}(O_*)$ by giving lengths to internal edges of trees and merging vertices according to the composition rules in $O$ when the connecting edges acquire length zero.} The composition rule in $W(O)$ is inherited from that of $\t{Free}(O_*)$, with the convention that each new internal edge which results from a composition by gluing two half-edges is attributed length $1$. 

Given a point $o\in W(O)_n$, we obtain a point of $O_n$ by composing the operations in the corresponding tree. This results in a functorial map of operads $W(O)\to O$ which is (essentially by construction) a homotopy equivalence, see \cite[Theorem 5.1]{berger_moerdijk_bv}.

The $W$-construction is useful for replacing maps of operads by cofibrations. Namely, we have the following theorem. 
\begin{theorem}[{\cite[Proposition~6.6]{berger_moerdijk_bv}}] \label{thm:cofibration_operads} 
If $O\to O'$ is a map of operads which is a cofibration on the level of $\fS$-equivariant graded spaces, then $W(O)\to W(O')$ is a cofibration.
\end{theorem}

A $\fS$-cofibration is understood levelwise with respect to the action of the symmetric groups $\fS_n$. While not giving the general definition, it will be enough for our purposes to record that a $\fS_n$-equivariant map $f:A\to B$ between $\fS_n$-spaces is a $\fS_n$-cofibration whenever the underlying non-equivariant map is a cofibration and the $\fS_n$-action on $A$ and $B$ is free. By convention the $\fS_n$-action on the empty set is free.

\begin{corollary}\label{cor:W-cofibrant} Let $O$ be an operad in $\Top$ and assume that each space $O_n$ is homotopy equivalent to a CW-complex and the action of $\fS_n$ on $O_n$ is free. Then $W(O)$ is a cofibrant replacement in the mixed model category structure. 
\end{corollary}

\begin{proof}
The conditions guarantee that $\emptyset\to O$ is a cofibration of $\fS$-equivariant graded spaces. By Theorem~\ref{thm:cofibration_operads}, the map $\emptyset=W(\emptyset)\to W(O)$ is a cofibration.
\end{proof}
 
  In the statement of the next corollary we use 
the Berger-Moerdijk model structure on $\t{Op}_{\Top}$ (\S\ref{sec:BM}) and the projective model structure on $\t{Op}_{\Top}^J$ with $J=\{a\longleftarrow b\longrightarrow c\}$ (\S\ref{sec:hocolim}).  

\begin{corollary} \label{cor:homotopy-pushout}
Let $\Top$ be endowed with 
the mixed model structure. Let 
$$
O'\longleftarrow O \longrightarrow O''
$$
be a diagram of topological operads which are levelwise homotopy equivalent to CW-complexes and which carry levelwise free $\fS_n$-actions. Assume further that each map is a levelwise cofibration.  
The homotopy colimit is computed as the colimit of the diagram 
$$
W(O')\longleftarrow W(O) \longrightarrow W(O'').
$$
\end{corollary}

\begin{proof}
The conditions imply that the maps $O\to O'$ and $O\to O"$ are cofibrations of $\fS$-equivariant graded spaces, so that $W(O)\to W(O')$ and $W(O)\to W(O")$ are cofibrations (Theorem~\ref{thm:cofibration_operads}). By Corollary~\ref{cor:W-cofibrant} applied to $O$, the operad $W(O)$ is cofibrant. By the discussion in~\S\ref{sec:hocolim}, the diagram $W(O')\longleftarrow W(O) \longrightarrow W(O'')$ is a cofibrant replacement of the diagram $O'\longleftarrow O \longrightarrow O''$ and therefore $\operatorname{hocolim}(O'\longleftarrow O \longrightarrow O'')=\operatorname{colim}(W(O')\longleftarrow W(O) \longrightarrow W(O''))$.
\end{proof}

\begin{proposition} \label{prop:hocolim2-operads}
Let $\Top$ be endowed with 
the mixed model structure. Let 
$$
O'\longleftarrow O \longrightarrow O''
$$
be a diagram of topological operads which are levelwise homotopy equivalent to CW-complexes and which carry levelwise free $\fS_n$-actions. Assume that the map $O\to O''$ is a levelwise cofibration.  
The homotopy colimit is computed as the colimit of the diagram 
$$
W(O')\longleftarrow W(O) \longrightarrow W(O'').
$$
\end{proposition}

\begin{proof} Since the homotopy colimit is defined in the homotopy category we have 
$$
\operatorname{hocolim}(O'\longleftarrow O \longrightarrow O'') = \operatorname{hocolim}(W(O')\longleftarrow W(O) \longrightarrow W(O'')). 
$$
Our assumptions ensure that the map $W(O)\to W(O'')$ is a cofibration and $W(O)$, $W(O')$ are cofibrant. By Lemma~\ref{lem:hocolim2} we then have 
\begin{align*}
\operatorname{hocolim}(W(O')\longleftarrow W(O) & \longrightarrow W(O'')) \\
& =\operatorname{colim}(W(O')\longleftarrow W(O) \longrightarrow W(O'')).
\end{align*}
\end{proof}

\begin{remark}
In the previous statements, the assumption that the operads are levelwise homotopy equivalent to CW-complexes can be dropped provided one uses the Str\o m model structure on $\Top$.
\end{remark}

\section{Proof of the main theorem}  \label{sec:proof_main}

\subsection{Homotopy colimits}  \label{sec:proof_main_hocolim}
Let us consider the map of $\fS$-equivariant spaces $\widetilde{\t{Ann}}\to \widetilde{\t{Fr}_\partial}$. This is a map of free $\fS$-spaces 
which at the level of the underlying topological spaces is an embedding of a connected component. Moreover, the remaining connected components of $\widetilde{\t{Fr}_\partial}$ are homotopy equivalent to CW-complexes, hence are cofibrant in the mixed model structure (and also in the Str\o m model structure). This map is therefore a levelwise cofibration in the mixed model structure (and in the Str\o m model structure).

Consider now the map of $\fS$-equivariant spaces $\widetilde{\t{Ann}}\to \widetilde{\t{NodAnn}}$. This is a map of free $\fS$-spaces which are homotopy equivalent to CW-complexes.

By Proposition~\ref{prop:hocolim2-operads}, the homotopy colimit with respect to these model structures
$$\t{hocolim}(\widetilde{\t{NodAnn}}\leftarrow\widetilde{\t{Ann}}\rightarrow \widetilde{\mathrm{Fr}}_\partial)$$
is computed by the operad colimit 
\[\t{colim}\big(W(\widetilde{\t{NodAnn}})\leftarrow W(\widetilde{\t{Ann}})\rightarrow W(\widetilde{\mathrm{Fr}}_\partial)\big).\]

Given a labeled tree of operations $\tau$ with $n$ inputs, recall from Definition~\ref{defi:split-tilde} the moduli space $\widetilde{\t{Split}}_\tau$ of framed split surfaces with possibly coinciding seams and dual graph $\tau$. Let ${\t{Edge}_\tau}$ be the set of edges of $\tau$. We glue together the spaces $\widetilde{\t{Split}}_\tau\times [0,1]^{\t{Edge}_\tau}$, where we view the coordinate $[0,1]^{\{e\}}$ for $e$ an edge as being attached to the seam $S_e$, into the full \emph{Humpty-Dumpty space}
\[
\widetilde{\t{HD}}^\t{tree} = \frac{\bigsqcup_\tau \widetilde{\t{Split}}_\tau\times [0,1]^{\t{Edge}_\tau}}{\sim},
\]
where $\sim$ is the equivalence relation 
\begin{equation*}
\begin{split}
\big(\Sigma, (S_1,\dots, S_E), & (t_1, \dots, t_e = 0, \dots, t_E)\big)  \sim \\
&  \big(\Sigma, (S_1, \dots, \hat{S}_e, \dots, S_E), (t_1,\dots, \hat{t_e}, \dots t_E)\big).
\end{split}
\end{equation*}
Similarly, we define 
\[
\t{HD}^\t{tree} = \frac{\bigsqcup_\tau \t{Split}_\tau\times [0,1]^{\t{Edge}_\tau}}{\sim},
\]
with respect to the same equivalence relation.

The spaces $\widetilde{\t{HD}}^\t{tree}$ and $\t{HD}^\t{tree}$ classify split surfaces with additional simplicial parameters that allow us to continuously ``put the curve back together'', hence the term ``Humpty-Dumpty''.

\begin{lemma} \label{lem:HDtree}
  We have homeomorphisms 
  $$
  \widetilde{\t{HD}}^\t{tree}\simeq W\big(\widetilde{\mathrm{Fr}}_\partial\big) 
  \quad \mbox{and} \quad  
  \t{HD}^\t{tree}\break\simeq W\big(\mathrm{Fr}_\partial\big).
  $$ 
\end{lemma}

\begin{proof}
This is a direct consequence of Lemma~\ref{lem:from_free_to_seamed_unital} and the description of the $W$-construction in~\S\ref{sec:W}. 
\end{proof}

Similarly, we define the \emph{moduli of nodal Humpty-Dumpty surfaces} as 
\[
\widetilde{\t{NodHD}}^\t{tree} = \frac{\bigsqcup_\tau \widetilde{\t{NodSplit}}^\t{tree}_\tau\times [0,1]^{\t{Edge}_\tau}}{\sim},
\]
and 
\[
\t{NodHD}^\t{tree} = \frac{\bigsqcup_\tau \t{NodSplit}^\t{tree}_\tau\times [0,1]^{\t{Edge}_\tau}}{\sim},
\]
where $\sim$ is the same equivalence relation as above.

\begin{remark}
Using the formalism of Segal operads developed in Appendix~\ref{sec:dendroidal}, it is the case that $\widetilde{\t{NodHD}}^\t{tree}$, resp. $\t{NodHD}^\t{tree}$, is a Segal operad which is equivalent to the $W$-construction applied to the Segal operad in topological moduli problems $\widetilde{\t{NodFr}}_\partial^\t{tree}$,  resp. $\t{NodFr}_\partial^\t{tree}$. 
We will not need this result, however, as we will be most interested in a stabilizer-free subspace of $\widetilde{\t{NodHD}}^\t{tree}$, resp. $\t{NodFr}_\partial^\t{tree}$.
\end{remark}

\begin{definition}
We define 
\[
\widetilde{\t{NodHD}}^\t{tree}_{\t{protected}}, \mbox{ and respectively} \quad \t{NodHD}^\t{tree}_{\t{protected}},
\] 
to be the space of tuples $(\Sigma, \{S_e\}, \{t_e\})$ as in $\widetilde{\t{NodHD}}^\t{tree}$, respectively $\t{NodHD}^\t{tree}$, such that each node of $\Sigma$ is surrounded on both sides by seams which can be contracted to the node and have weight $1$. 
\end{definition}

\begin{lemma} \label{lem:from-W-to-NodHD} We have canonical isomorphisms 
\[\t{colim}\big(W(\widetilde{\t{NodAnn}})\leftarrow W(\widetilde{\t{Ann}})\rightarrow W(\widetilde{\mathrm{Fr}}_\partial)\big)\cong \widetilde{\t{NodHD}}^\t{tree}_{\t{protected}},\]  
and
\[
\t{colim}\big(W(\t{NodAnn})\leftarrow W(\t{Ann})\rightarrow W(\mathrm{Fr}_\partial)\big)\cong \t{NodHD}^\t{tree}_{\t{protected}}.
\]   
\end{lemma}

\begin{proof} We give details only for the proof of the first isomorphism since the proof of the second one is verbatim the same after removing the symbol $\widetilde{\quad}$ everywhere. 

Using the equivalence $W\big(\widetilde{\mathrm{Fr}}_\partial\big)\simeq \widetilde{\t{HD}}^\t{tree}$ from Lemma~\ref{lem:HDtree} the statement becomes 
\begin{equation} \label{eq:NodHDtreeprotected}
\t{colim}\big(W(\widetilde{\t{NodAnn}})\leftarrow W(\widetilde{\t{Ann}})\rightarrow \widetilde{\t{HD}}^\t{tree}\big)\cong \widetilde{\t{NodHD}}^\t{tree}_{\t{protected}}.
\end{equation}
The operad $W(\widetilde{\t{Ann}})$ can be described as consisting of standard annuli with parametrized boundary and parametrized seams given by concentric circles, with simplicial parameters in $[0,1]$ attached to them. Seams with simplicial parameter equal to $0$ can be erased or added. The operad $W(\widetilde{\t{NodAnn}})$ can be described in a similar way, allowing also nodal annuli.

We have a map 
$$
W(\widetilde{\t{NodAnn}}) \sqcup_{W(\widetilde{\t{Ann}})}  \widetilde{\t{HD}}^\t{tree} \to \widetilde{\t{NodHD}}^\t{tree}_{\t{protected}}
$$
given by gluing along their boundaries the framed surfaces and nodal annuli by which we decorate the vertices of a tree, while keeping track of the boundary curves by interpreting them as additional seams with weight $1$. Reasoning as in the proof of the Geometric Pushout Theorem~\ref{thm:geomthm} we see that this map descends to a homeomorphism after imposing the pushout relations on the free product.   
\end{proof}


\subsection{Proof of the main theorem in genus 0}  \label{sec:DM-genus0}
In this section we prove our main theorem in genus 0, recovering thus by a different method the theorem of Drummond-Cole~\cite{drummond2013homotopically} mentioned in the Introduction. We denote $\overline{\M}_{0,*}$ the genus 0 Deligne-Mumford-Knudsen operad. Since genus $0$ curves with at least 3 marked points have no automorphisms, there are no stacky phenomena to take into account. We will use the subscript notation $\t{NodFr}^{\t{tree}}_{g=0}\subset \t{NodFr}^{\t{tree}}$ etc. to denote the genus 0 suboperads of the various operads that we use. 

\begin{lemma} \label{lem:from-DM-to-NodFr}
  We have a homotopy equivalence of operads 
  \[\funnel:
  \overline{\M}_{0,*}\stackrel\simeq\longrightarrow \widetilde{\t{NodFr}}^{\t{tree}}_{\partial,g=0}.
  \]
\end{lemma}
\begin{proof}
Let $\D$ be the standard unit disk, framed with the standard boundary parametrization $\theta\mapsto \exp(2\pi i\theta)$ for $\theta\in\R/\Z$ and let $0\in \D$ be the origin. Let $\overline{\D}$ be the disk framed with the reverse boundary parametrization $\theta \mapsto \exp(-2\pi i\theta)$. Let $A_\alpha$ be the standard annulus of modulus $\alpha\in (0,\infty)$ framed with the standard boundary parametrizations. We then have 
  \[
  \lim_{\alpha\to \infty} A_\alpha = \D\cup_0 \overline{\D}
  \] 
  in $\t{NodAnn}$, compatibly with boundary parametrizations. 
  
  Given a marked nodal surface $X\in \overline{\M}_{0,*}$, write 
  \[
  \funnel(X)
  \] 
  for the framed nodal surface obtained by gluing (at $0\in \D$) a copy of $\D$ at every input marked point of $X$ and a copy of $\overline{\D}$ at the output marked point of $X$. See Figure~\ref{fig:fr-genus0}. 
  
\begin{figure} [ht]
\centering
\input{fr-genus0.pstex_t}
\caption{By attaching disks at marked points one turns a surface into a framed nodal surface.}
\label{fig:fr-genus0}
\end{figure}
  
  By a simple stabilization argument we see that 
  \[
  \funnel:\overline{\M}_{0,*}\to \widetilde{\t{NodFr}}^{\t{tree}}_{\partial,g=0}
  \] 
  is a map of topological operads. On the other hand we see that $Fr_n$ is the embedding of a homotopy retract for each arity $n\ge 0$. Indeed, let 
  \[
  \t{\bf cap}:\widetilde{\t{NodFr}}^{\t{tree}}_{\partial,g=0}\to \t{DM}^{\t{tree}}
  \] 
  be the map (now of $\fS$--graded spaces, not operads) which assigns to a nodal surface with boundary the surface with marked points obtained by gluing a copy of $\overline{\D}$ at each input and a copy of $\D$ at the output, and marking all images of $0\in \overline{\D}$, respectively $\D$. See Figure~\ref{fig:cap}. 
  
\begin{figure} [ht]
\centering
\input{cap.pstex_t}
\caption{By attaching caps along the boundary and stabilizing one turns a framed nodal surface into a nodal surface with marked points.}
\label{fig:cap}
\end{figure}

  Then it is clear that 
  \[
  \t{\bf cap}\circ \funnel = \one_{\t{DM}}.
  \] 
  On the other hand, consider the maps 
  \[
  \t{\bf stretch}_\alpha: \widetilde{\t{NodFr}}^{\t{tree}}_{\partial,g=0}\to \widetilde{\t{NodFr}}^{\t{tree}}_{\partial,g=0},\qquad \alpha\in[0,\infty]
  \] 
  defined by gluing the standard annulus $A_\alpha$ at every input and output of a framed nodal surface. This defines a homotopy equal to the identity map at $\alpha = 0$ and equal to $\funnel\circ \t{\bf cap}$ at $\alpha = +\infty$, which proves the homotopy retract property.
\end{proof}

The following result will complete the proof of our main theorem. 
\begin{proposition} \label{prop:from-protected-to-unprotected} 
There is a weak equivalence of operads
\[
\boldsymbol{\pi}: \widetilde{\t{NodHD}}^\t{tree}_{\t{protected},g=0} \longrightarrow
\widetilde{\t{NodFr}}^{\t{tree}}_{\partial,g=0},
\]
where $\boldsymbol{\pi}$ forgets the simplex parameters and glues along the seams. 
\end{proposition}

\begin{proof}
It suffices to show that 
\begin{equation} \label{eq:pi-trivial-Serre-fib} 
\boldsymbol{\pi}: \t{NodHD}^\t{tree}_{\t{protected},g=0} \longrightarrow
\t{NodFr}^{\t{tree}}_{\partial,g=0}
\end{equation} 
is a trivial Serre fibration. Indeed, we can ignore thickness-zero curves because including them does not change the weak homotopy type. We do this in two steps: firstly we show that the fiber of~\eqref{eq:pi-trivial-Serre-fib} is contractible, and secondly we show that~\eqref{eq:pi-trivial-Serre-fib} is a Serre fibration.

\smallskip
{\it Step~1. We show that the fiber of~\eqref{eq:pi-trivial-Serre-fib} is contractible.} 

Let $\Sigma\in \t{NodFr}^{\t{tree}}_{\partial,g=0}$ be a surface, and let $\t{HD}_\Sigma: = \boldsymbol{\pi}^{-1}(\Sigma)$. 
Given mutually disjoint neighborhoods $\cV_i$, $i=1,\dots,k$ of the nodes $z_1,\dots,z_k$ of $\Sigma$, denote $\cV=\sqcup_{i=1}^k \cV_i$ and let $\t{HD}_\Sigma^\cV\subset \t{HD}_\Sigma$ be the subspace consisting of those tree-like split surface structures on $\Sigma$ whose seams lie away from $\overline{\cV}$. Since seams are not allowed to pass through nodes, the spaces $\t{HD}_\Sigma^\cV$ filter $\t{HD}_\Sigma$ by opens as $\cV$ runs over a neighborhood basis of the nodes of $\Sigma$. 

We consider now a specific neighborhood basis of the nodes: we choose a $1$-parameter smooth family of open neighborhoods of the nodes, denoted $\cV_i^\eps$, $0<\eps\le 1$, such that the $\overline{\cV_i^\eps}$ are mutually disjoint for every $\eps$, each $\overline{\cV_i^\eps}$ is analytically diffeomorphic to a nodal annulus, $\overline{\cV_i^\eps}\subset \cV_i^{\eps'}$ for $\eps<\eps'$ and $\bigcap_{\eps\in(0,1]} \cV_i^\eps=\{z_i\}$.

Let $\cV^\eps=\sqcup_{i=1}^k \cV_i^\eps$, let $\t{HD}_\Sigma^\eps=\t{HD}_\Sigma^{\cV^\eps}$, and let 
\[\t{HD}_\Sigma^I = \bigsqcup_{\eps\in (0,1]} \t{HD}_\Sigma^\eps\subset I\times \t{HD}_\Sigma\]
with $I=(0,1]$. The map 
\[
\t{HD}_\Sigma^I\to \t{HD}_\Sigma
\] 
is a homotopy equivalence because it is a fibration and all its fibers are nonempty intervals (with one open endpoint $0$). We are thus left to prove that $\t{HD}_\Sigma^I$ is contractible. 

Consider the canonical map 
\[
\t{\bf gap}:\t{HD}_\Sigma^I\to \t{HD}_\Sigma^I
\] 
given by the tautological inclusions $\t{HD}_\Sigma^\eps\hookrightarrow \t{HD}_\Sigma^{\eps/2}$, $\eps\in(0,1]$, i.e., 
\[
(w, \eps)\mapsto (w, \eps/2), \qquad w\in \t{HD}_\Sigma^\eps. 
\] 
This map is well-defined because the filtration is decreasing. It is a homotopy equivalence because it is homotopic to the identity via the family of maps $\t{\bf gap}_t$, $t\in[0,1]$ given by the tautological inclusions $\t{HD}_\Sigma^\eps\hookrightarrow \t{HD}_\Sigma^{\eps-t\eps/2}$, $(w,\eps)\mapsto (w,\eps-t\eps/2)$.

Choose a continuous family of seams $S_\eps$, $\eps\in(0,1/2]$ contained in $\cV^{2\eps}\setminus \overline{\cV^\eps}$, one on each side of a node. (We can take $S_\eps$ to be given by continuously varying parametrizations of the boundary of $\cV^{3\eps/2}$.) Define the continuous map 
\[
\t{\bf protect}: (0,1/2]\to \t{HD}_\Sigma^I
\] 
which takes $\eps$ to the pair $(\Sigma_\eps,\eps)$, where $\Sigma_\eps=(\Sigma, S_\eps)$ is the protected split surface with seams $S_\eps$ of weight $1$.

We claim that $\t{\bf gap}: \t{HD}_\Sigma^I\to \t{HD}_\Sigma^I$ is homotopic to the map 
$$
\bm{\tau}: (w, \eps)\mapsto \t{\bf protect}(\eps/2).
$$ 
The homotopy is constructed in two steps: starting from $\t{\bf gap}$, we first put in the seams 
$S_{\eps/2}$ with weight continuously changing from $0$ to $1$ (this is allowed because we are in the image of $\t{\bf gap}$ and there are no other seams at distance $\le \eps$ from a node), and then continuously reduce all the other weights to zero (this is allowed because the presence of the $S_{\eps/2}$ with weight $1$ guarantees that the nodes remain protected). At the endpoint of this homotopy we read the map $\bm{\tau}$. 

Since $\bm{\tau}$ factors through an interval, it is homotopic to a constant. Thus $\t{\bf gap}$ is at the same time homotopic to a constant and a homotopy equivalence, which implies that the space $\t{HD}_\Sigma^I$ is contractible. This finishes the proof of the Step~1.

\smallskip
{\it Step~2. We show that~\eqref{eq:pi-trivial-Serre-fib} is a Serre fibration.} 

It is enough to check the Serre fibration property locally on the base. Our proof is based on the existence of a continuously varying thin-thick decomposition in the neighborhood of any fixed surface in $\t{NodFr}^{\t{tree}}_{\partial,g=0}$. In the following we designate by $\Sigma$ both a point in the moduli space $\t{NodFr}^{\t{tree}}_{\partial,g=0}$, and the fiber at $\Sigma$ of the universal curve over the moduli space. 

Given $\Sigma\in \t{NodFr}^{\t{tree}}_{\partial,g=0}$ with nodes $z_1,\dots,z_k$, there exists a neighborhood $\cU\subset \t{NodFr}^{\t{tree}}_{\partial,g=0}$ of $\Sigma$ and a family $\cV_i^\eps(\Sigma')$, $i=1,\dots,k$, $0<\eps\le 1$, $\Sigma'\in \cU$ of open sets $\cV_i^\eps(\Sigma')\subset \Sigma'$ (the ``thin" parts) such that:
\begin{itemize}
\item for each $i$ the family $\cV_i^\eps(\Sigma')$ is continuous in $\eps$ and $\Sigma'$. 
\item for $\Sigma',\eps$ fixed, the closures $\overline{\cV_i^\eps(\Sigma')}$, $i=1,\dots, k$ are disjoint. 
\item for $\Sigma',i$ fixed we have $\overline{\cV_i^\eps(\Sigma')}\subset \cV_i^{\eps'}(\Sigma')$ for $\eps<\eps'$. 
\item each $\cV_i^\eps(\Sigma')$ is analytically diffeomorphic to a (possibly nodal) annulus. If $\cV_i^\eps(\Sigma')$ is analytically diffeomorphic to a nodal annulus for some $\eps>0$, then it is analytically diffeomorphic to a nodal annulus for all $\eps>0$, and in this case $\cap_{\eps>0}\cV_i^\eps(\Sigma')=\{z'_i\}$, the common node of these nodal annuli. 
\end{itemize}
These properties follow from the fact that a sufficiently small neighborhood of $\Sigma$ in $\t{NodFr}^{\t{tree}}_{\partial,g=0}$ consists of surfaces $\Sigma'$ that are obtained from $\Sigma$ by resolving some of its nodes. 

Let $I=[0,1]$. Let $\varphi:I^n\to \cU$ be a continuous map together with a lift $\tilde \varphi_0:I^{n-1}\times \{0\}\to \t{NodHD}^\t{tree}_{\t{protected},g=0}$ of $\varphi|_{I^{n-1}\times\{0\}}$, i.e., $\boldsymbol{\pi}\circ\tilde\varphi_0 = \varphi|_{I^{n-1}\times\{0\}}$. We need to construct $\tilde \varphi:I^n\to \t{NodHD}^\t{tree}_{\t{protected},g=0}$ such that $\boldsymbol{\pi}\circ\tilde\varphi = \varphi$ and $\tilde\varphi|_{I^{n-1}\times\{0\}}=\tilde\varphi_0$. Denote $\cD_i\subset\cU$, $i=1,\dots,k$ the closed set (smooth of complex codimension 1) consisting of curves $\Sigma'$ such that $\cV_i^\eps(\Sigma')$ is analytically diffeomorphic to a nodal annulus.

\begin{remark} Loosely said, the construction of the lift $\tilde \varphi$ will be done in the following steps. Firstly, we insert small protecting seams close to the nodes and close to $I^{n-1}\times \{0\}$. Secondly, we extend the already present seams close to $I^{n-1}\times \{0\}$, and then erase them. Thirdly, we extend the protecting seams throughout the family as boundaries of the thin part. 
\end{remark} 

We now proceed to construct the lift. 

By compactness of $I^{n-1}$, because the tree-like split structures encoded by $\tilde\varphi_0$ are protected, and because a protecting seam at a node survives in a neighborhood of the curve even if the node gets resolved, there exist $\eps>0$ and, for each $i=1,\dots,k$, an open neighborhood $\cU_i\subset \cU$ of $\cD_i$, such that, for $t=(t_1,\dots,t_{n-1},0)\in I^{n-1}\times\{0\}$, if $\varphi(t)\in \cU_i$ then no seam of $\tilde\varphi_0(t)$ intersects $\cV_i^\eps(\varphi(t))$. 

Let $\delta>0$ be small. We make a first extension of $\tilde \varphi_0$ over $I^{n-1}\times[0,\delta]$ in two steps as follows. 
\begin{itemize} 
\item {\it Step~1}. Let $\rho_i:\cU_i\to[0,1]$ be smooth cutoff functions, supported away from $\cU\setminus \cU_i$
and equal to $1$ on $\cD_i$. Let $\rho:[0,1]\to [0,1]$, $\rho(s)=s/\delta$ for $s\in[0,\delta]$, $\rho(s)=1$ for $s\in[\delta,1]$. We insert for each $i$ two seams as boundaries of $\cV_i^{\eps/2}(\varphi(t))$ for $t\in I^{n-1}\times[0,\delta]$, with weight equal to $\rho(t_n) \prod_i \rho_i(\varphi(t))$. These will act as ``protecting" seams in the sequel. 

\item {\it Step~2}. We extend the other seams from $I^{n-1}\times\{0\}$ to $I^{n-1}\times[0,\delta]$ as follows. We choose a smooth trivialization of the family with fiber $\Sigma'$ over $\cU\setminus \cup_i\cU_i$. This induces a trivialization of the family with fiber $\Sigma'\setminus \cup_i \cV_i^\eps(\Sigma')$ over $\cU\setminus \cup_i\cU_i$, which we extend to a trivialization of the family with the same fiber $\Sigma'\setminus \cup_i \cV_i^\eps(\Sigma')$ over $\cU$. For every $t\in I^{n-1}$, every seam $S_{t,0}$ on $\tilde\varphi_0(t,0)$, and every $t_n\in[0,\delta]$, we induce via these trivializations a uniquely determined seam $S_{t,t_n}$ on $\varphi(t,t_n)$. These seams are embedded and smoothly parametrized with continuously varying parametrizations, and they admit continuously varying analytic parametrizations. After making a choice of such continuously varying analytic para\-metrizations, we obtain a lift $\tilde \varphi_{[0,\delta]}$ of $\varphi$ over $I^{n-1}\times [0,\delta]$ that extends $\tilde\varphi_0$. 
\end{itemize}

We now modify the previous lift $\tilde \varphi_{[0,\delta]}$ in two steps as follows.
\begin{itemize}
\item {\it Step~1}. We weigh each of the seams $S_{t,t_n}$ from the previous construction by the function $1-\rho(2t_n)$ for $t_n\in[0,\delta/2]$, and erase them for $t_n\in[\delta/2,\delta]$. 

\medskip 

As a result, for $t\in I^{n-1}\times [\delta/2,\delta]$ the seam structure is the following: for each $i$ we have exactly two seams that are the boundaries of $\cV_i^{\eps/2}(\varphi(t))$, with weight equal to $\rho(t_n) \prod_i \rho_i(\varphi(t))$.

\medskip 

\item {\it Step~2}.  For $t\in I^{n-1}\times [\delta/2,\delta]$ we modify the weight of these seams to $(1-s)\rho(t_n) \prod_i \rho_i(\varphi(t)) + s$, with $s=\frac{2}{\delta}t_n-1$.
\end{itemize}

At this point we have obtained a lift $\tilde\varphi|_{[0,\delta]}$ of $\varphi|_{I^{n-1}\times[0,\delta]}$ such that, for $t\in I^{n-1}\times\{\delta\}$, the seam structure over $\varphi(t)$ consists of the boundaries of the $\cV_i^{\eps/2}(\varphi(t))$, $i=1,\dots,k$ with weight $1$. This seam structure extends tautologically for all $t\in I^{n-1}\times [\delta,1]$ and provides a lift of $\varphi$.

\end{proof}

\begin{proof}[Proof of the Main Theorem~\ref{thm:mainintro} in genus 0]
We saw in~\S\ref{sec:proof_main_hocolim} that the homotopy colimit
$$\t{hocolim}(\widetilde{\t{NodAnn}}\leftarrow\widetilde{\t{Ann}}\rightarrow \widetilde{\mathrm{Fr}}_{\partial,g=0})$$
is computed by the operad colimit 
$$\t{colim}\big(W(\widetilde{\t{NodAnn}})\leftarrow W(\widetilde{\t{Ann}})\rightarrow W(\widetilde{\mathrm{Fr}}_{\partial,g=0})\big).$$
In view of Lemma~\ref{lem:from-W-to-NodHD}, Proposition~\ref{prop:from-protected-to-unprotected} and Lemma~\ref{lem:from-DM-to-NodFr} we obtain the sequence of weak equivalences
\begin{eqnarray}\label{eq:eq_sequence0}
\lefteqn{\t{colim}\big(W(\widetilde{\t{NodAnn}})\leftarrow W(\widetilde{\t{Ann}})\rightarrow W(\widetilde{\mathrm{Fr}}_{\partial,g=0})\big)} \nonumber \\
&\cong& \widetilde{\t{NodHD}}^\t{tree}_{\t{protected},g=0} \xrightarrow[\boldsymbol{\pi}]{\simeq}
\widetilde{\t{NodFr}}^{\t{tree}}_{\partial,g=0} \xleftarrow[\funnel]{\simeq} \overline{\M}_{0,*}.
\end{eqnarray}
This proves the theorem for the unital version of the involved operads. 

The proof carries over verbatim to the non-unital versions of our operads $\t{Ann}$ and $\t{Fr}_\partial$. Alternatively, one can use the fact that the forgetful functor from operads to non-unital operads commutes with homotopy pushouts. 
\end{proof}


\subsection{Proof of the main theorem in arbitrary genus}  \label{sec:DM}
In this section we prove our main theorem in arbitrary genus. The main new phenomenon compared to that of genus 0 is the presence of stabilizers. 

We develop in Appendix~\ref{sec:app_moduli} the convenient language of \emph{topological moduli problems (TMP)} to deal with stabilizers. A topological moduli problem is a contravariant functor $\Top^{op}\to \t{Gpd}$, and the key point is that the spaces of operations of $\t{DM}^{\t{tree}}$ and $\mathrm{NodFr}_\partial^{\t{tree}}$ have natural structures of topological moduli problems. On the other hand, any topological space $X$ can be viewed as a topological moduli problem via the mapping-in functor $S\mapsto \t{Map}_{\Top}(S,X)$, so that the language of topological moduli problems provides a convenient common framework for all the spaces of operations considered in this paper.  

The classical language of operads is not well adapted to deal with spaces of operations that are topological moduli problems. 
Indeed, in the classical language of operads the associativity relations would translate into equalities of functors, whereas in our situation we only have isomorphisms. A convenient language is that of Segal operads, which we introduce in Appendix~\ref{sec:dendroidal}. Roughly speaking, whereas an operad $O$ associates to every arity $n\ge 0$ a space of operations $O_n$, a Segal operad $\Lambda$ associates to every \emph{tree of operations} $\tau$, of arbitrary arity, a space of operations $\Lambda_\tau$. We explain in~\S\ref{sec:Segal-op-and-preop} the construction of the (tree-like) \emph{Deligne-Mumford Segal operad} $\Lambda \t{DM}^{\t{tree}}$ and of the \emph{Segal operad of nodal framed surfaces} $\Lambda \widetilde{\t{NodFr}}^\t{tree}_\p$ as Segal operads in topological moduli problems. 

With this terminology in place, the proof of the genus 0 case of our main theorem goes through with essentially only minor modifications. 

The next lemma is the higher genus counterpart of Lemma~\ref{lem:from-DM-to-NodFr}.

\begin{lemma} \label{lem:from-DM-to-NodFr-Segal}
  We have a weak equivalence of Segal operads  
  \[\funnel:
  \Lambda\t{DM}^{\t{tree}}\stackrel\simeq\longrightarrow \Lambda\widetilde{\t{NodFr}}^{\t{tree}}_\partial
  \]
  induced by applying component-wise the map $\funnel$ from Lemma~\ref{lem:from-DM-to-NodFr}. 
\end{lemma}
\begin{proof} 
Being a weak equivalence of Segal operads means being a homotopy equivalence on the level of simplicial chains for every space of operations, i.e., every object associated to a corolla, see Appendix~\ref{sec:dendroidal}. We thus need to show that 
$$
\funnel:\t{DM}^{\t{tree}}_n \to \widetilde{\t{NodFr}}^{\t{tree}}_{\partial,n}
$$
induces a homotopy equivalence on simplicial realizations. 

The construction from Lemma~\ref{lem:from-DM-to-NodFr} can be applied in order to define a left inverse $\t{\bf cap}$ for the map $\funnel$, which presents $\funnel$ as a retract, i.e., such that $\funnel\circ \t{\bf cap}$ homotopy retracts to the identity. It is therefore enough to see that $\t{\bf cap}$ induces a homotopy equivalence on simplicial realizations. 

By Proposition~\ref{prop:combifib} the map $\t{\bf cap}$ is a combinatorial fibration. The points of the fibers of $\t{\bf cap}$ over a given marked curve $X$ correspond to framed curves which give $X$ after gluing disks onto each boundary component and stabilizing. These fibers have no automorphisms and are representable by topological spaces (see Remark \ref{rmk:topology_on_fr}), hence $\t{\bf cap}$ is a fiberwise representable fibration. The proof of Lemma~\ref{lem:from-DM-to-NodFr} goes through verbatim to show that the fibers of $\t{\bf cap}$ can be simultaneously contracted, which implies that $\t{\bf cap}$ is a trivial weak fibration of topological moduli problems. The conclusion then follows from Lemma~\ref{lm:fib_heq}.
\end{proof}

The next proposition is the higher genus counterpart of Proposition~\ref{prop:from-protected-to-unprotected}.

\begin{proposition} \label{prop:from-protected-to-unprotected-Segal} 
There is a weak equivalence of Segal operads
\[
\boldsymbol{\pi}: \Lambda\widetilde{\t{NodHD}}^\t{tree}_{\t{protected}} \longrightarrow
\Lambda\widetilde{\t{NodFr}}^{\t{tree}}_\partial,
\]
where $\boldsymbol{\pi}$ forgets the simplex parameters and glues along the seams. 
\end{proposition}

\begin{proof}
Being a weak equivalence of Segal operads means being a homotopy equivalence on the level of simplicial chains for every space of operations, i.e., every object associated to a corolla, see Appendix~\ref{sec:dendroidal}. 

  We need to show that the map is a trivial fibration on the level of each space of operations in the sense of Definition~\ref{def:fibtriv}, from which the conclusion follows by Lemma~\ref{lm:fib_heq}. More precisely, we need to prove that, for every corolla, i.e., for every arity $n\ge 0$, the induced map of TMPs 
\[
\boldsymbol{\pi}:\MM \widetilde{\t{NodHD}}^\t{tree}_{\t{protected}, n}\to \widetilde{\t{NodFr}}_{\partial, n}^\t{tree}.
\] 
is a trivial weak fibration. 

By Proposition~\ref{prop:combifib}, this map is a combinatorial fibration. It is also fiberwise representable since fibers have no automorphisms and choices of seams on a given curve form a topological space. That the fibers have no automorphisms is seen as follows: given a family $\sigma$, a point in the fiber $\big(\widetilde{\t{NodHD}}^\t{tree}_{\t{protected}, n})_\sigma$ is given by a family of nodal framed curves with additional data (given by weighted seams), together with an isomorphism between this family and the family $\sigma$. Since the isomorphism is part of the data and the family $\sigma$ is fixed, the only possible automorphism is the identity.

We now prove the trivial weak fibration property. For the proof, let $\sigma\in \widetilde{\t{NodFr}}^{\t{tree}}_{\partial,n}(\Delta^m)$ be a family of curves parametrized by an $m$-dimensional simplex and write $\big(\widetilde{\t{NodHD}}^\t{tree}_{\t{protected}, n})_\sigma$ for the fiber of $\boldsymbol{\pi}$ at $\sigma$. While $\widetilde{\t{NodFr}}^{\t{tree}}_{\partial,n}$ is a topological moduli problem, the fiber has no stabilizers as seen above. 
In this situation the proof of contractibility from Proposition~\ref{prop:from-protected-to-unprotected} goes through verbatim if one works locally on the base $\Delta^m$. This implies contractibility over the whole of $\Delta^m$, and a fortiori also weak contractibility.  
\end{proof}

\begin{proof}[Proof of the Main Theorem~\ref{thm:mainintro} in arbitrary genus]
We saw in~\S\ref{sec:proof_main_hocolim} that the homotopy colimit
$$\t{hocolim}(\widetilde{\t{NodAnn}}\leftarrow\widetilde{\t{Ann}}\rightarrow \widetilde{\mathrm{Fr}}_{\partial})$$
is computed by the operad colimit 
$$\t{colim}\big(W(\widetilde{\t{NodAnn}})\leftarrow W(\widetilde{\t{Ann}})\rightarrow W(\widetilde{\mathrm{Fr}}_{\partial})\big).$$
In view of Lemmas~\ref{lem:from-W-to-NodHD}, \ref{lem:from-DM-to-NodFr-Segal} and Proposition~\ref{prop:from-protected-to-unprotected-Segal} we obtain the sequence of weak equivalences of Segal operads
\begin{eqnarray}\label{eq:eq_sequence}
\lefteqn{\t{colim}\big(W(\widetilde{\t{NodAnn}})\leftarrow W(\widetilde{\t{Ann}})\rightarrow W(\widetilde{\mathrm{Fr}}_{\partial})\big)} \nonumber \\
&\cong& \widetilde{\t{NodHD}}^\t{tree}_{\t{protected}} \xrightarrow[\boldsymbol{\pi}]{\simeq}
\Lambda\widetilde{\t{NodFr}}^{\t{tree}}_{\partial,g=0} \xleftarrow[\funnel]{\simeq} \Lambda \t{DM}^{\t{tree}}.
\end{eqnarray}
This proves the theorem for the unital version of the involved operads. 

The proof carries over verbatim to the non-unital versions of our operads $\t{Ann}$ and $\t{Fr}_\partial$. 
\end{proof}

\appendix

\section{Topological moduli problems} \label{sec:app_moduli}
Most topological spaces we work with in this paper are \emph{moduli spaces} of one kind or another, i.e., spaces $\MM$ characterized by the mapping-in functor $S\mapsto \hom(S,\MM)$ which classifies solutions of an appropriate moduli problem --- usually, the problem of classifying certain curves with additional structure over $S$ up to isomorphism. In some cases, such as $\mathrm{Fr}_\partial,$ solutions to the moduli problem have no automorphisms, and thus $\MM$ is representable in the category of topological spaces. In other cases, such as $\overline{\cM}_{g,n},$ the solutions to the moduli problem may have automorphisms, and thus $\MM$ is a \emph{topological stack} of some kind. In this appendix we will describe a certain 
``strict'' version of topological stacks (or more precisely, pre-stacks) that is sufficient for our purposes and that we call ``topological moduli problems''. On the level of homotopy categories (for the notion of weak stack homotopy equivalence relevant to this paper), any topological (pre-)stack is homotopy equivalent to a topological space. For example, given a (topological) group $G$, the stack $\pt/G$ is homotopy equivalent to the classifying space $BG$, which is stabilizer-free. However, this ``resolution'' from topological stacks to topological spaces loses some 2-categorical information, and does not, at least naively, respect algebraic structures such as operads; thus even in order to define the operad $\t{DM}^{\t{tree}}$ with spaces of operations built out of the $\overline{\M}_{g,n}$ we must work inside a category of topological moduli problems. 

\subsection{Definitions} \label{sec:appA-definitions}

\begin{definition}
A \emph{groupoid} is a category $\C$ all of whose morphisms are invertible and such that the isomorphism classes of objects form a set, denoted $\pi_0(\C).$ 
\end{definition}
\begin{remark}[Set-theoretic technicalities]\label{rmk:set_theory}
We do not require for groupoids to be small categories, but rather only to be equivalent to small categories. In the constructions in the remainder of this appendix, we assume at every point that we have chosen representatives which are small categories in a consistent way. This is possible via standard small object arguments, which we assume implicitly throughout this appendix. One way to make this precise is to choose an inaccessible cardinal $\kappa$ and implicitly assume that the objects we work with are $\kappa$-small --- see for example \cite[\S1.2.15]{lurie-htt} for a discussion.
\end{remark}

\begin{definition}\label{def:topomopro}
A \emph{topological moduli problem (TMP)} is a (strict) contravariant functor $\Top^{op}\to \t{Gpd}$ from the category of topological spaces to the category of groupoids with (strict) groupoid functors. 
\end{definition}
We write $\t{TMP}$ for the category of such functors, with maps $\X\to \Y$ given by natural transformations. Given a map $f:S\to S'$ of topological spaces we write $f^*:\X(S')\to \X(S)$ for the (contravariantly) associated functor of groupoids.

The notion of TMP is a ``stricter" notion of a topological stack. When working with stacks one often works with non-strict functors (defined via slice categories) and imposes a sheaf condition on 
these functors; we will not require either of these sophistications here (though all topological moduli problems we consider will in fact also be stacks). We only need a sufficient formalism to study homotopically our main examples by means of an associated simplicial object. 

\begin{definition}
Given a topological space $X$, define the \emph{topological moduli problem represented by $X$} to be the functor 
$$
X:S\mapsto \t{Map}_{\Top}(S, X)
$$
(the set of continuous maps, viewed as a discrete groupoid, i.e., a groupoid with no morphisms other than $\mathrm{Id}$). For readability we shall often use the same notation $X$ for both the space and the corresponding moduli problem.
\end{definition}
\begin{remark}\label{rmk:yoneda_obj}
If $\X$ is a topological moduli problem and $S$ is a topological space, it makes sense to think of the groupoid $\X(S)$ as the ``groupoid of maps'' from $S$ (viewed as a TMP) to $\X$, even when $\X$ does not itself come from a topological space. Note that, if $\X$ takes values in \emph{discrete} groupoids, i.e., Sets, by the Yoneda lemma we have a natural isomorphism of sets 
$$
\hom_{TMP}(S,\X) = \X(S).
$$ 
If $\X$ is not valued in discrete groupoids, the groupoid $\X(S)$ can be interpreted as a ``Hom groupoid'' 
only after passing to a higher-categorical setting, which we will not pursue. It is nevertheless a useful intuition that $\X(S)$ is a Hom object. In particular, we will later define for a map $\X\to \Y$ of TMPs the notion of a fiber object $\X_\gamma$ 
over an ``$S$-object'' $\gamma\in \Y(S)$ (Definition~\ref{def:fib}). 
\end{remark}

\begin{remark}
Associated to any topological moduli problem $\X$ is a so-called \emph{coarse moduli space $\X^{\t{coarse}}$}. This is the topological space which is the initial object in the category of pairs $(X,f:\X\to X)$ consisting of a topological space $X$ with a map $f:\X\to X$. In certain favourable cases, e.g. orbifolds, the coarse moduli space is the set $\pi_0\X(\pt)$ with a natural topology induced from $\X$. See \cite[\S4.3]{noohi} for more details.
\end{remark}

\begin{definition}
\begin{enumerate}
\item We say that a map of topological moduli problems $\X\to \Y$ is an \emph{equivalence} if $\X(S)\to \Y(S)$ is an equivalence of groupoids for all topological spaces $S.$ 
\item We say that a moduli problem $\X$ is \emph{representable} if there is a topological space $X$ and a map $\X\to X$ which is an equivalence.
\end{enumerate}
\end{definition}
Note that, if such a space $X$ exists, it is unique. Namely, given a moduli problem $\X,$ let $\X^{\t{iso}}$ be the moduli problem defined by $\X^{\t{iso}}(S): = \pi_0\X(S),$ the set of isomorphism classes of objects of $\X(S).$ It is clear that if $\X$ is representable, the map $\X\to \X^{\t{iso}}$ is an equivalence. In this situation $\X^{\t{iso}}\simeq\X^{\t{coarse}}$.

All the spaces of operations of the operads we use or construct in this paper, and in particular the spaces involved in the sequence of homotopy equivalences (\ref{eq:eq_sequence}), i.e., 
\[
\big(\widetilde{\t{Fr}}_\partial\big)_n, \big(\widetilde{\t{NodFr}}_\partial\big)_n, \t{DM}^{\t{tree}}_n, \big(\t{NodHD}^\t{tree}_{\t{protected}}\big)_n,
\] 
either naturally have an interpretation as topological moduli problems, or are equivalent to such. For example, $(\mathrm{Fr}_\partial)_n$ is equivalent to the moduli problem $(\MM\mathrm{Fr}_\partial)_n$, where $(\MM\mathrm{Fr}_\partial)_n(S)$ is the category whose objects are all bundles of surfaces $X\to S$ with fiberwise 
complex structure and 
boundary parametrization which varies continuously (as explained in Remark \ref{rmk:topology_on_fr}), and whose morphisms are isomorphisms of such data. The resulting moduli problem is representable by a topological space, as we have seen; this is also true for the spaces of operations of $\widetilde{\mathrm{Fr}}_\partial, \widetilde{\mathrm{Ann}}, \widetilde{\mathrm{NodAnn}}, \text{ and }\t{NodHD}^\t{tree}_{\t{protected}},$ though not for $\t{DM}^{\t{tree}}_{n},$ $\mathrm{NodFr}_\partial^{\t{tree}},$ or $\t{NodHD}^{\t{tree}},$ which are non-representable.

We will add the letter ``$\MM$'' as in $\MM\widetilde{\mathrm{Fr}}_\partial$, etc., to denote the topological moduli problem represented by the corresponding topological spaces. When working with a result that is invariant under equivalence of (non-isomorphic) topological moduli problems, we will sometimes drop the $\MM$ (as $\MM\mathrm{Fr}_\partial$ is equivalent to $\mathrm{Fr}_\partial,$ etc.)

The moduli space $\t{DM}^{\t{tree}}_{n}$ is described as a TMP as follows: given a topological space $S$, the groupoid $\t{DM}^{\t{tree}}_{n}(S)$ is the category whose objects are all families $X\to S$ of stable nodal surfaces whose dual graph is a tree, with fiberwise complex structure and marked points which vary continuously, and whose morphisms are isomorphisms of such data. More explicitly, this can be described by viewing $\overline{\cM}_{g,n}$ as an algebro-geometric stack which classifies stable nodal analytic curves of genus $g$ with $n$ marked points, and then further applying the analytification functor from algebraic varieties to topological spaces. See for example~\cite{harris-morrison}.

Denote $\mathrm{NodFr}_{\partial,n}^{\t{tree}}$ the moduli space of stable nodal framed surfaces whose dual graph is a tree and which have $n$ boundary components. This is described as a TMP as follows: given a topological space $S$, the groupoid $\mathrm{NodFr}_{\partial,n}^{\t{tree}}(S)$ is the category whose objects are all families $X\to S$ of stable nodal framed surfaces whose dual graph is a tree and which have $n$ boundary components, with continuously varying fiberwise complex structure,  marked points and parametrizations, and whose morphisms are isomorphisms of such data. We prescribe the local structure near the nodes by associating canonically to a stable nodal framed Riemann surface $\Sigma$ with $n$ analytically parametrized boundary components a closed stable nodal Riemann surface $\t{\bf cap}_2(\Sigma)$ with $2n$ marked points and requiring that the corresponding $S$-family is continuous in the previous sense. The association $\Sigma\mapsto \t{\bf cap}_2(\Sigma)$ is defined as follows: we glue a standard disc $D_i$ to the $i$-th boundary component of $\Sigma$ and we place one marked point $x_i$ at the center of the disc and one other marked point $y_i$ at $1\in S^1=\partial D_i$ for $i=1,\dots,n$. The notation $\t{\bf cap}_2$ is motivated by the fact that we cap and add $2$ marked points for each boundary component of $\Sigma$. The continuity of the parametrization of the boundary components is to be understood in light of the fact that, on the one hand, the boundary does not contain nodes and, on the other hand, any $S$-family is locally trivial away from the nodes.

The correspondence $\Sigma\mapsto \t{\bf cap}_2(\Sigma)$ defines a morphism of TMPs 
$$
\t{\bf cap}_2:\mathrm{NodFr}_{\partial,n}^{\t{tree}}\rightarrow \t{DM}^{\t{tree}}_{2n}. 
$$
We actually get an embedding of groupoids $\t{\bf cap}_2(S):\mathrm{NodFr}_{\partial,n}^{\t{tree}}(S)\hookrightarrow \t{DM}^{\t{tree}}_{2n}(S)$ for any $S$. 
It follows from the construction that this embedding has the following property: if $y_i$ is a marked point such that, upon forgetting it, the surface becomes unstable, then $y_i$ lives on a sphere component which contains exactly two other special points, of which one is a node and the other is a marked point denoted $x_i$. Equivalently, this situation corresponds to one of the irreducible components of a framed curve being a disc. 

We do not make explicit the description of $\t{NodHD}^{\t{tree}}$ as a TMP because it will not be needed. It can be inferred from the above.

\subsection{Reminders about groupoids and simplicial sets} \label{sec:reminders-groupoids-SSet}

We give some basic definitions and results on groupoids and their connection to homotopy theory. References can be found in~\cite{dwyer-kan}. See also~\cite{vistoli}.

\begin{definition}
  \begin{enumerate}
  \item A groupoid is \emph{discrete} if there are no morphisms except identity morphisms. The category of discrete groupoids is equivalent to that of sets.
  \item Given a groupoid $G$, its underlying discrete groupoid $G_\t{discr}$ is the discrete groupoid on the set of 
isomorphism classes of objects.
  \item A groupoid is \emph{quasi-discrete} if there are no automorphisms except identity morphisms. Equivalently, a groupoid $G$ is quasi-discrete if the map $G\to G_\t{discr}$ is an equivalence of categories. 
  \end{enumerate}
\end{definition}

Given a map of groupoids $G\to H$ and an object in $H$, we have a notion of \emph{naive fiber} and a notion of \emph{homotopy fiber} (directly analogous to that of fiber, respectively homotopy fiber in topology). More precisely:

\begin{definition} \label{defi:fibers}
Let 
$$
\pi:G\to H
$$
be a map of groupoids and let $y$ be an object in $H$.  
\begin{enumerate}
\item The \emph{(naive) fiber} of $\pi$ over $y$ is the groupoid $$G_y$$ with objects $\{x\in G\mid \pi(x) = y\}$ and morphisms $\{f:x\to x'\mid \pi(f) = \t{Id}_y\}$.  
\item The \emph{homotopy fiber} of $\pi$ over $y$ is the groupoid $\tilde{G}_y$ with objects pairs $\{(x,g)\mid g:\pi(x)\to y\}$.  Morphisms $(x,g)\to (x', g')$ are maps $f:x\to x'$ that make the following diagram commutative.
$$
\xymatrix
@C=30pt
{
\pi(x) \ar[r]^-{\pi(f)} \ar[d]^g & \pi(x') \ar[dl]^{g'} \\
y & 
}
$$
\end{enumerate}
\end{definition}

\begin{definition} We say that a functor $G\to H$ is a \emph{fibration} (resp., \emph{trivial fibration}) of groupoids if the canonical inclusion functor $G_y\to \tilde{G}_y$ is an equivalence for every $y$ (resp., $G\to H$ is a fibration and all fibers are equivalent to the trivial groupoid $*$).
\end{definition}

The property of being a fibration is equivalent to the following property. 

\begin{definition} \label{defi:Cartesian}
A functor $\pi:G\to H$ is \emph{Cartesian} if,  
for every two objects $x\in G$, $y\in H$ and every map $g:\pi(x)\to y$, there exists an object $x'\in G$ with $\pi(x')=y$ and a morphism $f:x\to x'$ lifting $g$. 

A functor $\pi:G\to H$ is \emph{trivial Cartesian} if it is Cartesian and an equivalence of categories.  
\end{definition}
The following result is straightforward.
\begin{proposition} \label{prop:Cartesian-fibration}
A map of groupoids $G\to H$ is Cartesian (resp., trivial Cartesian) 
iff it is a fibration (resp., trivial fibration). 
\qed
\end{proposition}

We will deal in the sequel with fibrations with quasi-discrete fibers.
\begin{definition}
A fibration of groupoids $G\to H$ is \emph{discrete}, resp., \emph{quasi-discrete}, if every fiber $G_y$ is a discrete, resp., a quasi-discrete groupoid. 
\end{definition}
We will need the following proposition.
\begin{proposition}\label{prop:quasi-discrete}
Every quasi-discrete fibration of groupoids $G\to H$ canonically factors through a 
discrete fibration $G\to G'
\to H$ such that $G\to G'$ is an equivalence.
\end{proposition}
\begin{proof}
We have a functor $F:H\to\t{Set}$ given by $y\mapsto (G_y)_{\t{discr}}$, the isomorphism classes of objects over $y$. Let $G'$ 
be the Grothendieck construction applied to this functor: this is the category with objects $(y, \bar{x})$ for $\bar{x}$ an isomorphism class in $G_y$ and morphisms $\bar{x}\to \bar{x}'$ corresponding to pairs $(\bar{x}, g)$ for $\bar{x}$ over $y$ and $g:y\to y'$ such that $F(g)(\bar{x}) = \bar{x}'.$ It is a direct check that the resulting functor 
$G' \to H$ is a fibration and that, for $G\to H$ quasi-discrete, the functor 
$G\to G'$ (given by $x\mapsto (\pi(x),\bar{x})$) is an equivalence.
\end{proof}

\begin{remark} 
There is a model category structure on groupoids for which 
fibrations are Cartesian functors and in particular any arrow can be Cartesian resolved (for example by replacing $G$ by the union $\sqcup_{y\in Y}\tilde{G}_y$). We will however not use a model category structure at this level --- rather, we will view groupoids as homotopy objects by taking the simplicial nerve (a.k.a., the classifying space).
\end{remark}

We now discuss groupoids in a simplicial context. 
The simplicial category $\Delta$ is the category of nonempty finite totally ordered sets $[n]=\{0,\dots,n\}$. A \emph{simplicial set} is a presheaf on $\Delta,$ i.e., a functor $\Delta^{op}\to \t{Set}$. Given a simplicial set $X:\Delta^{op}\to \t{Set}$, we write $X_n = X[n].$ More generally, a \emph{simplicial object} in a category $\mathcal{C}$ is a functor $\Delta^{op}\to \C.$ Recall from~\S\ref{sec:model-SSet} that 
simplicial sets form a model category~\cite[II.3, Theorem~3]{quillen}, which is Quillen equivalent to the Quillen model category on topological spaces (with equivalences given by weak homotopy equivalences). 
The equivalence is given by the pair of functors
$$|\cdot|:\t{SSet}\leftrightarrows \Top:C_\Delta,$$ with $|\cdot|$ the topological realization functor and $C_\Delta$ the 
singular functor 
given by $C_\Delta(X)_n: = \hom(\Delta^n_\t{top}, X),$ for $X$ any topological space and $\Delta^n_\t{top}$ the topological $n$-simplex. The simplicial morphisms between the $C_\Delta(X)$ arise from the fact that the $\Delta^n_\t{top}$, $n\ge 0$ form a \emph{cosimplicial} object in $\Top,$ i.e., we have a functor $\Delta^*_\t{top}:\Delta\to \Top.$

Groupoids can also be interpreted in the context of homotopy theory, via the \emph{nerve} construction.

\begin{definition}\label{defi:nerve}
The \emph{nerve} of a groupoid $G$ is the simplicial set 
$$
NG: [n]\mapsto \hom_{\t{Cat}}([n],G).
$$ 
\end{definition}
More concretely, $NG_n$ is the set of composable sequences of morphisms $x_0\to x_1\to\cdots\to x_n.$ Note that, if $G$ is a group, the topological realization $|NG|$ of the nerve of $G$ is a classifying space for $G$. 

We will need the following combinatorial result.
\begin{lemma}\label{simp_etale}
Suppose $G\to H$ is a discrete fibration, and let $NG\to NH$ be the associated map of nerves. A simplex $(x_0\to \dots \to x_n)\in NG_n$ in the preimage of a simplex $\sigma = (y_0\to \dots\to y_n)\in NH_n$ is uniquely determined by $x_0.$ 
\end{lemma}
\begin{proof}
This follows from the fact that a map $f:x\to x'$ in a discrete fibration is uniquely determined by $x$ and a morphism $\pi(f)$ from $\pi(x)$ to $\pi(x')$. 
\end{proof}

Simplicial sets also have a notion of homotopy fiber, which we call \emph{simplicial preimage} and which we now explain. 
Write $\Delta^n_\t{simp}$ for the $n$-simplex, i.e., the simplicial set represented by $[n]$. Given a simplicial set $X$, an element $\sigma\in X_n$ determines canonically a map of simplicial sets $\sigma_\t{simp}:\Delta^n_\t{simp}\to X$. 

\begin{definition}
Let $f:Y\to X$ be a map of simplicial sets. 
The \emph{simplicial preimage of an element $\sigma\in A_n$ under $f$} is the homotopy fiber product 
$$
Y_\sigma= \Delta^n_\t{simp}\times_{X}^h Y
$$ 
with respect to the simplicial map $\sigma_\t{simp}: \Delta^n_\t{simp}\to X$ determined by $\sigma$.
\end{definition}

The homotopy fiber product is the homotopy limit of the diagram $\Delta^n_\t{simp}\stackrel{\sigma_\t{simp}}{\longrightarrow} X \stackrel{f}{\longleftarrow} Y$ for the Quillen model category structure on $\t{SSet}$.
   
We will need the following standard result from homotopy theory, usually cited as a special case of Quillen's ``Theorem A'' from~\cite[\S1]{quillen_k}.

\begin{proposition} \label{prop:simplicial-we-fibers}
A map of simplicial sets $f:Y\to X$ is a weak equivalence if and only if, for any integer $n\ge 0$ and any simplex $\sigma\in X_n$, the simplicial preimage $Y_\sigma$ is weakly contractible.
\end{proposition}

\begin{proof}
The Quillen model structure on $\t{SSet}$ is \emph{right proper} in the sense of~\S\ref{sec:hocolim}. Given a right proper model category, the homotopy limit of a diagram $Z\rightarrow X \leftarrow Y$ can be computed by replacing any of the two maps by a fibration and taking the ordinary limit of the resulting diagram~\cite[Proposition~13.3.7]{hirschhorn} (compare with Lemma~\ref{lem:left-proper-pushout}). As a consequence, we can assume without loss of generality that the map $f$ in the statement is a fibration for the Quillen model structure on $\t{SSet}$, i.e., a Kan fibration of simplicial sets. In this case the simplicial preimages are computed as ordinary pullbacks. 

We now prove that a Kan fibration is a weak equivalence (also called \emph{trivial Kan fibration}) if and only if all its simplicial preimages are contractible. This is stated in~\cite[II.2]{quillen}. The proof relies on the equivalent characterization of trivial Kan fibrations as maps of simplicial sets that have the right lifting property with respect to all the inclusions $\p \Delta^k\hookrightarrow \Delta^k$, $k\ge 0$ (\cite[II.3, Theorem~3]{quillen}, \cite[Theorem~11.2]{goerss-jardine}). We prove the direct implication: given a trivial Kan fibration $f:Y\to X$, its pullback along any map of simplicial sets is also a trivial Kan fibration because it has the right lifting property with respect to all inclusions $\p \Delta^k\hookrightarrow \Delta^k$. In particular, any simplicial preimage of $f$ is a trivial Kan fibration over a simplex, hence weakly contractible because the simplex is weakly contractible. We prove the converse implication: let $f:Y\to X$ be a Kan fibration and assume that all its simplicial preimages $Y_\sigma$ are weakly contractible. This is equivalent to saying that the canonical map $Y_\sigma\to \Delta^n$ is a trivial Kan fibration for all $\sigma\in X_n$, $n\ge 0$. We wish to construct the dotted arrow in the first diagram below. Enhancing it to the second diagram below, the right lifting property for the map $Y_\sigma\to\Delta^k$ with respect to the inclusion $\p \Delta^k\hookrightarrow \Delta^k$ determines a lift $\Delta^k\to Y_\sigma$. This lift can be post-composed with the map $Y_\sigma\to Y$ in order to obtain the desired lift $\Delta^k\to Y$. 
$$
\xymatrix{
\p \Delta^k \ar[rr] \ar@{^(->}[dd] & & Y \ar@{->>}[dd] \\
& & \\
\Delta^k\ar[rr]_\sigma \ar@{.>}[uurr] && X
}
\qquad 
\xymatrix{
\p \Delta^k \ar[dr] \ar[drrr] \ar[dd] & & & \\
& Y_\sigma=\sigma^*Y \ar@{^(->}[dd] \ar[rr] & & Y \ar@{->>}[dd]^f \\ 
\Delta^k \ar@{=}[dr] \ar@{.>}[ur] & & & \\
& \Delta^k \ar[rr]_\sigma & & X
}
$$
\end{proof}

In the next section we will also use the category of bisimplicial sets, defined as follows.
\begin{definition}\label{defi:bisimp}
The category of bisimplicial sets $$\t{S}^2\t{Set}:= \t{Fun}((\Delta^{op})^2,\t{Set})$$ is the category of presheaves on the \emph{bisimplicial category} $(\Delta^{op})^2 = \Delta^{op}\times \Delta^{op}.$ 
\end{definition}
Given a bisimplicial set $X,$ we write $X_{m,n}:= X([m]\times [n]).$
Note that, by adjunction, $\t{S}^2\t{Set}$ is equivalent to the category 
\[\t{Fun}\left(\Delta^{op}, \t{Fun}\left(\Delta^{op}, \t{Set}\right)\right)\] of simplicial objects in the category of simplicial sets. Here, if $X:\Delta^{op}\to \t{SSet}$ is a simplicial object in simplicial sets, the associated bisimplicial set has $X^\t{Bi}_{m,n}: = (X_m)_n,$ where $X_m$ is the simplicial set $X([m]).$

Given a bisimplicial set $X_{*,*}:(\Delta^{op})^2\to \t{Set}$, we have an associated ``total'' simplicial set defined as follows.
\begin{definition}\label{defi:sset_totalization} Define $$X^\t{tot}:= X\circ \t{Diag}_{\Delta^{op}}:\Delta^{op}\to \t{Set},$$ where the functor $\t{Diag}_{\Delta^{op}}:\Delta^{op}\to (\Delta^{op})^2$ is the diagonal functor.
\end{definition}

There are several different Quillen equivalent model structures on the category $\t{S}^2\t{Set}$ of bisimplicial sets, and all of them are Quillen equivalent to the category of simplicial sets via the totalization functor (see~\cite[Chapter~IV]{goerss-jardine}). In particular, the simplicial set $X^\t{tot}$ contains the same ``topological'' information as the bisimplicial set $X.$

\subsection{Simplicial realization of a topological moduli problem}
We will not attempt to construct a model category (or even a homotopy category) of topological moduli problems\footnote{Though note that topological moduli problems are a full subcategory of the category of simplicial presheaves on topological spaces, which has a structure of model category Quillen equivalent to that of simplicial sets, and it would be possible to work within this framework.}. Instead, we will associate to every topological moduli problem $\X$ 
a simplicial set\footnote{This is actually an $\infty$-groupoid.} $C_\Delta(\X)$
of ``simplicial chains'' and view $\X$ as represented by $C_\Delta(\X)$ in a homotopy theoretic sense. 

We have a standard covariant \emph{realization functor} $$R:\Delta\to \Top,$$ taking $[n]$ to the $n$-simplex $\Delta^n.$ We define the simplicial groupoid $C^\t{Gpd}_\Delta(\X)$
to be the strict functor 
$$
C^\t{Gpd}_\Delta(\X)=\X\circ R^{op}:\Delta^{op}\to \t{Gpd}.
$$ 
Note that, for $\X = X$ the moduli problem represented by a topological space, $C^\t{Gpd}_\Delta(X)$ is canonically equivalent to the levelwise discrete simplicial groupoid associated to the singular simplicial set of $X$, denoted $C_\Delta(X)$. 

Composing with the nerve functor (Definition \ref{defi:nerve}) $N:\t{Gpd}\to \t{SSet}$ from groupoids to simplicial sets we obtain a functor 
$$
N\circ C_\Delta^\t{Gpd}(\X):\Delta^{op}\to \t{SSet}.
$$ 
By definition, such a functor is equivalent to a bisimplicial set.   
We call it \emph{the bisimplicial realization} associated to the topological moduli problem $\X$, and denote it 
$$
C_\Delta^\t{Bi}(\X):(\Delta^{op})^2\to \t{Set}.
$$
\begin{definition}
For a topological moduli problem $\X,$ we denote $$C_\Delta(\X)$$ the simplicial set given by the totalization of the 
bisimplicial set $C_\Delta^\t{Bi}(\X)$. We call $C_\Delta(\X)$ the \emph{classifying simplicial set} of $\X$. 
\end{definition}
The assignment $\X\to C_\Delta(\X)$ is natural and defines a covariant functor from the category of topological moduli problems to simplicial sets. To illustrate the definition we give two examples, which are the extreme cases of topological moduli problems.
\begin{itemize}
\item If $\X = X$ is the moduli problem represented by a topological space, the bisimplicial set $C_\Delta^\t{Bi}(\X)$ canonically decouples as a product of two functors $\Delta^{op}\to \t{Set},$ namely $C_\Delta(X)\times \pt$. Thus the associated simplicial set $C_\Delta(\X)$ is canonically isomorphic to the usual simplicial set of singular simplices $C_\Delta(X),$ further justifying our frequent abuse of notation of using the same term for the space $X$ and the topological moduli problem $\X$ which it represents. 
\item If $\X=\pt/G$ is a point with a discrete group of automorphisms $G$, the bisimplicial set $C_\Delta^\t{Bi}(\X)$ canonically decouples as a product $\pt \times N G$, where $N$ is the nerve functor, i.e., the simplicial classifying space functor.
\end{itemize}
Thus, the totalization $C_\Delta(\X)$ combines the topological and  the groupoid features of the topological moduli problem $\X$.  

The \emph{classifying space functor} $\X\mapsto |C_\Delta(\X)|$ allows us to view any topological moduli problem as functorially represented by a topological space. However, as this topological space is big and difficult to work with, we avoid working with it directly. Instead we make comparisons in the category of topological moduli problems, and identify certain properties of maps of topological moduli problems which guarantee that the underlying maps of classifying spaces are homotopy equivalences.

\subsection{Trivial Serre fibrations of topological moduli problems}
To make sense of results such as Lemma~\ref{lem:from-DM-to-NodFr} in higher genus, we need a notion of fiber for certain maps of topological moduli problems. 

\begin{definition}\label{def:fib} 
A map $\X\to \Y$ of TMP is a \emph{combinatorial fibration} if $\X(S)\to \Y(S)$ is a fibration of groupoids for each $S$.
\end{definition}
Informally, if $\pi:\X\to\Y$ is a map of TMP, if $\X(S)$ and $\Y(S)$ are defined as some geometric data on $S$ up to isomorphism, and if $\pi(S)$ is a ``forgetful" map which discards some part of the geometric data, then $\pi$ is a combinatorial fibration. Indeed, given an isomorphism $i:\alpha\to\beta$ between two structures $\alpha,\beta\in \Y(S)$, and given an object $\tilde{\beta}$ of $\X(S)$ over $\beta$ which consists of some additional geometric data, we can pull back this geometric data under $i$ to produce a diagram verifying the fibration property. Using this idea we can prove the following proposition.

\begin{proposition}\label{prop:combifib}
For any $n,$ the maps \[\t{\bf cap}: \t{NodFr}^\t{tree}_{\partial,n}\to \t{DM}^\t{tree}_n\] and 
\[\boldsymbol{\pi}: \MM \t{NodHD}^\t{tree}_{\t{protected}, n}\to \widetilde{\t{NodFr}}_{\partial, n}^\t{tree}\] are combinatorial fibrations.
\end{proposition}

\begin{proof}
We first discuss the map $\t{\bf cap}$. This is the composition of the maps 
\[
\t{NodFr}^\t{tree}_{\partial,n}\stackrel{\t{\bf cap}_2}\longrightarrow \t{DM}^\t{tree}_{2n}\stackrel{\t{\bf forget}_n}\longrightarrow \t{DM}^\t{tree}_n,
\]
where $\t{\bf cap}_2$ was defined in~\S\ref{sec:appA-definitions} and $\t{\bf forget}_n$ is the map which forgets the marked points $y_1,\dots,y_n$ (the marked points for $\t{DM}^\t{tree}_{2n}$ are denoted $x_1,\dots,x_n,y_1,\dots,y_n$, cf.~\S\ref{sec:appA-definitions}). Note that the image of $\t{\bf cap}_2$ is contained in $\t{DM}^\t{tree,\prime}_{2n}$, which is the TMP such that, for each $S$, the groupoid $\t{DM}^\t{tree,\prime}_{2n}(S)$ is the full subgroupoid of $\t{DM}^\t{tree}_{2n}(S)$ consisting of families of curves with the property that, if $y_i$ is a marked point such that, upon forgetting it, the underlying surface becomes unstable, then $y_i$ lives on a sphere component which contains exactly two other special points, of which one is a node and the other is a marked point denoted $x_i$. We claim that 
\[
\t{NodFr}^\t{tree}_{\partial,n}\stackrel{\t{\bf cap}_2}\longrightarrow \t{DM}^\t{tree,\prime}_{2n}\stackrel{\t{\bf forget}_n}\longrightarrow \t{DM}^\t{tree}_n
\]
is a composition of combinatorial fibrations, hence is a combinatorial fibration. In view of Proposition~\ref{prop:Cartesian-fibration}, it is enough to show that, for any $S$, the maps  
\[
\t{NodFr}^\t{tree}_{\partial,n}(S)\stackrel{\t{\bf cap}_2(S)}\longrightarrow \t{DM}^\t{tree,\prime}_{2n}(S)\stackrel{\t{\bf forget}_n(S)}\longrightarrow \t{DM}^\t{tree}_n(S)
\]
are Cartesian (Definition~\ref{defi:Cartesian}). 
\begin{itemize}
\item We prove that the map $\t{\bf cap}_2(S)$ is Cartesian. Given two isomorphic $S$-families in $\t{DM}^\t{tree,\prime}_{2n}(S)$ with a preferred isomorphism $g$ and a lift to $\t{NodFr}^\t{tree}_{\partial,n}(S)$ of the first family, we can reinterpret that lift as the data of a continuously varying family of analytically parametrized simple curves in the fibers. The isomorphism $g$ can then be used in order to produce such a continuously varying family of analytically parametrized simple curves in the fibers of the second family. This provides a lift of the second family, as well as a lift of $g$. 
\item We prove that the map $\t{\bf forget}_n(S)$ is Cartesian. Consider two isomorphic $S$-families $E,F\in \t{DM}^\t{tree}_n(S)$, a preferred isomorphism $g:E\stackrel\simeq\longrightarrow F$, and a lift of $E$ to $\t{DM}^\t{tree,\prime}_{2n}(S)$ denoted $\tilde E$. Let $x_{1,s},\dots,x_{n,s}$ be the marked points of the curve $E_s$, $s\in S$, and $\tilde x_{1,s},\dots,\tilde x_{n,s},y_{1,s},\dots,y_{n,s}$ the marked points of its lift $\tilde E_s$. At each point $s\in S$ and for each $i=1,\dots,n$ one of the following two things can happen: either the removal of the point $y_{i,s}$ does not destabilize the irreducible component of $\tilde E_s$ on which it lies, or it destabilizes it, in which case that component is a sphere which is contracted under the forgetful map, the marked point $\tilde x_{i,s}$ becoming $x_{i,s}$. Equivalently, the lift $\tilde E_s$ differs from $E_s$ by either adding some non-destabilizing marked points $y_{i,s}$, or by blowing up some marked points $x_{i,s}$ into spheres which acquire marked points $\tilde x_{i,s}$ and $y_{i,s}$. We can then define a lift $\tilde F$ of $F$ as follows: we mark additional points $y'_{i,s}=g(y_{i,s})$ for all indices $i$ such that $y_{i,s}$ is not destablizing, and we blow up the points $x'_{i,s}=g(x_{i,s})$ for all indices $i$ such that $y_{i,s}$ is destabilizing, marking two new points $\tilde x'_{i,s}$ and $y'_{i,s}$ on the resulting sphere bubble. Since any two spheres with three marked points are uniquely isomorphic, the isomorphism $g$ extends uniquely to an isomorphism $\tilde g:\tilde E\stackrel\simeq\longrightarrow \tilde F$. 
\end{itemize}
This proves that the map $\t{\bf cap}$ is a combinatorial fibration.  
The proof that the map $\boldsymbol{\pi}$ is a combinatorial fibration is very similar to the proof for $\t{\bf cap}_2$, and we omit the details.  
\end{proof}

\begin{definition}[Fibers of topological moduli problems]\label{def:fibrep} \qquad 

Let $\pi:\X\to \Y$ be a map of TMP.
\begin{enumerate}
\item Given a topological space $I$ and an object $\gamma\in \Y(I)$ (to be thought of as a map $I\to \Y$), \emph{the fiber} $\X_\gamma$ is the topological moduli problem with $\X_\gamma(S): = \bigsqcup_{t:S\to I} \X(S)_{t^*\gamma}$. Here $t^*\gamma$ is the pullback in $\Y(S)$ (to be thought of as the composition $\gamma\circ t$) and $\X(S)_{t^*\gamma}$ is the fiber for the map of groupoids $\X(S)\to \Y(S).$

The fiber $\X_\gamma$ is endowed with a canonical map $\X_\gamma\to I$. 
\item We say that the map $\pi$ is a \emph{fiberwise representable fibration} if it is a combinatorial fibration and,
for any topological space $I$ and object $\gamma\in \Y(I)$, the fiber $\X_\gamma$ is quasi-discrete (i.e., equivalent to a topological space). We say that $\pi$ is \emph{strictly fiberwise representable} if any such fiber $\X_\gamma$ is discrete. 
\end{enumerate}
\end{definition}
It follows from Proposition~\ref{prop:quasi-discrete} that any fiberwise representable fibration is equivalent to one which is strictly fiberwise representable.

\begin{definition}\label{def:fibtriv}
Let $\pi:\X\to \Y$ be a map of TMP which is a fiberwise representable fibration. We say that $\pi$ is a  \emph{fiberwise representable trivial weak fibration} if, for any $m\ge 0$ and any simplex $\sigma\in \Y(\Delta^m)$, the fiber $\X_\sigma$ is weakly contractible.
\end{definition}
\begin{lemma}\label{lm:fib_heq} 
Suppose that $\pi:\E\to \X$ is a fiberwise representable trivial weak fibration. 
Then the induced map of simplicial sets $C_\Delta(\E)\to C_\Delta(\X)$ is a weak homotopy equivalence. 
\end{lemma}
\begin{proof}
As noted above, we can assume without loss of generality that $\E\to \X$ is strictly fiberwise representable. In particular, the fibers of $\E(S)\to \X(S)$ are discrete for any $S$.  Now we consider the map of bisimplicial sets $C_\Delta^\t{Bi}(\E)_{m,n}\to C_\Delta^\t{Bi}(\X)_{m,n},$ where $m$ is the simplex degree (corresponding to the category $\X(\Delta^m_\t{top})$ of maps from the topological simplex) and $n$ is the nerve degree. It is enough to show that, at a fixed nerve degree $n$, the map of simplicial sets $C_\Delta^\t{Bi}(\E)_{*,n}\to C_\Delta^\t{Bi}(\X)_{*,n}$ is a weak homotopy equivalence. We do this in two steps.

\noindent {\bf Case $n = 0$.} Since we set the nerve degree to zero, we are considering the simplicial set on \emph{objects} of $\X(\Delta^m_\t{top})$ (and similarly for $\E$). We use the following standard fact.

\begin{proposition} \label{prop:Serre-Kan-trivial}
Let $f:U\to V$ be a map of topological spaces. The following statements are equivalent: 
\begin{itemize}
\item $f$ is a trivial Serre fibration, i.e., a Serre fibration with weakly contractible fibers. 
\item the map of simplicial sets $C_\Delta(f):C_\Delta(U)\to C_\Delta(V)$ is a trivial Kan fibration in the sense of~\S\ref{sec:model-SSet}. 
\item the map of simplicial sets $C_\Delta(f):C_\Delta(U)\to C_\Delta(V)$ is a Kan fibration that has weakly contractible simplicial preimage $C_\Delta(U)_\sigma$ over any $m$-simplex $\sigma\in C_\Delta(V)$.
\end{itemize}
\end{proposition}

\begin{proof}[Proof of Proposition~\ref{prop:Serre-Kan-trivial}] That $f$ is a trivial Serre fibration if and only if $C_\Delta(f)$ is a trivial Kan fibration is proved in~\cite[II.3, Lemma~2]{quillen}, see also~\cite[Example~5.15]{heuts-moerdijk}.\footnote{It is also true that $f$ is a Serre fibration if and only if $C_\Delta(f)$ is a Serre fibration~\cite[Example~5.11]{heuts-moerdijk}.} On the other hand, we proved in Proposition~\ref{prop:simplicial-we-fibers} that a Kan fibration is trivial if and only if all its simplicial preimages are weakly contractible. 
\end{proof}

{\it Proof of Lemma~\ref{lm:fib_heq} continued.} 
Given an object $\sigma_\t{top}\in \X(\Delta^m_\t{top})$, we apply Proposition~\ref{prop:Serre-Kan-trivial} to the pair of topological spaces $V = \Delta^m_\t{top}$ and $U = \E_{\sigma_\t{top}}$, the fiber of $\E$ over $V$. We are using that, by strict fiberwise representability, $\E_{\sigma_\t{top}}$ is isomorphic to the (discrete) topological moduli problem represented by a space. We deduce that all simplicial preimages of simplices in $C_\Delta^\t{Bi}(\X)_{*,0}$ under the map $C_\Delta^\t{Bi}(\E)_{*,0}\to C_\Delta^\t{Bi}(\X)_{*,0}$ are weakly contractible, hence we get a weak homotopy equivalence of simplicial sets at $n = 0$.

\noindent {\bf Case $n\ge 0.$} From Lemma~\ref{simp_etale} we see that all simplicial preimages of the map $C_\Delta^\t{Bi}(\E)_{*,n}\to C_\Delta^\t{Bi}(\X)_{*,n}$ are isomorphic to simplicial preimages of 
$C_\Delta^\t{Bi}(\E)_{*,0}\to C_\Delta^\t{Bi}(\X)_{*,0},$ the corresponding map at $n = 0.$ We conclude by the $n = 0$ case. 
\end{proof}

\section{The dendroidal category and Segal operads} \label{sec:dendroidal}
In Proposition~\ref{prop:from-protected-to-unprotected-Segal} 
we prove that, for any number $n$ of inputs, there is a map $$\boldsymbol{\pi}:\MM\t{NodHD}^\t{tree}_{\t{protected}, n}\to \widetilde{\t{NodFr}}_{\partial, n}^\t{tree}$$ of topological moduli problems which induces a homotopy equivalence on underlying simplicial sets. Now the LHS is a model (in topological spaces) for the homotopy pushout of the diagram $\pt\from S^1\to \t{Fr}_\partial,$ and we also know from Lemma~\ref{lem:from-DM-to-NodFr-Segal} 
that the RHS is equivalent to $\t{DM}^{\t{tree}}_n$. 
In this appendix we explain how to ``upgrade" such homotopy equivalences on the level of spaces of operations to homotopy equivalences on the level of operads. 
For this we need to discuss operads valued in topological moduli problems. A convenient language in this context is the formalism of dendroidal objects and Segal operads (Cisinski-Moerdijk~\cite{cisinski-moerdijk-dendroidal-sets}), which is a relative of Lurie's theory of $\infty$-operads~\cite{lurie-higher-algebra}. We explain this formalism, how the operad objects and maps we have been working with, e.g., the operad $\t{DM}^\t{tree}$ and the map $\t{DM}^\t{tree}\to \t{NodFr}^\t{tree}_\partial$, fit into it, and how this formalism is compared to that of topological operads. 

The reason for using Segal operads is that the classical language of operads is not well adapted to deal with spaces of operations 
that are topological moduli problems. Indeed, if we tried to impose an operad structure, then the associativity relations would need to be equalities of functors, whereas in our situation we only have isomorphisms.  Said differently, we need to use the 2-category structure on groupoids to define operads such as $\t{DM}^\t{tree}$. 
(A useful analogy is symmetric monoidal groupoids vs. abelian monoids.)

\subsection{The Moerdijk-Weiss tree category}

Here we describe the dendroidal category $\Omega$ of Moerdijk and Weiss~\cite{moerdijk-weiss2007}, 
which plays in higher operad theory a role analogous to that played in homotopy theory by the simplicial category $\Delta$. In the same way that objects of $\Delta$ can be understood as combinatorial categories (finite ordered posets), the objects of $\Omega$ are combinatorial operads associated to trees.

\begin{definition}[The free colored operad on a tree] \label{defi:free-colored-operad-on-tree}
Let $\tau$ be a tree of operations (i.e., a finite tree with half-edges, see~\S\ref{sec:operad_of_labeled_rooted_trees}). The free colored operad $[\tau]$ on $\tau$ is the operad with colors $C = \mathrm{Edge}_\tau$ and operations generated by a set $\{o_v\}$ indexed by vertices $v\in \mathrm{Vert}_\tau,$ where $o_v$ has inputs $e_1^\t{in}(v),\dots, e_{|v|}^\t{in}(v)$, 
the incoming edges at $v$ with some choice of ordering, and output $e^\t{out}(v)$, the outgoing edge at $v$. 
\end{definition}

In this definition the generators are described in terms of a choice of ordering of the incoming edges at each vertex $v$, or equivalently in terms of a choice of planar structure for the tree. A different choice of planar structure would give different generators, but the same operad.

When $\tau$ is the linear tree $\tau_{[n]}:= \  \to \bullet\to\dots\to \bullet\to$ with $n$ vertices, the colored operad only has $1\to 1$ operations, i.e., is a category. In fact, it is the category associated to the quiver corresponding to the dual graph: it has $n+1$ objects associated to the linear tree and a unique arrow $i\to j$ for $i\le j$. in other words, it is precisely the linear poset category $[n]$. 

For a general tree $\tau,$ one can check that operations in $[\tau]$ are indexed by full labeled subtrees. Namely, we say that a subgraph $\tau'\subset \tau$ is \emph{full} if every edge or half-edge in $\tau$ with an endpoint in $\tau'$ is also an edge or a half-edge in $\tau'$
(note that a half-edge in $\tau'$ might have another endpoint in $\tau$). The operad $[\tau]$ has a (unique) operation with inputs $e_1,\dots, e_n$ and output $e^\t{out}$ if and only if there is a full labeled sub-tree $\tau'$ with leaves $e_1,\dots, e_n$ and root $e^\t{out}.$

\begin{definition}
The \emph{dendroidal category} $\Omega$ is the full subcategory of the category of colored operads on the objects $[\tau]$ as $\tau$ ranges over all trees.
\end{definition}

The next definition describes objects which generate $\Omega$ in a sense that we will make precise. 

\begin{definition} \qquad 

\begin{itemize}
\item Define $\mid$ to be the tree with one edge and no vertices. The operad $[\mid]\in \Omega$ is the operad with a single identity operation.
\item Define the \emph{$n$-corolla} $\kappa_n$ to be the tree of operations with one vertex, $n$ input edges, and one output edge. The operad $[\kappa_n]$ is the operad whose spaces of operations are empty except in arity $n$, where $\fS_n$ acts freely and transitively, i.e., $[\kappa_n]$ has a unique-up-to-relabeling operation in arity $n$. 
\end{itemize}
\end{definition}
Note that for a general colored operad $O$, maps $[\mid]\to O$ are in bijection with colors of $O$ and maps $[\kappa_n]\to O$ correspond to all operations with $n$ inputs (with arbitrary incoming and outgoing colors). In particular, if $\tau$ is a tree then $\hom([\mid], [\tau]) = \mathrm{Edge}_\tau$ and $\hom([\kappa_n],[\tau])$ is the set of full subtrees of $\tau$ with $n$ leaves. We note two important special cases.

\begin{definition}  \qquad 

\begin{itemize}
\item Given a vertex $v$, define $i_v:[\kappa_{|v|}]\to [\tau]$ to be the map corresponding to the subtree of edges around $v$ (equivalently, the map classifying the operation $v$ in $[\tau]$). 
\item For any tree $\tau$ with $n$ leaves, define the \emph{cocontraction map} $s_\tau:[\kappa_n]\to [\tau]$ to be the map classifying the operation corresponding to the tree $\tau$ itself viewed as a subtree (i.e., the composition of all operations in $\tau$).
\end{itemize}
\end{definition}

For any full subtree $\tau'\subset \tau,$ let $\tau''$ be the tree obtained by contracting all internal edges of $\tau'$ into a single point. Then there are obvious maps $i_{\tau,\tau'}:[\tau']\to [\tau]$ and $s_{\tau,\tau'}:[\tau'']\to [\tau]$ defined by extending $i_v$ for $v\in \tau'$, resp., $s_{\tau'},$ by the identity operation for all vertices in the complement 
of $\tau'.$ Maps of the form $i_{\tau,\tau'}, s_{\tau,\tau'}$ together with automorphisms of trees generate the tree category $\Omega$~\cite[\S3]{moerdijk-weiss2007}.

Let $\C$ be a category with monoidal structure given by products and $O$ a monochromatic operad in $\C$. This induces a contravariant representation of the dendroidal category $\Omega$ in $\C$, sending every tree $\tau$ to the set $\prod_{v\in \t{Vert}(\tau)} O_{|v|}$ and sending the cocontraction map $s_\tau$ to the operad composition operation $\prod_{v\in \t{Vert}(\tau)} O_{|v|}\to O_N$ (for $N$ the number of leaves). In particular, this representation determines the operad $O$ (the identity operation is the ``empty'' composition operation corresponding to $\tau = \ \mid$ and the action of $\sigma\in \fS_n$ on $O_n$ is recovered from the morphisms $\kappa_\sigma: \kappa_n\to \kappa_n$ which permute the incoming edges of a corolla). 

\subsection{Segal Operads and Pre-Operads} \label{sec:Segal-op-and-preop}

In this subsection we follow Cisinski-Moerdijk~\cite{cisinski-moerdijk-dendroidal-sets} and work with models for colored operads, despite the fact that the objects we care about are monochromatic. 

Suppose $\C$ is a category which is closed under products and coproducts and let $\t{Discr}_\C$ be the essential image of the standard functor $\t{Set}\to \C$ given by unions of the terminal object $*\in \C.$ 

\begin{definition}
A \emph{Segal pre-operad} $\Lambda$ in $\C$ is a $\C$-valued presheaf on $\Omega$, i.e., a functor $\Lambda:\Omega^{op}\to \C$ such that $\Lambda([\mid]) \in \t{Discr}_\C.$ 
\end{definition}
One can think of Segal pre-operads as colored operads with non-uniquely defined compositions; $\Lambda([\mid])$ is the set of colors. When $\Lambda$ is a monochromatic operad, it has a model with $\Lambda([\mid]) \cong *$.

If $\C$ is additionally endowed with a class of equivalences including $\t{Iso}_\C$ (for example, if $\C$ is a model category), we make the following definition.

\begin{definition}
A \emph{Segal operad in $\C$} is a Segal pre-operad in $\C$ such that, for every tree $\tau$, the product of the maps 
\begin{align}\label{eq:segal-tree}\Lambda(i_v):\Lambda([\tau])\to \Lambda([\kappa_{|v|}])
\end{align} 
is an equivalence between $\Lambda([\tau])$ and $\prod_{v\in \t{Vert}(\tau)} \Lambda([\kappa_{|v|}])$. 
\end{definition}

Note that an ordinary (strict) colored operad in $\C$ is precisely a Segal pre-operad where this map is an equality (i.e., where $\Lambda([\tau])$ is \emph{defined} as being $\prod_{v\in \t{Vert}(\tau)} \Lambda([\kappa_{|v|}])$). Given an ordinary operad $O$, we write $\t{Seg}(O)$ for the Segal operad associated to $O$, with
$$\t{Seg}(O)([\tau]) : = \prod_{v\in \t{Vert}(\tau)}O([\kappa_{|v|}]).$$

In fact, Segal (pre)operads have a model category structure which is related to the Berger-Moerdijk model category of topological operads 
(and also to Lurie's $\infty$-operads~\cite[\S2]{lurie-higher-algebra}) by a chain of Quillen equivalences. This is proved by Chu, Haugseng, and Heuts in the paper~\cite{chu-haugseng-heuts}, which relies on \cite{heuts-hinich-moerdijk}, \cite{cisinski-moerdijk-dendroidal-segal-spaces}, \cite{cisinski-moerdijk-dendroidal-sets}, \cite{barwick}. 
Since we will only be interested in the weak equivalences, we will not need the full model structure. Rather, we need the following results 
from~\cite{cisinski-moerdijk-dendroidal-segal-spaces,cisinski-moerdijk-dendroidal-sets}.

\begin{theorem}[\cite{cisinski-moerdijk-dendroidal-segal-spaces,cisinski-moerdijk-dendroidal-sets}] \qquad
\label{thm:CM}

\begin{enumerate}
  \item There is a ``strictification'' functor 
$$
  W_{\t{Seg}}:\t{PreSegOp}_
  \cC
  \to \t{Op}_
  \cC
$$
from Segal pre-operads in $\cC$ to strict operads in $\cC$, which is equivalent to the $W$-construction when applied to strict operads. 
In other words, we have a canonical isomorphism of functors
    \[W_{\t{Seg}}\circ \t{Seg}\cong W:
    \t{Op}_
    \cC
    \to \t{Op}_
    \cC.
    \]
\item There is a natural transformation $\t{Seg}\circ W_{\t{Seg}}\Rightarrow \t{Id}$ such that, 
for every Segal pre-operad $\Lambda\in \t{PreSegOp}_\cC$, the natural map $\t{Seg}\circ W_{\t{Seg}}(\Lambda)\to \Lambda$ is an equivalence in 
$\t{PreSegOp}_\cC$ (and, in particular, an equivalence on the level of each tree if $\Lambda$ is a Segal operad). \qed
    \end{enumerate}
\end{theorem}

\subsection{Segal operads in topological moduli problems} 
We will be interested in the category of Segal operads in topological moduli problems.
More precisely, we can define Segal operads in topological moduli problems fitting into a diagram
\[
\Lambda \MM\t{NodHD}^\t{tree}_{\t{protected}}\to \Lambda \widetilde{\t{NodFr}}^\t{tree}_\partial\from \Lambda \t{DM}^\t{tree}.
\]

The definition of the Segal operad $\Lambda \t{DM}^\t{tree}$ is the following. We need to define a topological moduli problem $\Lambda \t{DM}^\t{tree}([\tau])$ for every tree $\tau.$ For a test topological space $S$, we define $\Lambda \t{DM}^{\t{tree}}([\tau])(S)$ to be the collection of families $X\to S$, where $X$ is a family over $S$ of \emph{possibly disconnected} stable
nodal curves, mapping to the discrete set $\t{Vert}(\tau),$ with marked points indexed by half-edges of $\tau$ (including two half-edges for each full edge) and such that the component $X_v$ over a vertex $v\in \t{Vert}_\tau$ has incoming marked points indexed by incoming edges $\t{In}(v),$ and outgoing marked point indexed by the outgoing edge $\t{Out}(v).$ Maps between two objects of $\Lambda \t{DM}^\t{tree}([\tau])(S)$ are isomorphisms over $S$ that preserve all combinatorial data. 
In order to define the Segal structure we need to write down functors 
$$
\Lambda \t{DM}^\t{tree}(i_{\tau,\tau'})(S):\Lambda \t{DM}^\t{tree}([\tau])(S)\to \Lambda \t{DM}^\t{tree}([\tau'])(S)
$$ 
corresponding to restriction to a full subtree $\tau'\subset \tau$, and 
$$
\Lambda \t{DM}^\t{tree}(s_{\tau, \tau'})(S): \Lambda \t{DM}^\t{tree}([\tau])(S)\to \Lambda \t{DM}^\t{tree}([\tau''])(S)
$$ 
corresponding to contraction of a full internal subtree $\tau'\subset \tau.$ The former is simply given by removing all components of $X$ corresponding to vertices not in $\tau',$ and the latter by gluing all curves in the component $\tau'$ along their common edges to combine them into a nodal curve, then stabilizing\footnote{A note on set-theoretic issues. Note that both the procedures of gluing and stabilizing replace the set indexing the old complex curve $X$ by a quotient set, whose points are subsets of $X$, something that is possible in the $\kappa$-small world for $\kappa$ an inaccessible cardinal. In order to avoid sets with repeated indexes, we require that the set of points of $X$ must not contain an element which is a subset of $X$ of cardinality $\ge 2$.}.

Proceeding in a similar fashion, we also get Segal operads
\[\Lambda \MM\widetilde{\t{Fr}}_\partial, \Lambda \widetilde{\t{NodFr}}^\t{tree}_\partial, \Lambda \MM\t{NodHD}^\t{tree}_{\t{protected}}.\]
Note that, since $\widetilde{\t{Fr}}_\partial$ and $\t{NodHD}^\t{tree}_{\t{protected}}$ are topological spaces, we have equivalences $\Lambda\MM\widetilde{\t{Fr}}_\partial\simeq \t{Seg}(\widetilde{\t{Fr}}_\partial)$ and similarly for $\t{NodHD}^\t{tree}_{\t{protected}}$.

We finally have the tools to give concrete statements comparing the homotopy pushout $\t{NodHD}^\t{tree}_{\t{protected}}$ and $\t{DM}^\t{tree}$.

\begin{theorem} \label{thm:NodHDDM-Appendix}
There is a sequence of Segal operads of topological moduli problems, with maps that are levelwise weak equivalences (on the level of simplicial chains):
    \[\Lambda\MM\widetilde{\t{NodHD}}^\t{tree}_{\t{protected}}\to \Lambda\widetilde{\t{NodFr}}^\t{tree}_\partial\from \Lambda\t{DM}^{\t{tree}}.\]
    The first map is induced by applying component-wise the map $\boldsymbol{\pi}$, and the second map is induced by applying component-wise the map $\boldsymbol{\t{funnel}}$.
 \end{theorem}
 
\begin{proof} This is the combined conclusion of Lemma~\ref{lem:from-DM-to-NodFr-Segal} and Proposition~\ref{prop:from-protected-to-unprotected-Segal}. 
\end{proof} 
 
Combining the previous theorem with Theorem~\ref{thm:CM} we obtain the following 

\begin{corollary}
  There is a sequence of (weak) homotopy equivalences in the Berger-Moerdijk category of strict topological operads:
  \[
  \xymatrix
 @C=20pt
 {
  \t{NodHD}^\t{tree}_\t{protected} & W_{\t{Seg}}|C_\Delta\big(\t{NodHd}^\t{tree}_\t{protected}\big)| \ar[l]\ar[d] & \\
  & W_{\t{Seg}}|C_\Delta\big(\t{NodHD}^\t{tree}_\t{protected}\big)| &  W_{\t{Seg}}|C_\Delta (\t{DM}^\t{tree})| . \ar[l]
}
 \]
 \qed
  \end{corollary}
Here the operad $W_{\t{Seg}}|C_\Delta (\t{DM}^\t{tree})|$ is the ``topological classifying operad'' associated to the Deligne-Mumford stack, and hence represents a topological ``resolution'' of the operad in topological moduli problems $\t{DM}^\t{tree}$.

\bibliographystyle{plain}
\bibliography{operadbibliography}{}

\end{document}